\documentclass[11 pt,leqno]{amsart}

\usepackage[utf8]{inputenc} 
\usepackage{amsmath,amsthm,amssymb,amscd,graphicx}
\usepackage{color}
\usepackage[normalem]{ ulem }
\usepackage{soul}
\usepackage{hyperref}
\usepackage{enumerate}
\usepackage{mathrsfs}
\usepackage{multirow}

\newtheorem{thm}{Theorem}[section]
 \newtheorem{cor}[thm]{Corollary}

 \newtheorem{lem}[thm]{Lemma}
 
 \newtheorem{prop}[thm]{Proposition}
 
 \theoremstyle{definition}
 
 \theoremstyle{remark}
 \newtheorem{rem}[thm]{Remark}
 
 \numberwithin{equation}{section}
\def\be#1 {\begin{equation} \label{#1}}
\newcommand{\ee}{\end{equation}}
\renewcommand{\phi}{\varphi}

\def\s{\sigma}
\def\C{\mathbb C}
\def\R{\mathbb R}

\def\N{\mathbb N}
\def\E{\mathcal E}
\def\H{\mathcal H}
\def\W{\mathcal W}

\def\QQ{\mathcal Q}
\def\q{\mathfrak{q}}
\def\e{e}

\def\eps{\epsilon}
\def\dis{\displaystyle}    
 \newcommand{\om}{  \omega   }
  
    \renewcommand{\Im}{   {\mathfrak{Im}} }
\newcommand{\ov}{  \overline  }
\newcommand{\p}{  {\bf p}  }

\newcommand\<{\langle}
\renewcommand\>{\rangle}

\definecolor{gr}{rgb}   {0.,   0.69,   0.23 }
\definecolor{bl}{rgb}   {0.,   0.5,   1. }
\definecolor{mg}{rgb}   {0.85,  0.,    0.85}
\definecolor{yl}{rgb}   {0.8,  0.7,   0.}
\definecolor{or}{rgb}  {0.7,0.2,0.2}

\begin{document}

\thanks{N. Burq is  partially supported by the grant   "ISDEEC'' ANR-16-CE40-0013 and Institut Universitaire de France}
 
\thanks{L. Thomann is partially supported by the grants  "BEKAM''  ANR-15-CE40-0001 and  "ISDEEC'' ANR-16-CE40-0013}

\author{Nicolas Burq}
\address{Laboratoire de Math\'ematiques d'Orsay,
CNRS, Universit\'e Paris--Saclay, B\^atiment 307, F-91405 Orsay Cedex, and Institut Universitaire de France}
\email{nicolas.burq@universite-paris-saclay.fr}
%%%%%
%%%%%%%%
\author{ Laurent Thomann }
\address{Institut  \'Elie Cartan, Universit\'e de Lorraine, B.P. 70239,
F-54506 Vand\oe uvre-l\`es-Nancy Cedex}
\email{laurent.thomann@univ-lorraine.fr}

\title[A.S. scattering for the one dimensional NLS]{Almost sure scattering for the one dimensional nonlinear Schr\"odinger equation}



\begin{abstract}
We consider the one-dimensional nonlinear Schr\"odinger equation with a nonlinearity of degree $p>1$. We exhibit measures on the space of initial data for which we describe the non trivial evolution by the linear Schr\"odinger flow and we show that their nonlinear evolution is absolutely continuous with respect to this linear evolution. We deduce from this precise description the global well-posedness of the equation for $p>1$ and scattering for $p>3$. To the best of our knowledge, it is the first occurence where the description of quasi-invariant measures allows to get quantitative asymptotics (here scattering properties) for the nonlinear evolution.
\end{abstract}

\maketitle

  \tableofcontents
   \section{Introduction and results}

\subsection{General introduction}
Let $p>1$. In this paper we study long time dynamics for the one-dimensional nonlinear Schr\"odinger equation
 \begin{equation*} \tag{$NLS_p$}\label{C1} 
  \left\{
      \begin{aligned}
         &i\partial_sU+\partial_{y}^2U=|U|^{p-1}U,\quad (s,y) \in  \R\times \R,
       \\  &  U\mid_{s=s_0}  =U_0,
      \end{aligned}
    \right.
\end{equation*}
where $U_0$ is a  random initial condition, with low Sobolev regularity. The distribution of $U_0$ will be given by a Gaussian measure and we will study its evolution under the  nonlinear flow  of \eqref{C1}, denoted by $\Psi(s,s_0)$, and compare it with the evolution under the linear flow $\Psi_{lin}(s, s_0) = e^{i(s-s_0)\partial_y^2}$.

When working on compact manifolds $M$ instead of $\mathbb{R}_x$, there exists natural Gaussian measures~$\mu$ supported in some Sobolev spaces $H^{\sigma}(M)$ which are invariant by the flow $\Sigma_{lin}(s)$ of the linear equation (Wiener measures, see Section~\ref{sect.functional} for more details). At some particular scales of regularity these measures can be suitably modified (Gibbs measures) to ensure that they are invariant by the {\em nonlinear} flow~\cite{Bourgain2d}, or only quasi-invariant (renormalized energies)~\cite{Tz2015, OhTz1, OhTz2, OhSoTz}. 
In our context, and more generally on $\mathbb{R}^d_x$, the situation is different, since dispersion prohibits the existence of measures invariant by the  flow of the linear or nonlinear Schr\"odinger equation (see Proposition~\ref{propinv}, Proposition~\ref{propinv2} and Proposition~\ref{propinv3}). The purpose of the present work is twofold. First we define mesures on the space of initial data for which we can describe precisely the {\em non trivial} evolution by the {\em linear} flow (notice that even this first step is non trivial). Second, we prove that the  nonlinear evolution of these measures is absolutely continuous with respect to their (explicit) linear evolutions (we actually prove a precised {\em quantitative} version of the absolute continuity, characterizing the {\em integrability of the Radon-Nikodym derivative},  see Theorem~\ref{thm0}), and finally we get benefit from this precise description to prove {\em almost sure scattering} of our solutions of~\eqref{C1} for $p>3$. Let us emphasize that these precise {\em quantitative} estimates for the quasi-invariance are the key point to the proof of almost sure scattering. We refer to Section~\ref{section.measures} for complete statements.
To the best of our knowledge, the results in the present article are the first ones giving insight, {\em in a non compact setting} on the time evolution of the statistical distribution of solutions of a nonlinear PDE (see also Ammari-Nier~\cite{Ammari-Nier1, Ammari-Nier2,Ammari-Nier3} in a completely different context). They also are the first ones providing scattering for NLS {\em for large initial data} without assuming {decay at infinity}: our solutions are essentially in $L^2$, but they actually miss the~$L^2$ space by a logarithmic divergence both in space and in frequency, namely they are in the Besov space $\mathcal{B}^0_{2, \infty}(\R)$ built on the harmonic oscillator (see Appendix~\ref{sec.besov}). Finally, we are not aware of any other results using the existence and description of invariant or quasi-invariant measures to describe the large time behaviour of solutions to PDE's going beyond the globalisation argument from Bourgain~\cite{Bourgain1d, Bourgain2d} and the elementary Poincaré recurrence theorem.

\subsection{Measures with non trivial linear evolution}
We shall define families of measures supported essentially on $L^2( \mathbb{R})$ (modulo a logarithmic divergence, see Appendix~\ref{sec.besov}) for which one can get a good description of the (non trivial) linear evolution, and prove that the nonlinear evolution of the measure is indeed quasi-invariant with respect to this linear evolution. 

We denote by 
$$H=-\partial^2_x+x^{2}\,,$$
     the harmonic oscillator in one space dimension, and by $(e_{n})_{n \geq 0}$  the Hermite functions its {$L^2$-normalised} eigenfunctions, $He_n=\lambda_n^2 e_n =  (2n+1) e_n$. Recall that the family $(e_{n})_{n \geq 0}$ forms a Hilbert basis of $L^2(\R)$.   

For $\sigma\geq 0$, denote by 
 \begin{equation}\label{Hs} 
\mathcal{H}^\sigma ( \mathbb{R}) = \big\{ u\in L^2(\R): (1-\Delta)^{\sigma/2} u \in L^2(\R),\;\; |x|^\sigma u \in L^2(\R)\big\},
\end{equation}
and for $\sigma \geq 0$, $\mathcal{H}^{-\sigma} ( \mathbb{R})$ is its dual space. We shall denote by 
$X^0(\R)=  \bigcap_{\eps >0} \H^{-\eps}(\R)$. Notice that $\dis L^2(\R) \subset X^0(\R)$.

We start with a typical Gaussian measure. Consider a  probability space
$(\Omega, {\mathcal F}, { \p})$ and let   $(g_{n})_{n \geq 0}$     be a sequence of independent complex standard Gaussian variables.   Let $\eps>0$, we define the probability Gaussian measure $\mu_0$ on $\H^{-\eps}(\R)$ as the law of the random variable $\gamma$
 \begin{equation} \label{defmu}
 \begin{array}{rcl}
\Omega&\longrightarrow&\H^{-\eps}(\R)\\[3pt]
\dis  \omega&\longmapsto &\dis\gamma^{\om}= \sum_{n=0}^{+\infty} \frac1{\lambda_n}g_{n}(\om)e_n, 
 \end{array}
 \qquad \mu_0= {\p} \circ \gamma^{-1}.
 \end{equation}

The measure $\mu_0$ satisfies $\mu_0(L^2(\R))=0$ and    $\dis \mu_0 \big( X^0(\R)\big)=1$.
The following theorem gives a flavour of  our results in this paper.

  \begin{thm}\label{thmglobal}
 Let $p>1$  and assume that $s_0=0$ in~\eqref{C1}. For $\mu_0-$almost every initial data $U_0 \in X^0(\R)$, there exists a unique, global in time, solution $U=\Psi(s,0)U_0$ to~\eqref{C1}.  

Furthermore, the evolution of the measure $\mu_0$ by this nonlinear flow, $\Psi(s,0)  _{{\#}} \mu_0 $ is absolutely continuous with respect to the evolution by the linear flow, $\Psi_{lin}(s, 0 )  _{{\#}} \mu_0$.  

Finally, the solution takes the form
$$\Psi(s,0)U_0 = e^{is\partial_y^2} U_0 + V, $$ where $V$ satisfies for some $C, K>0$ and all $s\in \R$
$$\| V(s)\|_{\H^\sigma(\R)} \leq C \langle s \rangle ^K,$$
and where $\sigma<\sigma_0$ can be chosen arbitrarily close to    
$
 \sigma_0= \begin{cases}
\frac{p-1}2  &\text{ if } 1 < p\leq 2\\[5pt]
\frac12  &\text{ if } p\geq 2.
\end{cases}
$
 \end{thm}
 
  In the sequel, we will see that for all $s\in \R$, $\Psi_{lin}(s, 0 )  _{{\#}} \mu_0$  is  given by an explicit  time-dependent Gaussian measure. Moreover, we will see that these measures are supported   in the Besov space $\mathcal B^0_{2,\infty}(\R)$ based on the harmonic oscillator. We refer to Section~\ref{evol-meas} and Appendix \ref{Appendix-B} for more details.
 
 The values of $\sigma_0$ in Theorem~\ref{thmglobal}  will play a key role in the proof of the scattering result (Theorem~\ref{thm1}) for which we need the embedding $\H^{\sigma} \subset L^{p+1}$   to control the nonlinearity. Let us however mention  that the value of  $\sigma_0$ obtained in Theorem~\ref{thmglobal} in the case $1<p \leq 2$ is not optimal, and a slight modification in the proof may improve it.

\subsection{Scattering results}
Using a quantified version of the absolute continuity, we are able to go beyond the usual easy consequences (global existence, Poincaré recurrence theorem, logarithmic bounds on the time evolution of the {\em complexity} of solutions\dots). Namely, we shall use the precise knowledge of the nonlinear evolution of our measures to prove almost sure scattering properties of solutions of \eqref{C1} for $p>3$ (notice that quasi-invariance without estimates does not even imply Poincaré recurrence).
\begin{thm}\label{thm1} Assume that $p >1$. Then 
the solutions to~\eqref{C1}  we have  constructed in Theorem~\ref{thmglobal} disperse: {for $\mu_0-$almost every initial data $U_0 \in X^0(\R)$, there exists  a  constant $C>0$} such that  for all  $s\in \R$
\begin{equation*}
\| \Psi(s,0)U_0\|_{L^{p+1}(\R)} \leq \begin{cases}
C\frac{(1+  \log\<s\> )^{1/(p+1)}} {\langle s\rangle ^{ \frac 1 2 -\frac 1 {p+1}}} &\text{ if } 1<p <5 \\[10pt]
\frac{C} {\langle s\rangle ^{ \frac 1 2 -\frac 1 {p+1}}} &\text{ if }  p\geq 5.
\end{cases}
\end{equation*}
{Assume now that $p>3$. Then  there exist $\sigma, C, \eta>0$ and    ${W_\pm\in \mathcal{H}^\sigma(\R)}$ }such that for all~$s\in \R$
   \begin{equation} \label{ab1}
  \| \Psi(s,0) U_0- e^{is\partial_y^2} (U_0 +W_{\pm})\|_{\mathcal{H}^\sigma(\R)} \leq C \langle s \rangle^{- \eta},
   \end{equation}
and 
   \begin{equation} \label{ab2} 
  \|  e^{-is \partial^2_y}  \Psi(s,0) U_0- (   U_0+W_{\pm})   \|_{\H^{\sigma}(\R)}   \leq C \langle s \rangle^{- \eta}.
  \end{equation}
  {In the case  $p\geq 5$,  we can precise the result: for all $\sigma <\frac12 $ there exist $C,  \eta >0$ such that for all~$s\in \R$
   \begin{equation}\label{ab3} 
  \| \Psi(s,0) U_0- e^{is\partial_y^2} (U_0 +W_{\pm})\|_{H^\sigma(\R)} \leq C \langle s \rangle^{- \eta},
   \end{equation}}
\end{thm}

  Notice that since $e^{\pm i s \partial_y^2}$ does not act on $\mathcal{H}^{\sigma}(\R)$, the  properties \eqref{ab1} and \eqref{ab2}  are different.  
 Actually we prove a more general result. We construct a four-parameter family $(\mu_\q)$ of Gaussian measures on $X^0(\mathbb{R})$, for which the previous statement holds true (see Theorem~\ref{thm0} and Theorem~\ref{thm1.1.bis}).

In our previous work~\cite{BTT} we  performed part of the program above, namely we proved the scattering result {\em in the particular case $p\geq 5$}.  In this case monotonicity properties allow to greatly simplify the proof and  a fine description of the nonlinear evolution of the measures was unnecessary to get scattering properties. {We emphasize that the convergence in \eqref{ab3} holds in the usual Sobolev space~$H^\sigma$ but not in the weighted space~$\H^\sigma$ (the statement \eqref{ab1} is a corrected version of \cite[Theorem 1.2]{BTT}). This is due to a lack of continuity of the lens transform in the~$\H^\sigma$  spaces (see Lemma~\ref{lemA}).} 

In the case $p\leq 3$, Barab~\cite{Barab} showed that a non trivial solution to \eqref{C1} never scatters, thus even with a stochastic approach one can not hope for  scattering in this case.  Therefore the condition~${p>3}$ in Theorem~\ref{thm1} is   optimal. In \cite{TsYa}, Tsutsumi and Yajima proved a scattering result in $L^2(\R^d)$,  $d\geq 2$ but assuming additional $\mathcal{H}^1-$regularity  on the initial conditions.

We refer to \cite[Theorem 1.4]{PRT2} for an almost sure scattering result for the two-dimensional NLS. In this latter case, one could use a probabilistic smoothing property on the Hermite functions which only holds in dimension $d \geq2$.  For other almost sure scattering results for NLS, we refer to \cite{DLM, KMV}. {In particular, the results in this paper were recently generalised by   Latocca~\cite{Latocca}  in the multi-dimensional case, in the radial setting.} See also Nakanishi \cite{Kenji} for deterministic scattering results for \eqref{C1} in
Sobolev spaces $H^{\s}$ for $\s\geq 1$.

\subsection{Plan of the paper}
The plan of the paper is the following. In Section~\ref{section.measures} we define the measures, state our main results precisely and prove some properties about the measures (description of the linear evolution, absolute continuity\dots). In Section~\ref{sect.functional} we prove  elementary results on the non existence  of invariant measures for the Schr\"odinger equation on $\R$,  recall some tools of functional analysis and give a characterization  of weak $L^p$ regularity of Radon-Nikodym derivatives. Section~\ref{sect.evol} is devoted to the estimate of the time evolution of the measures under the Galerkin approximations of the nonlinear flow. In Section~\ref{Sect5}   we prove the main nonlinear estimates that we need in the sequel. There, a difficulty is induced by the low regularity of our nonlinearity $F(u) =|u|^{p-1} u$ which is not $C^2$ for $p<2$. In Section~\ref{sec.fonc} we introduce the spaces in which we are able to prove the local (and later global) well-posedness. Here, another difficulty is that due to a lack of smoothness of our initial data, the spaces for the initial data ($Y^{\rho,\eps}$), and the solutions ($X^\rho$) are different. For the initial data, we exploit some {\em probabilistic smoothness} while for the solution we gain some {\em deterministic smoothness}. In Section~\ref{sec.7} we show that almost surely our initial data are indeed in the spaces $Y^{\rho,\eps}$ and we prove large deviation estimates. In Section~\ref{sec.8}  we develop a suitable Cauchy theory for the nonlinear problem. In Section~\ref{sec.9} we prove the almost sure global well-posedness. Finally in Section~\ref{sec.10} we prove the quasi-invariance properties of our measures, while in Section~\ref{sec.11} we use this quasi-invariance to prove decay for $p>1$ and scattering for $p>3$. We gathered in an Appendix some technical results.

  %%%%%%%%%%%%%%%%%%%%%%%%%%%%%%%%%%%%%%%%%%%%%%%%%%%%%%%%%%%%

 \subsection{Notations} 
  In this paper $c,C>0$ denote constants the value of which may change
from line to line. These constants will always be universal, or uniformly bounded with respect to the other parameters.  We denote by $H=-\partial^2_x+x^{2}$ the harmonic oscillator on $\R$, and   for $\s\in \R$ we define the   Sobolev space~$\H^{\s}(\R)$  by the norm  $\|u\|_{\H^{\s}(\R)}=\|H^{\s/2}u\|_{L^{2}(\R)}$.  More generally, we define the spaces~$\W^{\s,p}(\R)$ by the norm $\|u\|_{\W^{\s,p}(\R)}=\|H^{\s/2}u\|_{L^{p}(\R)}$ (see also Section~\ref{sect32} for more details and notations). The Fourier transform is defined by $\mathcal{F}f(\xi)=\int_\R e^{-ix\xi}f(x)dx$, for $f \in \mathscr{S}(\R)$. The Fourier multiplier $D_x^\alpha$ is defined as a  tempered distribution $\mathcal{F}(D_x^\alpha  f)(\xi)= |\xi|^{\alpha} \mathcal{F}(f)(\xi)$ for $f \in \mathscr{S}'(\R)$.

%%%%%%%%%%%%%%%%%%%%%%%%%%%%%%%%%%%%%%%%%%%%%%%%%%%%%%%%%%%%%%%
%%%%%%%%%%%%%%%%%%%%%%%%%%%%%%%%%%%%%%%%%%%%%%%%%%%%%%%%%%%%%%%

\section{The measures, linear analysis}\label{section.measures}
\subsection{Definition of the Gaussian measure \texorpdfstring{$\mu_0$}{mu-0}} Consider    a system of complex, independent, centered, $L^2-$normalized  Gaussians $(g_n)_{0 \leq n \leq N}$ on a probability space $({\Omega, \mathcal{T},{\bf p}})$.    Let us first recall that the density distribution of $\dis \frac{1}{\lambda_n}g_n\sim \mathcal{N}_{\C}(0, \lambda_n^{-2})$ is given by 
 \begin{equation*} 
  \frac{\lambda_n^2} {\pi} e^{-  {\lambda_n^2}  { |u_n|}^2}du_n d\ov{u_n}=  \frac{\lambda_n^2} {\pi} e^{-  {\lambda_n^2}  ( a_n^2 + b_n^2)} da_n db_n, \qquad u_n = a_n + i b_n\;,
  \end{equation*}
and where  $du_nd\overline{u} _n$ is the Lebesgue measure on $\mathbb{C}$.   Denote by     $\mu_N$  the distribution   random variable
\begin{equation*} 
\omega\longmapsto \sum_{n=0}^{N}{\frac{1}{\lambda_{n}}}g_{n}(\omega)e_n(x)=:
\gamma_{N}(\omega,x).
\end{equation*}
  Let $\eps>0$, then $(\gamma_N)_{N\geq 0}$ is a Cauchy sequence in $L^2(\Omega;{\mathcal
H}^{-\eps}(\R))$ 
which defines
\begin{equation*} 
\gamma(\om,x):=\sum_{n=0}^{+\infty}{\frac{1}{\lambda_{n}}}g_{n}(\omega)e_n(x),
\end{equation*}
as the limit of $\gamma_N$. Now the map
$
\omega\mapsto\gamma(\omega,x)
$
defines a Gaussian measure on ${\mathcal H}^{-\eps}(\R)$ which we shall denote by  $\mu_0$.
Notice also that the measure $\mu_0$ can be decomposed into 
\begin{equation} \label{def-mun}
\mu_0= \mu_N \otimes {\mu}^N
\end{equation}
 where  $\mu^N$ is the  distribution of the random variable     $ \dis \sum_{n=N+1}^{+\infty}\frac{1}{\lambda_n}g_{n}(\omega)e_n(x)$        on $E_N^\perp$. In other words
$$ d\mu_N= \bigotimes_{0\leq n \leq N} \mathcal{N}_{\C} (0, \lambda_n^{-2}), \qquad d\mu^N= \bigotimes_{n > N} \mathcal{N}_{\C} (0, \lambda_n^{-2})\;,$$
and the measure $\mu_0$   can be represented (rather informally) by
\begin{equation*}
  d\mu_0 = \bigotimes_{n \geq 0} \mathcal{N}_{\C} (0, \lambda_n^{-2})= \bigotimes _{n\geq 0} \frac {\lambda_n^2}{\pi}e^{- \lambda_n^2 |u_n|^2}  du_nd\overline{u}_n\quad
\Rightarrow \quad\mu_0(A)=\int_{A}\prod_{n\geq 0} \frac {\lambda_n^2} {\pi}e^{- \|\sqrt{H}\,u\|_{L^2(\R)}^2}du_nd\overline{u}_n,
\end{equation*}
where we decompose $u= \sum_n u_n e_n$ and hence identify $\mathcal{H}^{-s}$ which supports the measure $\mu_0$ with~$\mathbb{C}^\mathbb{N}$. Finally we define
\begin{equation*} 
X^{0}(\R)=\bigcap_{\eps>0}\H^{-\eps}(\R),
\end{equation*}
so that $\mu_0$ is a probability measure on $X^{0}(\R)$.

 \subsection{Evolution of the Gaussian measures \texorpdfstring{$\mu_\q$}{mu-q}}\label{evol-meas} In this section we define a four-parameter family of Gaussian measures $(\mu_{\q})$, which relies on the symmetries of the linear Schr\"odinger group. 

 The space dilations
$u \mapsto \Lambda_\beta u  = \beta^{1/2} u(\beta \,\cdot )$, time translations 
{$ u \mapsto \Psi_{lin}(s,s_0) u:= e^{i (s-s_0) \partial_y ^2} u$},
space translations 
$ u \mapsto \tau_{\theta} u = u(\,\cdot \,- {\theta})$,
and homotheties
$ u \mapsto M_\alpha u = \alpha u$ are invariances of the Schr\"odinger flow and their actions on $L^2(\R)$  define a four-parameter family of measures. Set 
$$ \QQ:= \mathbb{R}\times \mathbb{C}^*\times\mathbb{R}^*_+ \times \mathbb{R}, \qquad \q=(s,  \alpha,  \beta, \theta) \in \QQ,$$
 and define the family of Gaussian measures
\begin{equation}\label{def-muq}
  (  \Psi_{lin}(s,0)  \circ  M_\alpha  \circ \Lambda_\beta\circ \tau_{\theta}    )  _{\#} \mu_0 = \mu_{(s,  \alpha,  \beta, \theta)}=\mu_{\q}\,,
 \end{equation}
 given by 
 $$  \mu_{\q}(A)=\mu_{(s,  \alpha, \beta, \theta)}(A) := \mu_0 \big((  \Psi_{lin}(s,0) \circ M_\alpha   \circ  \Lambda_\beta \circ \tau_{\theta}   )^{-1}A\big).$$
 In the particular case $\q=(0,1,1,0)$ we have $\mu_\q=\mu_0$. Notice  that since in the definition~\eqref{defmu}, the law of complex random variables $g_n$ is invariant by the multiplication by any complex number of modulus $1$, we have $\mu_{(s,\alpha,\beta,\theta)}=\mu_{(s,|\alpha|,\beta,\theta)}$. Notice also that  it is a direct consequence of the definition that 
  $$ \Psi_{lin}(s_1+s_0,s_0)_{\#}\mu_{(s_0, \alpha, \beta, \theta)}= (e^{is_1 \partial_y^2}) _{\#}\mu_{(s_0, \alpha, \beta, \theta)}= \mu_{(s_1+ s_0, \alpha, \beta, \theta)}.$$ 
  
   For all $\q\in \QQ$ and all $\eps >0$, the measure $\mu_\q$ is supported on $\mathcal{H}^{-\eps}(\R)$ while ${\mu_{\q}( L^2(\R)) =0}$. More precisely, we can prove that $\mu_{\q}$ is supported in the Besov space $\mathcal B^0_{2,\infty}(\R)$ based on the harmonic oscillator (see Proposition~\ref{1.1}). 
   
   The following result is proved in Section~\ref{appen.b}.
  \begin{prop}\label{prop2.2}
Let $j=1,2$ and $\q_j =(s_j,\alpha_j,\beta_j, {\theta}_j)\in \QQ$, then  the measures $\mu_{\q_1}$ and $\mu_{\q_2}$ are absolutely continuous with respect to each other if and only if 
  $$ \big(s_1, |\alpha_1|,\beta_1, \theta_1\big) = \big(s_2, |\alpha_2|, \beta_2, \theta_2\big).$$
  \end{prop}
  
When the measures are not absolutely continuous with respect to each other, they are mutually singular (supported on disjoint sets of $\mathcal{H}^{- \eps}(\R), \eps >0$). Actually, thanks to the Hajek-Feldman theorem \cite[Theorem~2.7.2]{Boga}, two Gaussian measures on the same space are either equivalent or mutually singular.

  We can now state precisely our main results. We assume that $s_0=0$ in \eqref{C1}.
   
  \begin{thm}\label{thm2}
Let $p>1$  and $\q_0 = (s_0, \alpha_0, \beta_0,\theta_0) \in \QQ$. Let $\q_s= (s+s_0, \alpha_0, \beta_0,\theta_0)$. There exists a set $\mathbf{S}\subset X^0(\R)$ of full $\mu_{\q_0}-$measure  such that  for all $U_0 \in \mathbf{S}$, there exits a unique, global in time,  solution to~\eqref{C1}   in the class
$$e^{is \partial^2_y} U_0 + C^0(\mathbb{R}; \H^\sigma(\R)), 
$$ where   $\sigma<\sigma_0$ can be chosen arbitrarily close to   
$
\sigma_0= \begin{cases}
\frac {p-1} 2  &\text{ if } 1 < p\leq 2\\[5pt]
\;\frac12  &\text{ if } p\geq 2 .
\end{cases}
$

 Denote by $\Psi(s,0)$ the flow such defined on $\mathbf{S}$ and denote by $\mathbf{S}_{s} = \Psi(s,0) { \mathbf{S}}$. Then the set $\mathbf{S}_s$ is of full $\mu_{\q_s}-$measure and there exists $K_p>0$ such that $\mu_{\q_0}-$almost surely there exists $C>0$ such that we have the estimates 
\begin{equation*}
 \Psi(s,0)U_0=   e^{is \partial^2_y} U_0  + V, \qquad \| V(s) \|_{\H^\sigma(\R)} \leq C \langle s\rangle ^{M_{p,\sigma}}.
\end{equation*} 
\end{thm}

We emphasize that the exponent $M_{p,\sigma}>0$ which appears in the previous estimate is deterministic, and depends only on $p>1$ and $\sigma>0$. Only the constant $C>0$ is probabilistic.
\begin{rem} In most of the previous works in which almost sure existence results are obtained for nonlinear dispersive equations, with arguments  relying on invariant measures, it is possible to show that  the corresponding set $\mathbf{S}$ of  initial data is invariant by the flow, namely  $\mathbf{S}_{s} =  { \mathbf{S}}$ for all $s\in \R$. In our situation we were not able to prove the invariance of  ${ \mathbf{S}}$ by the nonlinear flow. In some sense, since $\mathbf{S}$ has full $\mu_{\q_0}$ measure, while $\mathbf{S}_s$ has full $\mu_{\q_s}-$measure  and $\mu_{\q_0}$ and $\mu_{\q_s}$ are singular to each other  for any $s\neq 0$, {\it i.e.} supported on disjoint sets, this non invariance is natural.
\end{rem}

We can now state precisely our scattering result.

\begin{thm}\label{thm0}
Let $p>1$ and $\q_0\in \QQ$.  Denote by $\Psi(s,0)$ the flow    on $\mathbf{S}$ defined in Theorem \ref{thm2}. We have a fine description of the time evolution of the measures $\mu_{\q_0}$. 
\begin{itemize}
\item  For all $s\in \R$,    the measures $\Psi(s,0)  _{\#} \mu_{\q_0}$ and $\Psi_{lin}(s,0)_{\#}\mu_{\q_0}$ are equivalent (they have the same zero measure sets);
\item For all $s' \neq s$, the measures $\Psi(s,0)  _{\#} \mu_{\q_0}$ and  $\Psi(s',0)  _{\#} \mu_{\q_0}$ are mutually singular.

More precisely, in the particular case $\q_0 = (0,1,1,0)$, denoting by  
$$   \rho_s  = e^{- \frac{(1+ 4s^2)}{p+1} {\| u\|_{L^{p+1}}^{p+1}} }\mu_{\q_s},\qquad \q_s=(s,1,1,0), $$ 
we have 
for all $0\leq |s'| \leq |s|<+\infty$ and all $A \subset { \mathbf{S}}$, 
\begin{align}\label{equival-ter}
  \rho_{s}\big(\Psi(s,0)A\big) & \leq \begin{cases}  \rho_{s'}\big(\Psi(s',0)A\big)&\text{ if } 1\leq p \leq 5 \\[10pt]
  \bigl(\rho_{s'}\big(\Psi(s',0)A\big)\bigr) ^{\smash{\bigl( \frac{1+ 4 (s')^2}{1+ 4 s^2}\bigr) ^{\frac{ p-5} 4}}} &\text{ if } p \geq 5\end{cases}\\
\intertext{and} 
\rho_{s'}\bigl(\Psi(s',0) A \bigr)&\leq \begin{cases}\bigl(\rho_{s} \big(\Psi(s,0) A\big)\bigr) ^{\smash{\bigl( \frac{1+ 4 (s')^2}{1+ 4 s^2}\bigr) ^{\frac{ 5-p} 4}}} &\text{ if } 1\leq p \leq 5\\[10pt]
 \rho_{s} \big(\Psi(s,0) A\big) &\text{ if }  p \geq 5. \end{cases}
 \label{abso-cont-ter}
  \end{align}
  \item There exists $ \mu_{\q_0}-$almost surely a constant $C>0$ such that  for all  $s\in \R$
\begin{equation}\label{log-bound}
\| \Psi(s,0)U_0\|_{L^{p+1}(\R)} \leq \begin{cases}
C\frac{(1+  \log\<s\> )^{1/(p+1)}} {\langle s\rangle ^{ \frac 1 2 -\frac 1 {p+1}}} &\text{ if } 1<p <5\\[10pt]
\frac{C} {\langle s\rangle ^{ \frac 1 2 -\frac 1 {p+1}}} &\text{ if } p \geq5.
\end{cases}
\end{equation}
\item Assume moreover that $p>3$. {Then the solutions to~\eqref{C1} constructed above scatter $\mu_{\q_0}-$almost surely when $s \longrightarrow \pm \infty$ : there exist $C, \sigma, \eta >0$, and
${W_\pm\in \mathcal{H}^\sigma(\R)}$ such that  for all~$s\in \R$}
   \begin{equation*}  
  \| {\Psi(s,0)} U_0- e^{is\partial_y^2} (U_0 +W_{\pm})\|_{\mathcal{H}^\sigma(\R)} \leq C \langle s \rangle^{- \eta},
   \end{equation*}
and 
   \begin{equation*}  
  \|  e^{-is \partial^2_y}   { \Psi(s,0)} U_0- (   U_0+W_{\pm})   \|_{\H^{\sigma}(\R)}   \leq C \langle s \rangle^{- \eta}.
  \end{equation*}
  (Notice that since $e^{\pm i s \partial_y^2}$ does not act on $\mathcal{H}^{\sigma}(\R)$, the two estimates above are different.)
  \item  {In the case  $p\geq 5$,  we can precise the result: for all $\sigma <\frac12 $ there exist $C,  \eta >0$ such that for all~$s\in \R$
   \begin{equation*}  
  \| \Psi(s,0) U_0- e^{is\partial_y^2} (U_0 +W_{\pm})\|_{H^\sigma(\R)} \leq C \langle s \rangle^{- \eta},
   \end{equation*}}
 \end{itemize}
 \end{thm}
    
 Assume that $1< p<5$, then the equation \eqref{C1} is globally well-posed in $L^2(\R)$, see \cite[Section~4.6]{Cazenave}. In the case $p=5$, the equation is $L^2-$critical, and global well-posedness and scattering in $L^2(\R)$ has been proved by Dodson~\cite{Dodson}. 
Let $p>1$, then  \eqref{C1} is globally well-posed in~$H^1(\R)$ by \cite{Ginibre-Velo}. In \cite{Visciglia}, Visciglia shows moreover that for any $U_0 \in H^1(\R)$, and any $2<r\leq +\infty$, $\|  \Psi(s,0) U_0\|_{L^{r}(\R)} \longrightarrow 0$, when $s\longrightarrow +\infty$. Therefore~\eqref{log-bound} gives a similar rate of decay for rough but random initial conditions.

 For large times $|s|\gg1$, the bound \eqref{log-bound} is up to the logarithmic term, the decay of  linear solutions. Namely, recall that for all $\phi \in \mathscr{S}(\R)$ we have the classical dispersion bound
 \begin{equation*}
 \| \e^{is\partial_y^2}\phi\|_{L^{p+1}(\R)} \leq \frac{C} { |s|^{ \frac 1 2 -\frac 1 {p+1}}} \| \phi\|_{L^{(p+1)'}(\R)}, \qquad s \neq 0,
 \end{equation*}
therefore, the power decay in $s$ is optimal.  With the arguments developed in the present paper one could show that the logarithm can be removed for some $p<5$ close to 5, but we do not now what happens for general $1<p<5$.  
 For $s=0$, the bound \eqref{log-bound} is a consequence of the fact that the measures $\mu_\q$ are supported in $\dis \bigcap_{r>2} L^r(\R)$, see Section~\ref{sec.7}. 

{It is worth mentioning that the powers appearing in \eqref{equival-ter} and \eqref{abso-cont-ter} are optimal. Actually, any improvement in these exponents would imply stronger decay in \eqref{log-bound}, which is impossible.}

\begin{cor}\label{Coro-Radon}
Let $\q_0 = (0,1,1,0)$ and denote by 
$$   \rho_s  = e^{- \frac{(1+ 4s^2)}{p+1} {\| u\|_{L^{p+1}}^{p+1}} }\mu_{\q_s},\qquad \q_s=(s,1,1,0). $$  Assume that $0\leq |s'| \leq |s|<+\infty$. The measures $\Psi(s',s)_{\#} \rho_s$ and $\rho_{s'}$ are absolutely continuous with respect to each other and satisfy 
\begin{equation*}
 \begin{aligned} 
 \Psi(s',s)_{\#} \rho_{s} \leq  \rho_{s'},  \qquad \rho_{s'} \leq \bigl(\Psi(s', s)_{\#} \rho_{s}\bigr) ^{{\bigl( \frac{1+ 4 (s')^2}{1+ 4 s^2}\bigr) ^{\frac{  5-p} 4}}}  \qquad  &\text{ if } 1<p \leq 5 \\[5pt]
  \Psi(s',s)_{\#} \rho_{s} \leq  \bigl(\rho_{s'}\bigr) ^{{\bigl( \frac{1+ 4 (s')^2}{1+ 4 s^2}\bigr) ^{\frac{ p-5} 4}}},  \qquad  \rho_{s'} \leq \Psi(s',s)_{\#} \rho_{s}\qquad  &\text{ if } p \geq 5.
 \end{aligned}
  \end{equation*}
  As a consequence, from Proposition~\ref{prop-R-N}, if one denotes by $f_{s',s}$ the Radon-Nikodym derivative of $\Psi(s',s)_{\#} \rho_{s}$ with respect to $\rho_{s'}$,  we have  
  \begin{equation*}
  \begin{cases}
    f_{s',s} \in L^\infty , \;\; f_{s',s}^{-1} \in L^{p(s',s)}_w, \;\; \frac 1{p(s',s)} = 1 -{{\bigl( \frac{1+ 4 (s')^2}{1+ 4 s^2}\bigr) ^{\frac{ 5-p} 4}}} & \text { if } 1<p \leq 5 \\[5pt]
       f_{s',s} \in L^{p(s',s)}_w, \;\;  \frac 1{p(s',s)} =1 -  {{\bigl( \frac{1+ 4 (s')^2}{1+ 4 s^2}\bigr) ^{\frac{ p-5} 4}}}, \;\;   f_{s',s}^{-1} \in L^\infty& \text { if } p \geq 5\;,
    \end{cases}
    \end{equation*}
  where $L^q_w$ is the weak $L^q$ space.
\end{cor}

 Equivalently,   the previous result can be stated for the measures $\Psi(0,s)_{\#} \rho_s$ and $\Psi(0,s')_{\#} \rho_{s'}$, with $0\leq |s'| \leq |s|<+\infty$. \medskip
 
  \subsection{From \texorpdfstring{$(NLS_p)$}{NLS-p}  to \texorpdfstring{$(NLSH_p)$}{NLSH-p}}
As in \cite{Th09,BTT}, we use the lens transform which allows to work with the Schr\"odinger equation with harmonic potential.  More precisely, suppose that $U(s,y)$ is a solution of the problem \eqref{C1}. Then the function  $u(t,x)$ defined  for $|t|<\frac{\pi}{4}$ and $x\in\R$ by
\begin{equation}\label{lens}
 u(t,x)=  \mathscr{L} (U)(t,x):= \frac{1}{\cos^{\frac{1}{2}}(2t)} U\big(\frac{\tan(2t)} 2, \frac{x}{\cos(2t)}\big)e^{-\frac{ix^2{\rm tan}(2t)}{2}}:=\mathscr{L}_t (U\mid_{s= \frac{\tan(2t)} 2})(x)
\end{equation}
where
\begin{equation}\label{lensbis}
 \mathscr{L}_t (G)(x)=\frac{1}{\cos^{\frac{1}{2}}(2t)}G\big(\frac{x}{\cos(2t)}\big)
e^{-\frac{ix^2{\rm tan}(2t)}{2}}\,,
\end{equation}
 solves the problem
 \begin{equation*} \tag{$NLSH_p$}\label{C3}   
  \left\{
      \begin{aligned}
         &i\partial_t u-Hu=\cos^{\frac{p-5}{2}}(2t)|u|^{p-1}u,\quad |t|<\frac{\pi}{4},\, x\in\R,
       \\  &  u(0,\cdot)  =U_0,
      \end{aligned}
    \right.
\end{equation*}
where $H=-\partial^2_x+x^{2}$. Similarly, if $U = e^{is \partial_y^2} U_0$ is a solution of the linear Schr\"odinger equation, then~$u=e^{-itH} U_0=\mathscr{L}(U)$ is the solution of the linear harmonic Schr\"odinger equation with the same initial data.  In other words, if we denote by $\Psi(s,s') $ the map which sends the data at time $t'$ to the solution at time $t$ of~\eqref{C3}, the family  $  (\mathscr{L}_t)_{|t|< \frac \pi 4}$ conjugates the linear and the nonlinear flows: with $t(s) =\frac{\arctan(2s)} 2$, $s(t) = \frac{ \tan (2t)} 2$,
\begin{equation}\label{conjlin}
  \mathscr{L}_{t(s)} \circ e^{i(s-s')\partial_y^2}   =e^{-i (t(s) -t(s'))H}\circ  \mathscr{L} _{t(s')}.
\end{equation}
and 
\begin{equation}\label{conjnl}
\mathscr{L}_{t(s)} \circ  \Psi(s,s')  = \Phi(t(s), t(s')) \circ \mathscr{L} _{t(s')}.
\end{equation}
As a consequence, precise description of the time evolutions of our measures on the harmonic oscillator side (for the functions $u(t,x)$ solutions to~\eqref{C3}) will imply precise descriptions of the evolution on the NLS side (for the functions $U(s,y)$ solutions to~\eqref{C1}).    

 Denote by   $\q_s=(s,1,1,0)$, then for all $s\in \R$
\begin{equation}\label{compol}
\mathscr{L}^{-1}_{t(s)} \# \mu_0= \mu_{\q_s} .
\end{equation}
Actually, set $\Psi_{lin}(t,t')= e^{-i (t-t')H}$, then   by \eqref{conjlin} and \eqref{inv-lin}, for all measurable set $A \in \H^{-\eps}$
\begin{equation*}
\mu_0\big(\mathscr{L}_{t(s)}A \big)=  \mu_0\big( \Phi_{lin}(t(s), 0)  \Psi_{lin}^{-1}(s,0)  A \big)   =    \mu_0\big(   \Psi_{lin}^{-1}(s,0)  A \big)  = \mu_{\q_s}(A) \;,
\end{equation*}
hence the result.

%%%%%%%%%%%%%%%%%%%%%%%%%%%%%%%%%%%%%%%%%%%%%%%%%%%%%%%%%%%%%%%
%%%%%%%%%%%%%%%%%%%%%%%%%%%%%%%%%%%%%%%%%%%%%%%%%%%%%%%%%%%%%%%

\section{Measures and functional analysis}\label{sect.functional}

\subsection{On invariant measures}\label{sect-inv} Assume that $M$ is a compact manifold and denote by $\Delta_M$ the corresponding Laplace-Beltrami operator. Then there exists a Hilbert basis of $L^2(M)$, denoted by $(h_n)_{n\geq 0}$, composed  of eigenfunctions of $\Delta_M$ and we write $-\Delta_M h_n= \lambda^2_n h_n$ for all $n\geq 0$.
 
 Consider a  probability space
$(\Omega, {\mathcal F}, { \p})$ and let  $(g_n)_{n\geq 0}$  be a sequence of independent complex standard Gaussian variables.  Let $(\alpha_n)_{n \geq 0}$ and define the probability measure $\mu$   via the map 
 \begin{equation*}
 \begin{array}{rcl}
\dis  \omega&\longmapsto &\dis\gamma^{\om}=\sum_{n=0}^{+\infty} \alpha_n g_{n}(\om)h_n.
 \end{array}
 \end{equation*}
 Then, by the  invariance of the complex Gaussians  $(g_n)_{n\geq 0}$, by multiplication with $z_n$ such that ${|z_n| =1}$, for all $t\in \R$, the random variable $\e^{it \Delta_M} \dis\gamma^{\om}=\sum_{n=0}^{+\infty} \alpha_n e^{-it \lambda^2_n t}g_{n}(\om)h_n$ has the same distribution as $\gamma^{\om}$. In other words, the measure $\mu$ is invariant under the flow of the equation
  \begin{equation*}  
i\partial_sU+\Delta_MU=0,\quad (s,y) \in  \R\times M.
     \end{equation*}
The same remark holds when $M$ is replaced by $\R$, and $\Delta_M$ by $H=-\partial^2_x+x^2$ (in which case, the $(h_n)_{n\geq 0}$ are the Hermite functions and $\lambda^2_n=2n+1$). 
 
 Without some compactness in phase-space, the situation is dramatically different, as  shown by the next  elementary result.

\begin{prop}\label{propinv}
Let $\s\in \R$ and consider a probability measure $\mu$ on $H^{\sigma}(\R)$ (endowed with the cylindrical sigma-algebra). Assume that $\mu$ is invariant under the flow $\Sigma_{lin}$ of equation
 \begin{equation*}  
  \left\{
      \begin{aligned}
         &i\partial_sU+\partial_{y}^2U=0,\quad (s,y) \in  \R\times \R,
       \\  &  U(0,\cdot)  =U_0.
      \end{aligned}
    \right.
\end{equation*}
 Then $\mu=\delta_{0}$.
\end{prop}

\begin{proof}
Let $\s\in \R$ and assume that  $\mu$ is a probability measure on $H^{\sigma}(\R)$ which is invariant.  Let $\chi \in  {C}_0^{\infty}(\R)$. By invariance of the measure, for all $t \in \R$ we have 
\begin{equation*}  
 \int_{H^{\s}(\R)} \frac{\| \chi u \|_{H^{\sigma}}}{1+\|  u \|_{H^{\sigma}}}d \mu(u)= \int_{H^{\s}(\R)} \frac{\| \chi \Sigma_{lin}(t)u \|_{H^{\sigma}}}{1+\|  \Sigma_{lin}(t)u \|_{H^{\sigma}}}d \mu(u),
\end{equation*}
and by unitarity of the linear flow in $H^{\sigma}$, we get
\begin{equation}\label{chgt} 
 \int_{H^{\s}(\R)} \frac{\| \chi u \|_{H^{\sigma}}}{1+\|  u \|_{H^{\sigma}}}d \mu(u)= \int_{H^{\s}(\R)} \frac{\| \chi \Sigma_{lin}(t)u \|_{H^{\sigma}}}{1+\| u \|_{H^{\sigma}}}d \mu(u).
\end{equation}
We now prove that the right hand side of \eqref{chgt} tends to 0 when $t \to +\infty$. This will in turn imply that $ \| \chi u \|_{H^{\sigma}}=0$, $\mu-$almost surely, and therefore we will have, since the cut-off $\chi \in {C}_0^{\infty}(\R)$ is arbitrary,  that $u\equiv 0$ on the support of $\mu$, namely $\mu=\delta_0$.

By continuity of the product by $\chi$ in  $H^{\sigma}$ and unitarity of the linear flow in $H^{\sigma}$, we have
\begin{equation*} 
 \frac{ \| \chi \Sigma_{lin}(t)u \|_{H^{\sigma}}}{1+\| u \|_{H^{\sigma}}} \leq C  \frac{ \|   \Sigma_{lin}(t)u \|_{H^{\sigma}} }{1+\| u \|_{H^{\sigma}}}=C \frac{  \|    u \|_{H^{\sigma}}}{1+\| u \|_{H^{\sigma}}} \leq C.
\end{equation*}
Let $u \in H^{\s}(\R)$. For all $\delta>0$, there exists $u_{\delta} \in  {C}_0^{\infty}(\R)$ such that $\| u-u_{\delta}\|_{H^{\sigma}} \leq \delta$ and we have
\begin{equation} \label{decom25}
  \| \chi \Sigma_{lin}(t)u \|_{H^{\sigma}} \leq    C \|   \Sigma_{lin}(t)(u-u_{\delta}) \|_{H^{\sigma}}  + \| \chi \Sigma_{lin}(t)u_{\delta} \|_{H^{\sigma}}.
  \end{equation}
  The first term in the previous line can be bounded as follows
  \begin{equation*} 
  \|   \Sigma_{lin}(t)(u-u_{\delta}) \|_{H^{\sigma}}    \leq   C \|   u-u_{\delta} \|_{H^{\sigma}} \leq C\delta.
  \end{equation*}  
  For the second term we distinguish the cases $\sigma \geq 0$ and $\sigma \leq 0$. If $\sigma \leq 0$, 
  \begin{equation} \label{eq33}
 \| \chi \Sigma_{lin}(t)u_{\delta} \|_{H^{\sigma}} \leq  \| \chi \Sigma_{lin}(t)u_{\delta} \|_{L^2} \leq   \| \chi \|_{L^{4}}     \|\Sigma_{lin}(t) u_{\delta} \|_{L^4} .
  \end{equation}  
  Now we use the classical dispersion inequality $   \|\Sigma_{lin}(t) u_{\delta} \|_{L^4} \leq C t^{-1/4}     \|   u_{\delta} \|_{L^{4/3}} $, which proves, together with~\eqref{decom25} and~\eqref{eq33} that  $ \| \chi \Sigma_{lin}(t)u \|_{H^{\sigma}} \longrightarrow 0$, when $t \longrightarrow  \infty$. We conclude with the Lebesgue dominated convergence theorem. Now assume that $\sigma \geq 0$. Then by the fractional Leibniz rule
    \begin{equation*} 
 \| \chi \Sigma_{lin}(t)u_{\delta} \|_{H^{\sigma}} \leq    \| \chi \|_{W^{\sigma,4}}     \|\Sigma_{lin}(t) u_{\delta} \|_{W^{\sigma,4}}\;,
  \end{equation*}  
and  the  dispersion inequality $   \|\Sigma_{lin}(t) u_{\delta} \|_{W^{\sigma,4}} \leq C t^{-1/4}     \|   u_{\delta} \|_{W^{\sigma,4/3}} $, allows to conclude similarly.
\end{proof}

 The previous argument can be adapted to the case of a nonlinear equation, provided that one has a suitable global existence result with  scattering on the support of the measure. Let us illustrate this with the ($L^2-$critical) quintic Schr\"odinger equation.

 \begin{prop}\label{propinv2}
Consider a probability measure $\mu$ on $L^{2}(\R)$. Assume that $\mu$ is invariant under the flow $\Sigma$ of the equation
 \begin{equation}\label{NLS-ex}
  \left\{
      \begin{aligned}
         &i\partial_sU+\partial_{y}^2U=|U|^4U,\quad (s,y) \in  \R\times \R,
       \\  &  U(0,\cdot)  =U_0.
      \end{aligned}
    \right.
\end{equation}
 Then $\mu=\delta_{0}$.
\end{prop}
 
 \begin{proof}
By \cite{Dodson}, the equation \eqref{NLS-ex} is globally well-posed in $L^2(\R)$, and we denote by $\Sigma$ its flow. Moreover, the solution scatters: for all $U_0 \in L^2(\R)$ there exists $U^+_0 \in L^2(\R)$ such that 
 \begin{equation}\label{tt0}
 \|\Sigma(s)U_0-\e^{is\partial^2_y}U^+_0\|_{L^2(\R)}\longrightarrow 0,\quad 
\text{when}\quad s\longrightarrow +\infty.
\end{equation}
We follow the same strategy as in the proof of Proposition~\ref{propinv}. Assume that $\mu$ is a probability measure on~$L^2(\R)$ which is invariant by $\Sigma$ and  let $\chi \in  {C}_0^{\infty}(\R)$. By invariance of the measure, for all $s \in \R$ we have 
\begin{equation}\label{L12}
 \int_{L^2(\R)} \frac{\| \chi u \|_{L^1}}{1+\|  u \|_{L^2}}d \mu(u)= \int_{L^2(\R)} \frac{\| \chi \Sigma(s)u \|_{L^1}}{1+\|  \Sigma(s)u \|_{L^2}}d \mu(u)=\int_{L^2(\R)} \frac{\| \chi \Sigma(s)u \|_{L^1}}{1+\|  u \|_{L^2}}d \mu(u),
\end{equation}
where we used the conservation of the $L^2-$norm. By  Cauchy-Schwarz we have
\begin{equation*} 
\frac{\| \chi \Sigma(s)u \|_{L^1}}{1+\|  u \|_{L^2}} \leq   \frac{\|  \chi \|_{L^2}  \|  u \|_{L^2}}{1+\|  u \|_{L^2}}   \leq C.
\end{equation*}
 This bound allows to use the Lebesgue theorem to show that \eqref{L12} tends to 0, provided that we show a punctual decay. Actually,  
\begin{eqnarray*}  
\| \chi \Sigma(s)u \|_{L^1} &\leq&    \| \chi \e^{is\partial^2_y}u^+ \|_{L^1}  + \| \chi \big(\Sigma(s)u-\e^{is\partial^2_y}u^+\big) \|_{L^1}  \\
    &\leq&      \| \chi \e^{is\partial^2_y}u^+ \|_{L^1}  + \| \chi\|_{L^2} \| \Sigma(s)u-\e^{is\partial^2_y}u^+ \|_{L^2},   
\end{eqnarray*}
where $u^+$ is chosen as in \eqref{tt0}. The first term can be treated as in the proof of  Proposition~\ref{propinv} and the second tends to 0 by~\eqref{tt0}.
\end{proof}

Let us state a third result which shows that in the previous argument scattering can be replaced by the decay of the nonlinear solution in some $L^r$ norm, $r>2$.
 \begin{prop}\label{propinv3}
Consider a probability measure $\mu$ on $H^{1}(\R)$. Let $p>1$. Assume that $\mu$ is invariant under the flow $\Sigma$ of the equation
 \begin{equation}\label{NLS-ex1}
  \left\{
      \begin{aligned}
         &i\partial_sU+\partial_{y}^2U=|U|^{p-1}U,\quad (s,y) \in  \R\times \R,
       \\  &  U(0,\cdot)  =U_0.
      \end{aligned}
    \right.
\end{equation}
 Then $\mu=\delta_{0}$.
\end{prop}
 
 \begin{proof}
The  equation \eqref{NLS-ex1} is globally well-posed in $H^1(\R)$, by \cite{Ginibre-Velo}, and we denote by $\Sigma$ its flow. Under these assumptions, we do not know whether the solution scatters, but in~\cite{Visciglia}, Visciglia shows that when $u \in H^1(\R)$, then  for all $2<r\leq +\infty$, 
\begin{equation}\label{zero}
 \|\Sigma(s)u\|_{L^r(\R)} \longrightarrow 0,\quad \text{when} \quad s \longrightarrow +\infty.
\end{equation}
Assume that $\mu$ is a probability measure on $H^1(\R)$ which is invariant by $\Sigma$ and  let $\chi \in  {C}_0^{\infty}(\R)$. By invariance of the measure, as in \eqref{L12},    for all $s \in \R$, we get 
\begin{equation}\label{L39} 
 \int_{H^1(\R)} \frac{\| \chi u \|_{L^2}}{1+\|  u \|_{L^2}}d \mu(u)= \int_{H^1(\R)} \frac{\| \chi \Sigma(s)u \|_{L^2}}{1+\|  \Sigma(s)u \|_{L^2}}d \mu(u)=\int_{H^1(\R)} \frac{\| \chi \Sigma(s)u \|_{L^2}}{1+\|  u \|_{L^2}}d \mu(u).
\end{equation}
On the one hand we have the bound
\begin{equation*} 
\frac{\| \chi \Sigma(s)u \|_{L^2}}{1+\|  u \|_{L^2}} \leq   \frac{\|  \chi \|_{L^\infty}  \| \Sigma(s) u \|_{L^2}}{1+\|  u \|_{L^2}} = \frac{\|  \chi \|_{L^\infty}  \|  u \|_{L^2}}{1+\|  u \|_{L^2}}     \leq C,
\end{equation*}
and on the other hand, by \eqref{zero}
\begin{equation*} 
\frac{\| \chi \Sigma(s)u \|_{L^2}}{1+\|  u \|_{L^2}} \leq   \frac{\|  \chi \|_{L^2}  \|  \Sigma(s)u \|_{L^{\infty}}}{1+\|  u \|_{L^2}}   \longrightarrow 0,
\end{equation*}
when $s \longrightarrow +\infty$. By the Lebesgue theorem, every term in \eqref{L39}  cancels, which implies that $\mu-$a.s. ${\|\chi u\|_{L^2} =0}$, hence (since $\chi$ is arbitrary) the result. 
\end{proof}
\begin{rem}The main ingredient in all these results is the knowledge that locally in space, the solutions of the PDE (linear or nonlinear) tend to $0$ when $t\rightarrow + \infty$. It would be an interesting question to know whether this is true in the simplest case of~\eqref{C1}, for $1<p<5$ and for general initial data in $L^2$.
\end{rem}

\subsection{Some functional analysis}\label{sect32} 
\subsubsection{The harmonic oscillator} Let us recall some elementary facts concerning $H=-\partial_x^2 +x^{2}$  (we refer to \cite{Parme} for more details). 
The operator $H$ has a self-adjoint extension on $L^{2}(\R)$ (still denoted
by $H$) and has eigenfunctions $(e_{n})_{n \geq 0}$, called the Hermite functions, which form a Hilbert basis of $L^{2}(\R)$ and satisfy
$
He_n=\lambda_{n}^{2}e_n
$
with $\lambda_{n}=\sqrt{2n+1}$. 
Indeed, $e_n$ is given by the explicit formula 
\begin{equation*}
e_n(x)=(-1)^{n}c_{n}\,\e^{x^{2}/2}\frac{d^{n}}{dx^{n}}
\big(\,\e^{-x^{2}}\,\big),\;\;\text{with}\;\;\;\frac1{c_{n}}
=\big(n\,!\big)^{\frac12}\,2^{\frac{n}2}\,\pi^{\frac14}.
\end{equation*}
 
 \subsubsection{Projectors}
We define the finite dimensional  complex vector space $E_N$ by 
$$E_N={\rm span}_{\C}(e_0,e_1,\dots,e_N).$$
Then we introduce the spectral projector $\Pi_{N}$ on $E_{N}$ by 
\begin{equation*}
\Pi_{N}\big(\sum_{n=0}^{+\infty}c_n e_n\big)=\sum_{n=0}^{N}c_n e_n\,,
\end{equation*}
and we set $\Pi^N=I-\Pi_N$. Let $\chi\in  {C}_{0}^{\infty}(-1,1)$,  so that $\chi=1$ on $[-\frac12,\frac12]$ and $0\leq \chi \leq 1$. Let $S_{N}$ be the operator
\begin{equation}\label{S_NN}
S_{N}\big(\sum_{n=0}^{+\infty}c_n e_n\big)=\sum_{n=0}^{+\infty}\chi\big(\frac{2n+1}{2N+1}\big)c_n e_n
=\chi\big(\frac{H}{2N+1}\big)\big(\sum_{n=0}^{+\infty}c_n e_n\big)\,.
\end{equation}
It is clear that $\|S_{N}\|_{\mathcal{L} ( L^2 ( \mathbb{R}))}= \|\Pi_{N}\|_{\mathcal{L} ( L^2 ( \mathbb{R}))}=1$ and we have 
\begin{equation*} 
S_{N}\,\Pi_{N}=\Pi_{N}\,S_{N}=S_{N},\quad \text{and}\quad S_{N}^{*}=S_{N}.
\end{equation*}
 The smooth cut-off $S_{N}$ is continuous on all the  $L^q$ spaces, for  $1 \leq q\leq +\infty$ (see \cite[Proposition 4.1]{BTT}), 
\begin{equation}\label{prop.cont}
\| S_{N} \|_{\mathcal{L} ( L^q ( \mathbb{R}))} \leq C\;, 
\end{equation}
uniformly with respect to $N\geq 0$. Such a property does not hold true for $\Pi_N$.

 \subsubsection{Usual Sobolev and Besov spaces} The Sobolev spaces on $\R$ are defined
 for $\s \in \R$, $p\geq 1$ by 
 \begin{equation*} 
         W^{\s, p}= W^{\s, p}(\R) = \big\{ u\in  \mathscr{S}'(\R),\; (1-\Delta)^{\s/2}u\in L^p(\R)\big\},
       \end{equation*}
       and 
       \begin{equation*}
           H^{\s}=   H^{\s}(\R) = W^{\s, 2}.
       \end{equation*}
          These spaces are endowed by the  natural norms $\|u\|_{{W}^{\sigma,p}(\R)}=\|(1-\Delta)^{\s/2}u\|_{L^p(\R)}$.  
                  
         We consider a partition of unity on $\mathbb{R}^+$
\begin{equation}\label{partition}
1 = \sum_{j=0}^{+\infty} \chi_j( \xi), \qquad \forall j \geq 1,\quad   \chi_j(\xi ) = \chi(2^{-j} \xi), \quad \chi \in C^\infty_0 ( \frac 1 2 , 2),
\end{equation}
 and we define the Fourier multiplier ${\Delta}_j u = \chi_j(\sqrt{1-\Delta}) u$.  Let  $\sigma \in \R$ and $1\leq p,q\leq +\infty$.  Then  the Besov spaces   are defined by
        \begin{equation*} 
         {B}^{\sigma} _{p,q}= {B}^{\sigma} _{p,q}(\R) = \big\{ u\in \mathscr{S}'(\R), \;\; \|2^{j\sigma} \widetilde{\Delta}_j u \| _{L^p(\R)}  \in \ell ^q\big\},
       \end{equation*}
and we equip them with the natural norm
        \begin{equation} \label{norB}
         \| u\| _{B^{\sigma} _{p,q}} = \Bigl( \sum_{j \geq 0} \|2^{j\sigma}  {\Delta}_j u \| ^q_{L^p(\R)} \Bigr)^{\frac 1 q}.
       \end{equation}
 Moreover,   one can check that $B^{\sigma} _{2,2}=H^{\sigma}$.
 
Assume that $0<\sigma< 1$ and $1\leq p,q\leq +\infty$. Then by \cite[Theorem 2, p. 242]{Triebel}, the spaces $B^\sigma_{p,q}$ can be characterized by 
 \begin{equation} \label{charac}
{B}^{\sigma} _{p,q}=\Bigl\{ u \in L^p (\R), \;\; \| u \|_{L^p} + \Bigl(\int_{|t| <1}  \frac{ \| u(\cdot + t) - u( \cdot)\|_{L^p( \R)}^{q} } {|t|^{1+ \sigma q } }dt \Bigr)^{\frac 1 q} <+\infty \Bigr\}, 
       \end{equation}
and the corresponding norm is equivalent to  \eqref{norB}.

 \subsubsection{Sobolev and Besov spaces based on the harmonic oscillator} Similarly we define the Sobolev spaces 
associated to $H$  for $\s \in \R$, $p\geq 1$ by 
 \begin{equation*} 
         \W^{\s, p}= \W^{\s, p}(\R) = \big\{ u\in L^p(\R),\; {H}^{\s/2}u\in L^p(\R)\big\},
       \end{equation*}
       and 
       \begin{equation*}
           \H^{\s}=   {\mathcal H}^{\s}(\R) = \W^{\s, 2}.
       \end{equation*}
          These spaces are endowed by the  norms $\|u\|_{{\mathcal W}^{\sigma,p}(\R)}=\|H^{\sigma/2}u\|_{L^p(\R)}$.  It turns out  (see \cite[Lemma~2.4]{YajimaZhang2}), that  for $1<p<+\infty$,    and up to equivalence of norms we have
        \begin{equation} \label{eq-nor}
      \Vert u\Vert_{\W^{\s,p}} = \Vert  H^{\s/2}u\Vert_{L^{p}} \equiv \Vert (-\Delta)^{\s/2} u\Vert_{L^{p}} + 
       \Vert\<x\>^{\s}u\Vert_{L^{p}}.
 \end{equation}
       In particular, we recover the characterization   \eqref{Hs}.   Recall that we also have the following description of the~$\H^s$ norm: if $\dis u=\sum_{n=0}^{+\infty}c_n e_n$, then 
$
\dis  \|u\|^2_{\H^s}=\sum_{n=0}^{+\infty}\lambda_n^s |c_n|^2.
$
        
         We consider a partition of unity on $\mathbb{R}^+$ as in \eqref{partition} and we define the Hermite multiplier $\widetilde{\Delta}_j u = \chi_j(\sqrt{H}) u$.  Then  the Besov spaces based on the harmonic oscillator are defined by
        \begin{equation} \label{def-besov}
         \mathcal{B}^{\sigma} _{p,q}= \mathcal{B}^{\sigma} _{p,q}(\R) = \big\{ u\in \mathscr{S}'(\R), \;\; \|2^{j\sigma} \widetilde{\Delta}_j u \| _{L^p(\R)}  \in \ell ^q\big\},
       \end{equation}
and are   endowed with the natural norm
$$ \| u\| _{\mathcal{B}^{\sigma} _{p,q}} = \Bigl( \sum_{j \geq 0} \|2^{j\sigma} \widetilde{\Delta}_j u \| ^q_{L^p(\R)} \Bigr)^{\frac 1 q}.
$$    
In particular, one can check that $\mathcal{B}^{\sigma} _{2,2}=\H^{\sigma}$. Moreover, for any $1\leq p, q \leq + \infty$, and any $\rho\geq 0$, one has the continuous embedding
$ \mathcal{B}^{\rho} _{p,q} \subset B^{\rho} _{p,q}$ (see Lemma~\ref{lem-Besov}).

\subsection{Radon-Nikodym derivatives}
In the sequel we shall give quantitative estimates on the quasi-invariance of measures transported by linear or nonlinear flows. The next result shows that the bounds on the measures that we will obtain  in Proposition~\ref{lemeq} are actually equivalent to bounds on the Radon-Nikodym derivative.

\begin{prop}\label{prop-R-N}
Let $\mu, \nu$ be two finite  measures on a measurable space $(X,\mathcal{T})$. Assume that 
\begin{equation}\label{ac}
\mu \ll \nu,
\end{equation} and more precisely 
\begin{equation} \label{R-N}
\exists \,  0< \alpha \leq1, \quad \exists \,C>0, \quad \forall A \in \mathcal{T}, \quad \mu(A) \leq C \nu(A)^{\alpha}.
\end{equation}
By the  Radon-Nikodym theorem, assumption~\eqref{ac} implies that there exists a $f\in L^1(d\nu)$  with $f\geq 0$, such that  $ d \mu = f d\nu$. We call $f= \frac{d\mu}{d\nu}$ the Radon-Nikodym derivative of the measure $\mu$ with respect to the measure~$\nu$.

\begin{enumerate}[$(i)$]
\item The assertion~\eqref{R-N} is satisfied with $0< \alpha<1$ iff  $f\in L^p_{w}(d\nu)\cap L^1(d\nu)$ with   $p = \frac 1 {1- \alpha}$. In other words, $f\in L^1(d\nu)$ and 
$$\nu \big(\big\{ x : \; |f(x) |\geq \lambda \big\}\big) \leq C' \langle \lambda\rangle ^{- p}, \qquad \forall\, \lambda >0.$$
\item The assertion~\eqref{R-N} is satisfied with $ \alpha=1$ iff  $f\in L^{\infty}(d\nu)\cap L^1(d\nu)$.
\end{enumerate}
Recall that the weak $L^p$ spaces, denoted by $L^p_w$, satisfy
$$L^p_{w}(d\nu)\cap L^1(d\nu) \subset L^q( d\nu),\quad  \forall\, 1\leq q <p.$$
\end{prop}

\begin{proof}
The first part of the statement is the classical Radon-Nikodym theorem, see {\it e.g.}~\cite[Theorem~10.22]{CasRafe} for a proof.  

$(i)$ Assume that there exist $\alpha \in (0,1)$ and $C>0$ such that $\mu(A) \leq C \nu(A)^{\alpha}$. Let $\lambda>0$, then with $A=\big\{ f\geq \lambda \big\}$ we get 
\begin{equation*} 
\lambda \nu \big(f\geq \lambda \big) \leq \int_{\{ f\geq \lambda \}} fd\nu=\mu \big(f\geq \lambda \big)  \leq C{\nu \big(f\geq \lambda \big) }^{\alpha},
\end{equation*}
which implies  $ \nu \big(f\geq \lambda \big) \leq c \lambda^{-1/(1-\alpha)}$, which was the claim.

Assume now that $f= \frac{d\mu}{d\nu} \in L^p_{w}(d\nu)\cap L^1(d\nu)$ with   $p = \frac 1 {1- \alpha}$, and let $A \in \mathcal{T}$. Then 
\begin{equation} \label{decoup}
\mu(A)=\int_{\{ f\geq \lambda \}\cap A} fd\nu+ \int_{\{ f< \lambda \}\cap A} fd\nu \leq \int_{\{ f\geq \lambda \}} fd\nu+ \lambda\int_{    A} d\nu.
\end{equation}
Now we claim that for any $f \in L^p_{w}(d\nu)$ and any $E \in \mathcal{T}$ such that $\nu(E)<+\infty$,
\begin{equation}\label{weak-lp}
\dis \int_E fd\nu \leq \frac{p}{p-1}\nu(E)^{1-1/p} \|f\|_{L^p_w(d\nu)},
\end{equation}
where $\|f\|_{L^p_w(d\nu)}$ is defined by $\dis \|f\|^p_{L^p_w(d\nu)} =\sup_{\lambda>0} \big\{     \lambda^p \nu\big(f \geq \lambda \big)     \big\}$. For $\Lambda>0$, we write
\begin{eqnarray*} 
\dis \int_E fd\nu =\int_{0}^{+\infty} \nu\big( {\bf 1}_{E}f >\lambda    \big)  d\lambda       &=&  \int_{0}^{\Lambda} \nu\big( {\bf 1}_{E}f >\lambda    \big)  d\lambda  +       \int_{\Lambda}^{+\infty} \nu\big( {\bf 1}_{E}f >\lambda    \big)  d\lambda   \\
&\leq &   \nu(E)\Lambda +     \|f\|^p_{L^p_w(d\nu)}   \int_{\Lambda}^{+\infty} \lambda^{-p}  d\lambda \\
&=&   \nu(E)\Lambda +   \frac1{p-1}  \|f\|^p_{L^p_w(d\nu)}    \Lambda^{-p+1}.   
\end{eqnarray*}
Finally, we choose $\Lambda=\nu(E)^{-1/p} \|f\|_{L^p_w(d\nu)}$ which implies \eqref{weak-lp}. 

We apply \eqref{weak-lp} with  $E=\{ f\geq  \lambda \}$ and together with \eqref{decoup} we get
\begin{equation*}
\mu(A) \leq C{\nu \big(f\geq \lambda \big) }^{1-1/p} +\lambda \nu(A)\leq C \lambda^{-p+1}+\lambda \nu(A).
\end{equation*}
Now we optimize the previous inequality with $\lambda=\nu(A)^{-1/p}$.

$(ii)$ Assume that $f \notin L^{\infty}(d\nu)$. Then for all $M>0$, there exists $A \in \mathcal{T}$ such that $f>M$ on $A$ and $\nu(A)>0$. Then $\mu(A)=\int_A fd\nu>M\nu(A)$ which is the contraposition of \eqref{R-N}.

If $f \in L^{\infty}(d\nu)$, then for all $A \in \mathcal{T}$, $\mu(A) \leq \|f\|_{L^{\infty}(d\nu)} \nu(A)$, which is \eqref{R-N}.
\end{proof}

%%%%%%%%%%%%%%%%%%%%%%%%%%%%%%%%%%%%%%%%%%%%%%%%%%%%%%%%%%%%%%%
%%%%%%%%%%%%%%%%%%%%%%%%%%%%%%%%%%%%%%%%%%%%%%%%%%%%%%%%%%%%%%%

   \section{Evolution of the Gaussian measure}\label{sect.evol}
\subsection{Hamiltonian structure of the approximate problem}
Recall the definition \eqref{S_NN} of the operator $S_N$ and that $E_N={\rm span}_{\C}(e_0,e_1,\dots,e_N)$.  
We consider the truncated equation
 \begin{equation} \label{C3_N} 
  \left\{
      \begin{aligned}
         &i\partial_t u-Hu=\cos^{\frac{p-5}{2}}(2t)S_N(|S_Nu|^{p-1}S_Nu), \quad |t|<\frac{\pi}{4},\, x\in\R,
       \\  &  u_{|t= s}  \in E_N.
      \end{aligned}
    \right.
\end{equation}
 For $u\in E_{N}$, write 
$$
u=\sum_{n=0}^{N}c_n e_{n}=\sum_{n=0}^{N}(a_n+ib_n)e_{n},\quad a_n,b_n\in \R.
$$
Then we have the following result.
\begin{lem} 
Set
\begin{multline*}
J_N(t,u) = \int\Bigl( \frac{|\nabla_x u|^2 + |xu|^2 } 2 +\frac {\cos(2t)^{\frac {p-5}2}} {p+1} |S_Nu|^{p+1}\Bigr)dx =\\
\begin{aligned}
&= J_N(t, a_0,\dots,a_N,b_0,\dots, b_{N})\\
&= \frac{1}{2}\sum_{n=0}^{+\infty}\lambda_{n}^{2}(a_n^2+b_n^2)
+\frac{\cos(2t)^{\frac {p-5}2}}{p+1}
\big\|S_N\Big(\sum_{n=0}^{N}(a_n+ib_n) e_{n}\Big)\big \|_{L^{p+1}(\R)}^{p+1}\,.
\end{aligned}
\end{multline*}
The equation \eqref{C3_N} is a Hamiltonian ODE of the form
\begin{equation*}
\dot{a}_n=\frac{\partial J_N}{\partial b_n},\quad
\dot{b}_n=-\frac{\partial J_N}{\partial a_n},\quad 0\leq n\leq N.
\end{equation*}
\end{lem}

 Remark that $J_N$ is not conserved by the flow of~\eqref{C3_N}, due to the time dependance of the Hamiltonian~$J_N$, however the mass
\begin{equation*} 
\|u\|^{2}_{L^{2}(\R)}=\sum_{n=0}^{N}(a_n^2+b_n^2)
\end{equation*}
is conserved under the flow of \eqref{C3_N}. 
As a consequence, \eqref{C3_N} has a well-defined global flow ${\Phi}_{N}$ because it is actually an ordinary differential equation for $n\leq N$ and a linear equation for $n>N$.

Set
$$
\mathcal{E}_{N}(t,u(t))=\frac{1}{2}\|\sqrt{H}\,u(t)\|_{L^2(\R)}^2+\frac{\cos^{\frac{p-5}{2}}(2t)}{p+1}
\|S_Nu(t)\|_{L^{p+1}(\R)}^{p+1}\,.
$$
A direct computation shows that along the flow of \eqref{C3_N} one has
\begin{equation}\label{deri}
\frac{d}{dt}\big(\mathcal{E}_{N}(t,u(t))\big)=
\frac{(5-p)\sin(2t)\cos^{\frac{p-7}{2}}(2t)}{p+1}
\|S_Nu(t)\|_{L^{p+1}(\R)}^{p+1}\,.
\end{equation}
Observe that actually we have 
$$\mathcal{E}_{N}(t,u(t))= \mathcal{E}_{N}(t,\Pi_N(u(t))) + \frac 1 2 \| \sqrt{H}\, \Pi^N(u(t))\|_{L^2}^2,$$
and the last term is constant along the evolution given by ${\Phi}_N$. 

\subsection{Evolution of measures: a first result} Recall that the measures $\mu_N$ and $\mu^N$ are defined in \eqref{def-mun}, then we define the measure $\nu_{N,t}$  on~$E_N$ by
$$ d\nu_{N,t}= e^{-\frac{\cos^{\frac{p-5}{2}}(2t)}{p+1}\|S_N u\|_{L^{p+1}(\R)}^{p+1}}d\mu_N,\qquad \forall\, N\geq 1, \quad -\frac{\pi}4<t<\frac{\pi}4,  $$
and notice that $\nu_{N,t}$ has finite total mass but is not a probability measure. This definition implies that for all measurable set $A\subset E_N$, $t\in (- \frac \pi 4, \frac \pi 4)$,
\begin{equation*} 
  \nu_{N,t}(A) \leq \mu_N(A).
  \end{equation*}
  Similarly,  we define the measures~$\nu_{t}$ and $\widetilde{\nu}_{N,t}= {\nu}_{N, t}\otimes \mu^N$ by 
\begin{equation}\label{defnutilde}
d\nu_{t}= e^{-\frac{\cos^{\frac{p-5}{2}}(2t)}{p+1}\|u\|_{L^{p+1}(\R)}^{p+1}}d\mu_0, \qquad d\widetilde{\nu}_{N, t}= e^{-\frac{\cos^{\frac{p-5}{2}}(2t)}{p+1}\|S_Nu\|_{L^{p+1}(\R)}^{p+1}}d\mu_0,\qquad -\frac{\pi}4<t<\frac{\pi}4.
  \end{equation}

\begin{prop}\label{lemeq}
For all $s,t\in (- \frac \pi 4, \frac \pi 4)$,
\begin{equation}\label{eqmu}
 \Phi_N(t,s) _{\#} \mu_N \ll \mu_N \ll \Phi_N(t,s) _{\#}\mu_N. 
\end{equation}
More precisely, for all $0\leq |s| \leq |t| < \frac \pi 4$,
\begin{align} 
  \nu_{N,t}\big(\Phi_N(t,s)A\big) & \leq \begin{cases} \nu_{N,s}(A)&\text{ if } 1\leq p \leq 5 \\[5pt]
  \big[\nu_{N,s}(A)\big] ^{\smash{\bigl( \frac{\cos(2t)}{ \cos( 2s)}\bigr) ^{\frac{ p-5} 2}}} &\text{ if } p \geq 5, \end{cases}\nonumber\\
\intertext{and} 
\nu_{N,s}(A )&\leq \begin{cases}\Big[\nu_{N,t} \big(\Phi_N(t,s) A\big)\Big] ^{\smash{\bigl( \frac{\cos(2t)} {\cos(2s)}\bigr) ^{\frac{ 5-p} 2}}}&\text{ if } 1\leq p \leq 5\\[5pt]
 \nu_{N,t} \big(\Phi_N(t,s) A\big) &\text{ if }  p \geq 5. \end{cases}
 \label{abso-cont}
 \end{align}
 Another way to state the previous result is : \\
$\bullet $ For $1\leq p\leq 5$
$$ \big[\nu_{N,s}(A)\big]  ^{\smash{\bigl( \frac{\cos(2t)} {\cos(2s)}\bigr)^{\frac{ p-5} 2}}} \leq  \nu_{N,t} \big(\Phi_N(t,s) A\big) \leq\nu_{N,s}(A). $$
$\bullet $ For $p\geq 5$
$$\nu_{N,s}(A) \leq  \nu_{N,t} \big(\Phi_N(t,s) A\big) \leq  \big[\nu_{N,s}(A)\big] ^{\smash{\bigl( \frac{\cos(2t)} {\cos(2s)}\bigr) ^{\frac{ p-5} 2}}}. $$
\end{prop}

\begin{proof}
  By definition we have, for all $t\in  (- \frac \pi 4, \frac \pi 4)$ 
    \begin{equation*} 
    \mu_N \ll \nu_{N,t} \qquad \text{and}\qquad \nu_{N,t} \ll \mu_N, 
    \end{equation*}
    and in particular 
    \begin{equation}\label{eq}
    \mu_N \ll \nu_0 \qquad \text{and}\qquad \nu_0 \ll \mu_N. 
    \end{equation}
This will imply \eqref{eqmu} thanks to \eqref{eq}.  By \eqref{deri}, if we write  $u(t)=\Phi_N(t,s)u_0$, we have  
\begin{equation}\label{equadiff}
\begin{aligned}
\frac{d}{dt}\nu_{N,t}&(\Phi_N(t,s)A)=\\
&=\frac{d}{dt} \int_{v \in \Phi_N(t,s)A}  e^{-\frac{1}{2}\|\sqrt{H}\,v\|_{L^2(\R)}^2-
\frac{\cos^{\frac{p-5}{2}}(2t)}{p+1}\|S_N v\|_{L^{p+1}(\R)}^{p+1}}dv\\
&=\frac{d}{dt} \int_{v_M = \Pi_M(v) \in \Pi_M (\Phi_N(t,s)A_M)}  e^{-\frac{1}{2}\|\sqrt{H}\,v\|_{L^2(\R)}^2-
\frac{\cos^{\frac{p-5}{2}}(2t)}{p+1}\|S_N v\|_{L^{p+1}(\R)}^{p+1}}dv\\
&=\frac{d}{dt} \int_{u_{0,M}=\Pi_M(u_0) \in A_M} e^{-\frac{1}{2}\|\sqrt{H}\,u(t)\|_{L^2(\R)}^2-
\frac{\cos^{\frac{p-5}{2}}(2t)}{p+1}
\|S_N u(t)\|_{L^{p+1}(\R)}^{p+1}}du_{0,M}\\
&=    \int_{A_M} \frac{(p-5)\sin(2t)\cos^{\frac{p-7}{2}}(2t)}{p+1}
\|S_Nu(t)\|_{L^{p+1}(\R)}^{p+1} e^{-\mathcal{E}_{N}(t,u(t))}du_{0,M}\\
&=    \int_{A} \frac{(p-5)\sin(2t)\cos^{\frac{p-7}{2}}(2t)}{p+1}
\|S_Nu(t)\|_{L^{p+1}(\R)}^{p+1} e^{-\mathcal{E}_{N}(t,u(t))}du_{0}\\
&=  (p-5) \tan(2t)  \int_A \alpha\big(t,u(t)\big)  e^{-\mathcal{E}_{N}(t,u(t))}du_0,
\end{aligned}
\end{equation}
where $\alpha(t,u)=\frac{\cos^{\frac{p-5}{2}}(2t)}{p+1}
\|S_N u\|_{L^{p+1}(\R)}^{p+1}$, and to pass from the first line to the second line we used that according to Liouville Theorem (see Appendix ~\ref{Sec.E}) the Jacobian of the change of variables $v= \Phi_N(t,s) u_0 \mapsto u_0$ is equal to $1$.

In the following, we assume that $0\leq s \leq t < \frac \pi 4 $.  If $1 \leq p\leq 5$ the r.h.s. of~\eqref{equadiff} is non positive and we get by monotonicity
\begin{equation*} 
  \nu_{N,t}\big(\Phi_N(t,s)A\big) \leq \nu_{N,s}(A).
  \end{equation*}
  When $p\geq 5$  by the H\"older inequality, for any $k\geq 1$,
\begin{eqnarray*}
\frac{d}{dt}\nu_{N,t}(\Phi_N(t,s)A)& \leq&  (p-5) \tan(2t)  \big( \int_A \alpha^k(t,u)  e^{-\mathcal{E}_{N}(t,u(t))}du_0\big)^{\frac1k} \big( \int_A    e^{-\mathcal{E}_{N}(t,u(t))}du_0\big)^{1-\frac1k}\\
&= &(p-5) \tan(2t)  \big( \int_A \alpha^k(t,u)  e^{- \alpha(t,u)-\frac{1}{2}\|\sqrt{H}\,u(t)\|_{L^2(\R)}^2}du_0\big)^{\frac1k} \big(\nu_{N,t}(\Phi_N(t,s)A)\big)^{1-\frac1k}.
\end{eqnarray*}
We use that $ \alpha^k(t,u)  e^{- \alpha(t,u) }\leq k^k\e^{-k} $, then 
\begin{equation*}
\frac{d}{dt}\nu_{N,t}(\Phi_N(t,s)A) \leq  (p-5) \tan(2t) \frac { k} e\big(\nu_{N,t}(\Phi_N(t,s)A)\big)^{1-\frac1k}.
\end{equation*}
We now optimize the inequality above by choosing $k = - \log\big(\nu_{N,t}( \Phi_N(t,s)A)\big)$, which gives
$$ 
\frac{d}{dt}\nu_{N,t}\big(\Phi_N(t,s)A\big) \leq - (p-5) \tan(2t) \log\Big(\nu_{N,t}\big( \Phi_N(t,s)A\big)\Big)\nu_{N,t}\big(\Phi_N(t,s)A\big).
$$
This in turn  implies 
$$
- \frac d {dt}\log \Big( - \log\big(\nu_{N,t}(\Phi_N(t,s)A) \big)\Big) \leq ( p-5) \tan (2t) = -\frac{ (p-5)} 2 \frac d {dt}\log(\cos(2t)) 
$$ 
and consequently 
$$- \log\big(\nu_{N,t}(\Phi_N(t,s)A)\big) \geq -  \log\big(\nu_{N,s}(A)\big) \Big(\frac{ \cos(2t)}{ \cos(2s)}   \Big) ^{\frac{ (p-5) }2}.$$ 
Therefore, for all  $0\leq s \leq t < \frac \pi 4 $,
\begin{equation*} 
  \nu_{N,t}\big(\Phi_N(t,s)A\big) \leq  \big[\nu_{N,s}(A)\big] ^{\smash{\bigl( \frac{\cos(2t)}{ \cos( 2s)}\bigr) ^{\frac{ p-5} 2}}} .
  \end{equation*}
 The reverse inequality is obtained by backward integration of the estimate  and reads similarly when $p\geq 5$ 
 \begin{equation*} 
 \nu_{N,s}(A) \leq  \nu_{N,t} \big(\Phi_N(t,s) A\big) ,
 \end{equation*}
and when $1<p\leq 5 $ we get 
\begin{equation*} 
 \nu_{N,s}(A) \leq \Big[\nu_{N,t} \big(\Phi_N(t,s) A\big)\Big] ^{\smash{\bigl( \frac{\cos(2t)} {\cos(2s)}\bigr) ^{\frac{ 5-p} 2}}},
 \end{equation*}
 which concludes the proof.
\end{proof}
\begin{rem}\label{extension}
We can extend the flow $\Phi_N$ to be the flow of
 \begin{equation*} 
  \left\{
      \begin{aligned}
         &i\partial_t u-Hu=\cos^{\frac{p-5}{2}}(2t)S_N(|S_Nu|^{p-1}S_Nu), \quad |t|<\frac{\pi}{4},\, x\in\R,
       \\  &  u_{|t= s}  \in \H^{-\epsilon},
      \end{aligned}
    \right.
\end{equation*}
where  the extension on $E_N^\perp= \text{Vect} \{e_n, n >N\}$ is given by the linear flow $e^{-itH} u$.
Using that $\widetilde{\nu}_{N,t} = \nu_{N,t} \otimes \mu^N$ and in the decomposition 
$$\H^{- \epsilon} = E_N \times E_N^\perp$$
the flow takes the form 
 \begin{equation}\label{def-phitilde}
\widetilde{\Phi}_N: = \Phi_N\otimes e^{-itH}.
\end{equation}
Since the measure $\mu^N$ is invariant by the flow of $e^{-itH}$, we get that with this extension,  Proposition~\ref{lemeq} is still true with  $\nu_{N,t}$ replaced by $\widetilde{\nu}_{N,t}$, see definition \eqref{defnutilde}.

\end{rem}

%%%%%%%%%%%%%%%%%%%%%%%%%%%%%%%%%%%%%%%%%%%%%%%%%%%%%%%%%%%%%%%%%%%%%%
%%%%%%%%%%%%%%%%%%%%%%%%%%%%%%%%%%%%%%%%%%%%%%%%%%%%%%%%%%%%%%%%%%%%%%

\section{Non linear estimates }\label{Sect5}

\subsection{Strichartz estimates} To begin with, we recall the Strichartz  estimates for the harmonic oscillator.  A couple $(q,r)\in [2,+\infty]^2$ is called admissible if 
\begin{equation*}
\frac2q+\frac{1}{r}=\frac{1}2,
\end{equation*}
and if one defines 
\begin{equation*}
 {Z}^{\s}_{T}:=  L ^\infty\big( [-T, T] \,; \H^{\s}( \R )\big)\cap L^4 \big( [-T, T] \,; \W^{\s, \infty}( \R )\big),
\end{equation*}
then for all $T>0$ there exists $C_{T}>0$ so that for all $u_{0}\in \H^{\s}(\R)$ we have 
\begin{equation} \label{stri0}
\|\e^{-itH}u_{0}\|_{Z^{\s}_{T}}\leq C_{T}\|u_{0}\|_{\H^{\s}(\R)}.
\end{equation}
We will also need the inhomogeneous version of the Strichartz estimates:  for all $ T>0 $, there exists $C_{T}>0$ so that for all   admissible couple       $ ( q, r ) $ and function  $ F \in L^{q'}( (T,T); \W^{\s,r'} (\R)) $,
\begin{equation} \label{stri1} 
 \big\|  \int _0^t \e^{-i(t-s)H} F(s) ds   \big\| _{   Z^\s_T} \leq C_{T} \| F \|_{  L^{q'} ((-T,T);{\W}^{\s,r'} (\R)) },
\end{equation}
where $ q' $ and $ r'$ are the H\"older conjugate of $ q $ and $ r $. We refer to \cite{poiret2} for a proof.

In the sequel, we will also need similar estimates in Besov spaces (recall definition \eqref{def-besov}).
\begin{prop}\label{prop.5.1}
Assume that $(q,r)$ is admissible and $\rho \geq 0$. Then there exists $C_T>0$ such that 
\begin{equation}\label{besov1}
\| e^{-itH} u_0 \|_{L^q ((-T, T); \mathcal{B}^{\rho}_{r,2})} \leq C_T \| u_0\|_{\mathcal{H}^\rho},
\end{equation}
and
\begin{equation}\label{besov2}
\big\|  \int _0^t \e^{-i(t-s)H} F(s) ds   \big\| _{  {L^q ((-T, T); \mathcal{B}^{\rho}_{r,2})}} \leq C_{T} \| F \|_{L^1((-T, T); \mathcal{H}^\rho)}.
\end{equation}
\end{prop}

\begin{proof}
For the  inequality~\eqref{besov1}, we use the Minkowski inequality  (because $q \geq2$) and \eqref{stri0}
\begin{multline*} \| e^{-itH} u_0 \|_{L^q ((-T, T); \mathcal{B}^{\rho}_{r,2})} = \|2^{j\rho} \widetilde{\Delta}_j e^{-itH} u_0 \|_{L^q_t;  \ell^2_j; L^r_x} \leq \\
 \leq \|2^{j\rho} e^{-itH} \widetilde{\Delta}_j u_0\|_{  \ell^2_j; L^q_t; L^r_x}\leq C \| 2^{j\rho} \widetilde{\Delta}_j   u_0\|_{\ell^2_j} \leq C \| u_0 \|_{\H^\rho}\;,
\end{multline*}
which was the claim. The inequality \eqref{besov2}  follows from   \eqref{besov1}  and the Minkowski inequality. 
\end{proof}
Let 
$$F(u) = |u|^{p-1} u.$$
 In the following, the analysis of the Cauchy problem will be different according whether $p\geq 2$ (the nonlinearity is $C^2$) or $1<p<2$ (the nonlinearity is only $C^1$). 

\subsection{The estimates for the local theory}\label{sec.5.1}  In this section we prove the main estimates allowing to perform a fixed point to establish the local existence for our nonlinear equations. When $p\leq2$,  the lack of smoothness of the nonlinearity $F(u) = |u|^{p-1} u$ forces us to prove {\em a priori} estimates in strong norm (see Proposition~\ref{prop.5.3}) and contraction in weak norm (see Proposition~\ref{prop.5.4}). As a consequence, in Section~\ref{sec.8} we shall in this case perform a quasi-linear type fixed point.

The starting point is the following set of results from Christ-Weinstein~\cite{ChWe91}. In the sequel, $D_x^\alpha u$ is defined as a  tempered distribution using the Fourier transform: $\mathcal{F}(D_x^\alpha  u)(\xi)= |\xi|^{\alpha} \mathcal{F}(u)(\xi)$.

\begin{prop} [\protect{\cite[Proposition 3.1]{ChWe91}}]\label{prop.C.W.1} Assume that $G\in C^1( \C;\C)$, and let $0\leq \alpha\leq 1$, $1< p,q,r<+\infty$, $r^{-1} = p^{-1} + q^{-1}$. Assume that  $u \in W^{\alpha, q}$ and $G'(u) \in L^p$, then $G(u) \in W^{\alpha , r}$, and 
$$ \| D_x^\alpha G(u) \|_{L^{r}} \leq C  \| G'(u) \| _{L^p} \| D_x^\alpha u \|_{L^{ q}}.
$$
\end{prop}
\begin{prop}[\protect{\cite[Proposition 3.3]{ChWe91}}]\label{prop.C.W.2}
Let $0\leq \alpha \leq 1$, $1< r, p_i, q_i <+\infty$, $ i=1,2$, and $r^{-1} = p_i^{-1} + q_i^{-1}$. Assume that $f\in L^{p_1}, D_x^\alpha f \in L^{p_2}$, and $g\in L^{q_2}, D_x^\alpha g \in L^{q_1}$. Then $D_x^\alpha (fg) \in L^{r}$ and 
$$ \| D_x^\alpha (fg) \|_{L^{r}} \leq C\Bigl( \| f\| _{L^{p_1}} \| D_x^\alpha g\|_{L^{q_1}} + \|D_x^\alpha  f\| _{L^{ p_2}} \| g\|_{L^{q_2}} \Bigr). 
$$
\end{prop}
From the previous results we deduce 
\begin{prop}\label{prop.5.3} Let $p>1$ and $\rho,\sigma\geq 0$. Then 
\begin{equation}\label{5.3.0}
 \| F(u) \|_{\mathcal{H}^\rho ( \mathbb{R})} \leq C \|\frac{ u}{ \langle x \rangle ^{\sigma}} \|_{\mathcal{W}^{\rho, 4} ( \mathbb{R})}\| \langle x \rangle ^{\frac{\sigma} {p-1}}u\|^{p-1}_{L^{4(p-1)}(\R)}.
  \end{equation}
Assume  moreover that  $ 1<p<\frac 3 2 $. Then for any $2< r\leq \frac{2} {1- 2( p-1)}$, there exists $C>0$ such that, with  $s=\frac{ 2( p-1)r} {r-2} \;(\geq 1) $ 
\begin{equation}\label{5.4.1}
 \| F(u) \|_{\mathcal{H}^\rho ( \mathbb{R})} \leq C \| u \|_{\mathcal{W}^{\rho, r} ( \mathbb{R})}\| u\|^{p-1}_{L^{s}(\R)}. 
  \end{equation}
\end{prop}
\begin{prop}\label{prop.5.4}
Let $ p>2 $ and $\rho,\sigma\geq 0$. Then there exists $C>0$ such that
\begin{multline}\label{5.5.0}
 \| F(u)- F(v) \|_{\H^\rho ( \mathbb{R})} \leq C  \| \frac{ u-v}{ \langle x \rangle ^{\sigma}}\|_{\W^{\rho, 4}} \bigl( \| \langle x \rangle ^{\frac{\sigma} {p-1}}u\|^{p-1}_{L^{4(p-1)}} + \| \langle x\rangle ^{\frac{\sigma} {p-1}}v\|^{p-1}_{L^{4(p-1)}}   \bigr)\\
+C   \| \langle x \rangle ^{\frac{\sigma} {p-1}}(u-v)\|_{L^{4(p-1)}}\bigl(\| \langle x\rangle ^{\frac{\sigma} {p-1}}u\|_{L^{4(p-1)}}^{p-2} + \| \langle x \rangle ^{\frac{\sigma} {p-1}}v\|_{L^{4(p-1)}}^{p-2}\bigr) \bigl(\| \frac{u} {\langle x \rangle ^{\sigma}}\|_{\W^{\rho, 4}} + \| \frac{v} {\langle x \rangle ^{\sigma}} \|_{\W^{\rho,4}} \bigr).
 \end{multline}
Let $ 1< p  \leq 2 $, then for any $ r\geq 2$, there exists $C>0$ such that, with  $s=\frac{ 2( p-1)r} {r-2} $ 
   \begin{equation}\label{5.7.1}
  \| F(u)- F(v) \|_{L^2 ( \mathbb{R})} \leq C \| u-v\|_{L^r ( \mathbb{R})} \big( \| u\|^{p-1}_{L^s(\R)}+ \| v\|^{p-1}_{L^{s}(\R)} \big).
  \end{equation}
  \end{prop}
  
 \begin{proof}[Proof of Proposition~\ref{prop.5.3}] Let us first show \eqref{5.3.0}. We use \eqref{eq-nor} which gives
  \begin{equation*} 
 \| F(u) \|_{\mathcal{H}^\rho ( \mathbb{R})} \sim \| \langle x \rangle ^{\rho} F(u) \|_{L^2(\R)} + \| D_x^\rho F(u)\|_{L^2(\R)}.
  \end{equation*}
 The contribution of the first term  is bounded by 
 $$\| \langle x \rangle ^{\rho } F(u) \|_{L^2} \leq \| \langle x \rangle ^{\rho} \frac{ 1} {\langle x \rangle ^{\sigma} }u \|_{L^4} \| \langle x \rangle ^{\sigma} |u|^{p-1}\|_{L^4} \leq C  \| \frac{ 1} {\langle x \rangle ^{\sigma} } u \|_{\W^{\rho,4}} \| \langle x \rangle ^{\frac{\sigma} {p-1}} u\|_{L^{4(p-1)}}^{p-1}
 $$
 while the contribution of the second term is bounded using Proposition~\ref{prop.C.W.1} with the choice $G(z) = |z|^{p-1} z $ which is $C^1$ because $p>1$ (it is clearly $C^1$ away from $(0,0)$ and its differential is homogeneous of degree $p-1$ in $(x,y)$, hence vanishing at $(0,0)$). Therefore we obtain 
    \begin{equation}\label{cas0}
 \| D_x^{\rho } F(u) \|_{L^2} \leq C\| D_x^\rho u\|_{L^4} \| |u|^{p-1}\|_{L^4} \leq C  \| u \|_{\W^{\rho,4}} \| u\|_{L^{4(p-1)}}^{p-1},
    \end{equation}
which is the result with $\sigma=0$. In order to treat the general case $\sigma \geq 0$, we  introduce a partition of unity 
\begin{equation}\label{partition2}
 1 = \sum_{j=0}^{+\infty} \chi_j (x), \quad \forall j \geq 1, \;\;  \chi_j(x) = \chi(2^{-j} x), \quad \text{ supp }  \chi \subset \big\{x; \;\; \frac 1 2\leq |x| \leq 2\big\}\;,
\end{equation}
and choose  $\widetilde{\chi}_j$ equal to $1$ on the support of $\chi_j$. Then by \eqref{cas0}
\begin{eqnarray*}
\|  F(u)\|^2_{\H^\rho } &\leq &\sum_{j\geq 0} \|\chi_j F(u)\|^2_{\H^\rho}  =\sum_{j\geq 0} \|\chi_j F(\widetilde{\chi}_ju)\|^2_{\H^\rho}\\&\leq &C  \sum_{j\geq 0}  \| \widetilde{\chi}_j u \|^2_{\W^{\rho,4}} \| \widetilde{\chi}_j u\|_{L^{4(p-1)}}^{2(p-1)}
=C  \sum_{j\geq 0}  \| \widetilde{\chi}_j  \frac{u} {2^{j \sigma} } \|^2_{\W^{\rho,4}} \| 2^{j\frac{\sigma}{p-1}} \widetilde{\chi}_j u\|_{L^{4(p-1)}}^{2(p-1)}.
\end{eqnarray*}
Therefore we get 
\begin{eqnarray*}
\|  F(u)\|^2_{\H^\rho } &
\leq& C  \Big(\sum_{j\geq 0}  \| \widetilde{\chi}_j \frac u {\langle x \rangle ^{\sigma}} \|^4_{\W^{\rho,4}} \Big)^{\frac 1 2} \Big(\sum_{j\geq 0} \| \widetilde{\chi}_j {\langle x \rangle ^{\frac \sigma {p-1}}} u\|^{4(p-1)}_{L^{4(p-1)} }\Big)^{\frac 1 2}\\
&\leq & C   \| \frac u {\langle x \rangle ^{\sigma}} \|^2_{\W^{\rho,4}}\|{\langle x \rangle ^{\frac \sigma {p-1}}} u\|^{2(p-1)}_{L^{4(p-1)} }\;, 
\end{eqnarray*}
which gives~\eqref{5.3.0}. To get~\eqref{5.4.1}, we use Proposition~\ref{prop.C.W.1} to estimate
 $$\| D_x^{\rho } F(u) \|_{L^2} \leq C\| D_x^\rho u\|_{L^r} \| |u|^{p-1}\|_{L^{\frac{2r} {r-2}} }\leq C  \| u \|_{\W^{\rho,r}} \| u\|_{L^s}^{p-1}\;,
 $$
 which was the claim.
\end{proof}

 \begin{proof}[Proof of Proposition~\ref{prop.5.4}]
Let us now turn to the proof of~\eqref{5.5.0}.  By the Taylor formula,
\begin{multline}\label{513.1}
F(u) - F(v) = ( u-v) \int_{0}^1 \partial_z F\big(v + \theta (u-v), \overline{v} + \theta (\overline{u-v}) \big)d \theta +\\
+ ( \overline{u-v}) \int_{0}^1 \partial_{\overline{z}} F\big(v + \theta (u-v), \overline{v} + \theta (\overline{u-v}) \big)d \theta.
\end{multline}
On the first hand, since $|\partial F (z, \overline{z} ) |\leq C |z| ^{p-1}$, we deduce 
$$\| \langle x \rangle ^\rho \big(F(u) - F(v)\big) \| _{L^2} \leq C \| \frac{1} {\langle x \rangle ^\sigma}\langle x \rangle ^\rho (u-v)\|_{L^4} \Bigl( \| \langle x \rangle^{\frac \sigma {p-1}}u\|_{L^{4(p-1)}}^{p-1}+   \| \langle x \rangle^{\frac \sigma {p-1}}v\|_{L^{4(p-1)}}^{p-1} \Bigr).
$$
On the other hand, according to Proposition~\ref{prop.C.W.2}, with $r=2, (p_1, q_1) = (4(p-1), \frac{ 4(p-1) }{2p -3}), (p_2, q_2) = (4,4)$, we have, using again the same partition of unity~\eqref{partition2},
$$ \| F(u) - F(v) \|^2_{\H^\rho} \sim \sum_{j\geq 0} \| \chi_j(F(u) - F(v)) \|^2_{\H^\rho}  =  \sum_{j\geq 0} \| \chi_j(F(\widetilde{\chi}_j u) - F(\widetilde{\chi}_j v)) \|^2_{\H^\rho}.
$$
Next, for all $j\geq 0$, we compute
\begin{multline}\label{5.9.1}
\big\| D_x^\rho \Bigl(\widetilde{\chi}_j  ( u-v) \int_{0}^1 \partial_z F\big(\widetilde{\chi}_j (v + \theta  (u-v)), \widetilde{\chi}_j (\overline{v} + \theta(\overline{u-v}))\big) d \theta\Bigr)\big\|_{L^2} \leq \\
\leq  C   \| \widetilde{\chi}_j (u-v)\|_{L^{4(p-1)}} \big\| D_x^\rho   \int_{0}^1 \partial_z F\big(\widetilde{\chi}_j (v + \theta (u-v)), \widetilde{\chi}_j (\overline{v} + \theta (\overline{u-v})) \big)d \theta\big\|_{L^{\frac{ 4(p-1)}{2p-3}}}\\
+C \| D_x^\rho\widetilde{\chi}_j  (u-v)\|_{L^4} \big\|\int_{0}^1 \partial_z F\big(\widetilde{\chi}_j (v + \theta (u-v)), \widetilde{\chi}_j (\overline{v} + \theta (\overline{u-v}) )\big)d \theta\big\|_{L^4} .
\end{multline}
The second term in the r.h.s. of~\eqref{5.9.1} is easily bounded by (recall that $|\partial_z F |\leq C |z|^{p-1}$) 
\begin{multline}\label{5.10.1}
 \big\| D_x^\rho \widetilde{\chi}_j  (u-v)\big\|_{L^4} \big\|  \int_{0}^1 \partial_z F\big(\widetilde{\chi}_j (v + \theta (u-v)), \widetilde{\chi}_j (\overline{v} + \theta (\overline{u-v})) \big)d \theta\big\|_{L^4} \leq \\
 \leq \big\|D_x^\rho  \widetilde{\chi}_j (u-v)\big\|_{L^4}\int_{0}^1  \big\|   \partial_z F\big(\widetilde{\chi}_j (v + \theta (u-v)), \widetilde{\chi}_j (\overline{v} + \theta (\overline{u-v})) \big)\big\|_{L^4}d \theta\\
 \leq \big\|  D_x^\rho \widetilde{\chi}_j(u-v)\big\|_{L^4}\Bigl( \| \widetilde{\chi}_j v\|_{L^{4(p-1)}}^{p-1} + \| \widetilde{\chi}_j  u\|_{L^{4(p-1)}}^{p-1} \Bigr).
 \end{multline}
 To bound the first term in the r.h.s. of~\eqref{5.9.1}, we apply Proposition~\ref{prop.C.W.1} with the choice of functions $G(u) = \partial_z F(u, \overline{u})$ which is $C^1$ (because $p>2$ and the second derivative which is defined away from $(0,0)$ is homogeneous of degree $p-2$, hence vanishing at $(0,0)$). We get with the choice $(r,p,q) =(\frac{4(p-1)}{2p-3},\frac{4(p-1)}{p-2}, 4)$,  
 \begin{multline}\label{5.11.1}
 \big\| D_x^\rho     \int_{0}^1 \partial_z F\big(\widetilde{\chi}_j (v + \theta (u-v)), \widetilde{\chi}_j (\overline{v} + \theta (\overline{u-v}) )\big) d \theta\big\|_{L^{\frac{ 4(p-1)}{2p-3}}} \leq \\
  \begin{aligned}
 &\leq \int_{0}^1 \big\| D_x^\rho   \partial_z F\big(\widetilde{\chi}_j (v + \theta (u-v)), \widetilde{\chi}_j (\overline{v} + \theta (\overline{u-v}) )\big)  \big\|_{L^{\frac{ 4(p-1)}{2p-3}}}d \theta\\
& \leq C\int_{0}^1 \|\widetilde{\chi}_j   |v + \theta (u-v)|^{p-2} \|_{L^{\frac{4(p-1)}{p-2}}} \| D_x^\rho \widetilde{\chi}_j  (v+ \theta (u-v)) \| _{L^4}\\
& \leq C \Bigl(\|\widetilde{\chi}_j   u\|_{L^{4(p-1)}}^{p-2} + \| \widetilde{\chi}_j  v\|_{L^{4(p-1)}}^{p-2}\Bigr) \Bigl(\|  D_x^\rho \widetilde{\chi}_j u\|_{L^4} + \|D_x^\rho \widetilde{\chi}_j  v \|_{L^4} \Bigr).
   \end{aligned}
 \end{multline}
  From~\eqref{5.9.1},~\eqref{5.10.1} and~\eqref{5.11.1} we deduce
  \begin{multline*}
 \big\| D_x^\rho \Bigl(\widetilde{\chi}_j  ( u-v) \int_{0}^1 \partial_z F\big(\widetilde{\chi}_j (v + \theta  (u-v)), \widetilde{\chi}_j (\overline{v} + \theta(\overline{u-v}))\big) d \theta\Bigr)\big\|_{L^2} \leq \\
\qquad \leq  C \| D_x^\rho\widetilde{\chi}_j   (u-v)\|_{L^4}\big( \| \widetilde{\chi}_j v\|_{L^{4(p-1)}}^{p-1} + \| \widetilde{\chi}_j u\|_{L^{4(p-1)}}^{p-1}  \big)\hfill\\
+  C\| \widetilde{\chi}_j (u-v)\|_{L^{4(p-1)}}\bigl(\| \widetilde{\chi}_j u\|_{L^{4(p-1)}}^{p-2} + \| \widetilde{\chi}_j v\|_{L^{4(p-1)}}^{p-2}\bigr) \bigl(\| D_x^\rho \widetilde{\chi}_j u\|_{L^4} + \| D_x^\rho \widetilde{\chi}_j v \|_{L^4} \bigr).
\end{multline*}
Putting the weights $\langle x\rangle$ and using that these weights are essentially constant on the support of $\widetilde{\chi}_j $, we get  the estimate for the contribution of the first term in the r.h.s. of~\eqref{513.1}. The estimate for the second term is similar. 
This concludes the proof of ~\eqref{5.5.0}.  Finally~\eqref{5.7.1} follows from~\eqref{513.1} and the H\"older inequality.
\end{proof}

\subsection{The estimates for the continuity of the flow}
\begin{prop}\label{prop.besov}  Let  $ \frac 3 2 \leq  p \leq 2 $ and $\rho, \sigma \geq 0$. Then there exists $C>0$ such that
 \begin{multline}\label{5.6.1bis}
  \| F(u)- F(v) \|_{\mathcal{H}^\rho ( \mathbb{R})} \leq  \\
\leq C \| \frac{u-v}{\langle x \rangle ^\sigma}\|_{\mathcal{B}^{\rho}_{4,2}  } \Bigl( \| \langle x \rangle ^{\frac \sigma {p-1}}u\|^{p-1}_{L^{4(p-1)} }+ \| \langle x \rangle ^{\frac \sigma {p-1}} v\|^{p-1}_{L^{4(p-1)}}  \Bigr)  + C \| \frac v{\langle x \rangle ^\sigma}\|_{\mathcal{B}^{\rho}_{4,2} ( \mathbb{R})} \|\langle x \rangle ^{\frac \sigma {p-1}}( u-v)\|^{p-1}_{L^{4(p-1)} } .
  \end{multline}
Let  $ 1< p <\frac 3 2 $  and $\rho \geq 0$. Then for any $2< r< \frac{2} {2-p}$, there exists $C>0$ such that, with  $s=\frac{ 2( p-1)r} {r-2}>2 $ 
   \begin{equation}\label{5.7.1bis}
  \| F(u)- F(v) \|_{\H^\rho ( \mathbb{R})} \leq C \| u-v\|_{B^{\rho}_{r,2}  } \Big( \| u\|^{p-1}_{L^{s} }+ \| v\|^{p-1}_{L^{s}}  \Big) + C \| v\|_{B^{\rho}_{r,2} ( \mathbb{R})} \| u-v\|^{p-1}_{L^s } .
  \end{equation}
 \end{prop}
\begin{proof}
We first consider the case $\sigma=0$. Since $0<\rho<1$, we can use the    characterization~\eqref{charac} of the usual $H^{\rho}(\R)$ norm, namely
\begin{equation}\label{diffsobo}
\|g\|^2_{H^{\rho}(\R)}  = \| g \|_{L^2(\R)}^2 +  \int_{|t| <1}\frac{|g(x+t)-g(x)|^{2}}{|t|^{2\rho+1}}dtdx .
\end{equation} 

We  only prove~\eqref{5.6.1bis}, the proof of~\eqref{5.7.1bis} being similar. We  have 
\begin{multline*} 
F(u)(x)  - F(u) (y)  =\\
= \bigl( u (x) - u(y) \bigr)  \int_{0}^1 \partial_z F\big(u(x) + \theta (u(x)-u(y)), \overline{u}(x)  + \theta (\overline{u(x) - u(y)}) \big)d \theta \\
+ \bigl( \overline{u(x) -u(y)}\bigr)   \int_{0}^1 \partial_{\overline{z}} F\big(u(x)  + \theta (u(x) - u(y))(x), \overline{u}(x) + \theta (\overline{u(x) - u(y) })\big) d \theta.
\end{multline*}
We deduce
\begin{multline}\label{513.2ter}
\big(F(u)(x)  - F(u) (y)\big) -  \big(F(v)(x)  - F(v) (y)\big)=\\
= \bigl( (u-v)  (x) - (u-v) (y) \bigr)  \int_{0}^1 \partial_z F\big(u(x) + \theta (u(x)-u(y)), \overline{u}(x)  + \theta (\overline{u(x) - u(y)})\big) d \theta \\
+ \bigl( (\overline{u-v})(x) -(\overline{u-v})(y)\bigr)   \int_{0}^1 \partial_{\overline{z}} F\big(u(x)  + \theta (u(x) - u(y))(x), \overline{u}(x) + \theta (\overline{u(x) - u(y) }) \big)d \theta\\
+ (v(x) - v(y)) \int_{0}^1 \Bigl( \partial_z F\big(u(x) + \theta (u(x)-u(y)), \overline{u}(x)  + \theta (\overline{u(x) - u(y)})\big)\\
\hfill  - \partial_z F\bigl(v(x) + \theta (v(x)-v(y)), \overline{v}(x)  + \theta (\overline{v(x) - v(y)})\bigr)\Bigr) d \theta \\
+ (\overline{v}(x) -\overline{ v}(y) ) \int_{0}^1 \Bigl( \partial_{\overline{z}} F\big(u(x) + \theta (u(x)-u(y)), \overline{u}(x)  + \theta (\overline{u(x) - u(y)})\big)\\
\hfill  - \partial_{\overline{z}} F\bigl(v(x) + \theta (v(x)-v(y)), \overline{v}(x)  + \theta (\overline{v(x) - v(y)})\bigr)\Bigr) d \theta.
\end{multline}
Recall that, assuming $1< p \leq 2$, we have  $\partial_z F(z, \overline{z}) = \frac{ p+1} 2 |z|^{p-1}$ and $\partial_{\overline{z}} F(z, \overline{z}) = \frac{ p-1} 2 |z|^{p-2}z$ which satisfy (see {\it e.g.}~\cite[(2.26) \& (2.27)]{CaFaHa11})
\begin{equation}\label{elem}
\big| |z_1|^{p-1} - |z_2|^{p-1} \big| \leq C |z_1 - z_2|^{p-1}, \qquad \big| |z_1|^{p-2}z_1- |z_2|^{p-2}z_2 \big| \leq C |z_1 - z_2| ^{p-1}.
\end{equation}
From~\eqref{513.2ter} and~\eqref{elem} we deduce
\begin{multline*}
\bigl| \big(F(u)(x)  - F(u) (y)\big) -  \big(F(v)(x)  - F(v) (y)\big)\bigr|\leq \\
\leq C \bigl| (u-v)  (x) - (u-v) (y)\bigr| \Bigl( |u|(x) +|u|(y) + |v|(x) +|v|(y)\Bigr)^{p-1}\\
+ C \bigl| v(x) - v(y) \bigr| \bigl( |u-v| (x)+ |u-v| (y)\bigr) ^{p-1} .
\end{multline*} 
Plugging this estimate into~\eqref{diffsobo} we get (notice that the roles of $x$ and $y$ are symetric)
\begin{multline*}
\|F(u) - F(v) \|^2_{\mathcal{H}^{\rho}(\R)}\leq  \| \langle x \rangle ^{\rho} (F(u) - F(v)) \|_{L^2(\R)}^2 +\\
+  \int_{|t| <1}\frac{\big|(F(u) - F(v))(x+t)-(F(u) - F(v))(x)\big|^{2}}{|t|^{2\rho+1}}dxdt 
\end{multline*} 
\begin{multline*}
\qquad \leq \| \langle x \rangle ^{\rho} (F(u) - F(v)) \|_{L^2(\R)}^2 + \\
+ C \int_{|t| <1}\frac{|(u-v)(x+t)-(u-v)(x)|^{2}}{|t|^{2\rho+1}}  \bigl( |u| (x)+ |v| (x)+|u| (x)+ |v| (x)\bigr) ^{2(p-1)} dx dt\\
+C \int_{|t| <1}\frac{|v(x+t)-v(x)|^{2}}{|t|^{2\rho+1}} \bigr(  |u-v| (x+t)+ |u-v| (x)\bigr) ^{2(p-1)} dtdx .
\end{multline*} 
The first term, $\| \langle x \rangle ^{\rho} (F(u) - F(v)) \|_{L^2(\R)}$ is easily bounded by 
$$ \| \langle x \rangle ^{\rho} (u-v)\|_{L^4} \bigl(\| u\|_{L^{4( p-1)}}^{p-1}+\| v\|_{L^{4( p-1)}}^{p-1} \bigr)\leq C \|  u-v\|_{\mathcal{W}^{\rho, 4}} \big(\| u\|_{L^{4( p-1)}}^{p-1}+\| v\|_{L^{4( p-1)}}^{p-1}  \big).
$$
 
Using the Cauchy-Schwarz inequality for the $x$ integral and~\eqref{charac}, the contributions of the second and third  term are bounded by 
$$C \| u-v\|_{B^{\rho}_{4,2}  } \big( \| u\|^{p-1}_{L^{4(p-1)} }+ \| v\|^{p-1}_{L^{4(p-1)}}  \big) + C \| v\|_{B^{\rho}_{4,2} } \| u-v\|^{p-1}_{L^{4(p-1)} }. $$
This proves~\eqref{5.6.1bis} for $\sigma=0$ with $B^{\rho}_{p,2}$ instead of $\mathcal{B}^{\rho}_{p,2}$. To conclude (when $\sigma =0$) we apply Lemma~\ref{lem-Besov}. 

To get the result for any $\sigma \geq 0$, we follow the same method as in Section~\ref{sec.5.1} using the partition of unity. Finally, the proof of ~\eqref{5.7.1bis} is similar.
  \end{proof}

%%%%%%%%%%%%%%%%%%%%%%%%%%%%%%%%%%%%%%%%%%%%%%%%%%%%%%%%%%%%%%%
%%%%%%%%%%%%%%%%%%%%%%%%%%%%%%%%%%%%%%%%%%%%%%%%%%%%%%%%%%%%%%%

\section{Functional spaces \texorpdfstring{$Y^{\rho, \epsilon}$} {Y}  and \texorpdfstring{$X^\rho_{t_0, \tau}$}{X}}\label{sec.fonc}
In this section we define the spaces required to develop our Cauchy theory. The rule of the game is the following: the spaces for the initial data must be of full measures, while the spaces for the solutions must be strong enough to perform fixed points, and to ensure that after solving, the final data is controlled in the space of the data. Last but not least, to be able to show the decay properties of the $L^{p+1}$ norms, all  the norms involved must control the $L^\infty; L^{p+1}$ norm (and its variation during the fixed point).

In the sequel we will work on the interval  
$$I_{t_0, \tau}=(t_0- \tau, t_0 +  \tau)\subset (- \frac \pi 4, \frac \pi 4).$$
In the next sections, we define the spaces $Y^{\rho, \epsilon}$ of the initial conditions, and the solution spaces~$X^\rho_{t_0, \tau}$. Before we define precisely these spaces, let us state the main properties they need to satisfy:
\begin{itemize}
\item The space $Y^{\rho, \eps}$ is of full $\mu_0-$measure and large deviation estimates are available (see Proposition~\ref{large-dev}). 
\item The space $Y^{\rho, \eps}$ is invariant by the linear flow: if $u \in Y^{\rho, \eps}$, then for all $t\in\R$ we have $e^{-it H}u\in Y^{\rho, \eps}$ and $\|e^{-itH}u\|_{Y^{\rho, \eps}} = \|u\|_{Y^{\rho, \eps}}$
\item The spaces   $(Y^{\rho, \eps})_{\rho>0, \eps>0}$ are included in each other, with compact embeddings: if $0<\rho<\rho'$ and $0<\eps'<\eps$, then   $Y^{\rho', \eps'} \subset Y^{\rho, \eps}$.
\item We have the continuous embedding $\H^\rho \subset Y^{\rho, \eps}$ (see Lemma~\ref{inclu}).
\item There exists $\delta>0$ such that for $u \in Y^{\rho, \eps}$ we have $\|e^{-itH}u\|_{L^{\infty}((-\pi, \pi);\W^{\delta,p+1})} \leq C \|u\|_{Y^{\rho, \eps}}$. This will be proved in Proposition~\ref{proplarge}.
\item The solution space~$X^\rho_{t_0, \tau}$ is the usual Strichartz space for the Schr\"odinger equation with harmonic potential and data in $\H^\rho$, and including a Besov-norm when $F$ is not regular enough ($1<p\leq 2$).
\end{itemize}

 The definitions of these spaces depend  whether  $1<p\leq \frac 32$ or $\frac32 <p \leq 2$ or $p>2$. This is due to the regularity of the nonlinearity $F$ and is a consequence of the results of Section \ref{Sect5}.

Set 
$$  \rho_0=  \frac 1 2 - \frac 1 {p+1}  =  \frac{p-1}{2(p+1)} \;, $$
then we have the Sobolev embedding $\H^{\rho_0}  \subset L^{ p+1}$ and for all $\rho \geq \rho_0$
\begin{equation}\label{sobo}
  \H^\rho \subset \W^{\rho- \rho_0, p+1}. 
\end{equation}
In the following we will also need the notation
\begin{equation}\label{def-sigma}
\sigma_q = \begin{cases} \frac 1 2 - \frac 1 q  &  \text{ if } \; 2\leq q\leq 4 \\[5pt]
\frac 1 6 + \frac 1 { 3q} & \text{ if } \; q \geq 4 \;,
\end{cases}
\end{equation}
and we refer to Proposition~\ref{proplarge} for a justification of this parameter.

 \subsection{Spaces    for \texorpdfstring{$p > 2 $}{p larger than 2}} \label{sec.6.1}
 Let $\max(0, \frac 1 2 - \frac 1 { 2(2p-3)},  \rho_0) <\rho < \frac 1 2 $.  Denote by $\eta= \min(\frac{\rho-\rho_0}2, \sigma_{p+1})$, where~$\sigma_{p+1}$ is defined in \eqref{def-sigma}. For $0<\eps <\eta$, we define the spaces for the initial data $Y^{\rho, \epsilon}$
\begin{multline*}
Y^{\rho, \epsilon}: =\Big\{u\in {\mathcal H}^{-\varepsilon}: \;\; e^{-itH}u\in L^{8(p-1)}\big((-\pi, \pi);\W ^{\frac{\rho}{2(p-1)}, 4 (p-1)}\big), \\
 e^{-itH}u\in C^0([-  \pi,   \pi]; \W^{ \eta-\eps, p+1}),\;\;  \frac 1 {\langle x \rangle ^{\rho/2}}e^{-itH}u\in   L^{8}\big((- \pi,   \pi); \W^{\rho,4}\big)\Big\}\;,
\end{multline*}
and we  equip them with the natural norm
\begin{multline*}
\|u\|_{Y^{\rho, \epsilon}}=\|u\|_{{\mathcal H}^{-\varepsilon}}+
\|e^{-itH}u\|_{L^{8(p-1)}((-\pi, \pi);\W ^{\frac {\rho} {2(p-1)}, 4 (p-1)})}\\
+\|e^{-itH}u\|_{L^{\infty}((-\pi, \pi);\W^{\eta-\eps,p+1})}
+\| \frac 1 {\langle x \rangle ^{\rho/2}} e^{-itH}u\|_{L^{8}((- \pi,   \pi); \W^{\rho,4})}.
\end{multline*}
We define 
 $X^\rho_{t_0, \tau}$  the spaces for the solutions by
\begin{equation*}
X^\rho_{t_0, \tau}:=  C^0(I_{t_0, \tau}; {\mathcal H}^{\rho}) \cap  L^{4}(I_{t_0, \tau}; \W ^{\rho, \infty}) ,
 \end{equation*}
 equipped with  the natural norm
 \begin{equation*}
\|u\|_{X^\rho_{t_0, \tau}} =  \|u\|_{L^{\infty}(I_{t_0, \tau}; \H^{\rho})}+ \|u\|_{L^{4}(I_{t_0, \tau}; \W^{\rho, \infty})}.
 \end{equation*}
 
 %%%%%%%%%%%%%%%%%%%%%%%%%%%%%%%%%%%%%%%%%%%%%%%%%%%%%%%%%%%

 \subsection{Spaces    for \texorpdfstring{$\frac 3 2<p\leq 2 $}{p between 3/2 and 2}} \label{sec.6.2}
 We can check that for all $\frac 3 2<p\leq 2 $ we have $\max(0,   \frac{4p-7}{4(p-1)},  \rho_0)   < \frac {p-1} 2 $ and we consider       $\max(0,   \frac{4p-7}{4(p-1)},  \rho_0) <\rho < \frac {p-1} 2$. We  denote by $\eta= \min(\frac{\rho-\rho_0}2, \sigma_{p+1})$.  Let  $0<\eps <\eta$, and  define the spaces for the initial data $Y^{\rho, \epsilon}$ by 
\begin{multline*}
Y^{\rho, \epsilon}: =\Big\{u\in {\mathcal H}^{-\varepsilon}: \;\; e^{-itH}u\in L^{8(p-1)}\big((-\pi, \pi);\W ^{\frac{2p-3}{4(p-1)}-\eps, 4 (p-1)}\big), \\
 e^{-itH}u\in C^0([-  \pi,   \pi]; \W^{ \eta-\eps, p+1}),\;\;  \frac 1 {\langle x \rangle ^{(2p-3)/4}}e^{-itH}u\in   L^{8}\big((- \pi,   \pi); \W^{\rho,4}\big)\Big\}\;,
\end{multline*}
and the spaces $X^\rho_{t_0, \tau}$   by
\begin{equation*}
X^\rho_{t_0, \tau}:= C^0(I_{t_0, \tau}; {\mathcal H}^{\rho}) \cap  L^{4}(I_{t_0, \tau}; \W ^{\rho, \infty}) \cap L^8(I_{t_0, \tau}; \mathcal{B}^{\rho} _{4,2}).
 \end{equation*}
 All spaces are equipped with their natural norms (recall that the Besov spaces are defined in~\eqref{def-besov}).
 %%%%
 
  \subsection{Spaces  for  \texorpdfstring{$1< p \leq \frac 3 2$}{p between 1 and three half}}\label{sec.6.3}
Let $ \rho_0 <\rho <\frac {p-1} 2$.  Similarly to the previous cases, denote by $\eta= \min(\frac{\rho-\rho_0}2, \sigma_{p+1})$. Then for $\kappa>0$ small enough, we consider the  Strichartz-admissible couple  $(q,r)= (\frac{4}{p-1-\kappa}, \frac{ 2} {2-p + \kappa})$. The spaces for the initial data are given by
$$
Y^{\rho, \epsilon}:=\big\{u\in {\mathcal H}^{-\varepsilon}\,:\, e^{-itH}u\in L^\infty(( -  \pi,   \pi); \mathcal{H}^{- \epsilon})\cap L^{q}((-   \pi,  \pi ); \W ^{\rho, r}) \cap C^0( [-  \pi,  \pi]; \W^{\eta-\epsilon, p+1})\big\},
$$
 where $0<\eps< \eta$, and the spaces for the solutions~$X^\rho_{t_0, \tau}$ are given by
$$X^\rho_{t_0, \tau}:= C^0(I_{t_0, \tau}; {\mathcal H}^{\rho}) \cap L^4(I_{t_0, \tau}; \W^{\rho, \infty} ) \cap L^{q}(I_{t_0, \tau}; \mathcal{B} ^{\rho}_{r,2}).
$$
 All spaces are equipped with their natural norms. 
 
Define  $s = \frac{ 2( p-1) r} { r-2}= \frac{2(p-1)}{p-1-\kappa}$.  We shall choose $\rho < \frac{ p-1} 2$ arbitrarily close to $\frac{ p-1} 2$, then chose $\kappa>0$ small enough so that $2<s< r$, and then for $\epsilon >0$ small enough, since $q> 8$, we have 
\begin{equation}\label{inter714}
 \| e^{-itH} u\|_{L^8 ((-\pi, \pi); \W^{\rho, r} )} \leq \| e^{-itH} u\|_{L^q ((-\pi, \pi); \W^{\rho, r} )} \leq \| u\|_{Y^{\rho, \epsilon}}.
 \end{equation}
Observe also that by Sobolev embedding, the previous line implies that there exists $C>0$ such that 
\begin{equation}\label{inter715}
  \|e^{-itH} u\|_{L^8 ((-\pi, \pi); L^s)}\leq C \| u\|_{Y^{\rho, \epsilon}}.
\end{equation}
 %%%%
 
\subsection{The space \texorpdfstring{$X_{t_0,\tau}^0$} {X t-0 tau}} In  the sequel, we will also need the space 
$$X_{t_0,\tau}^0=L^{\infty}(I_{t_0, \tau}; L^2) \cap  L^{4}(I_{t_0, \tau}; L^\infty).$$

\subsection{Further properties of \texorpdfstring{$Y^{\rho, \epsilon}$}{Y rho}}

As a consequence of the Strichartz inequalities, we have

\begin{lem}\label{inclu}
Let $p>1$.  Under the above assumptions on $ \rho$, we have for all $\epsilon >0$ small enough
\begin{equation}\label{incl}
  \H ^\rho \subset Y^{\rho, \epsilon}.
  \end{equation}
  Namely, there exists $C_0>0$ such that for all $u \in Y^{\rho, \epsilon}$
  \begin{equation*} 
 \| u\|_{Y^{\rho, \eps}} \leq C_0 \| u\|_{\H^\rho}.
  \end{equation*}
As a consequence, 
\begin{equation}\label{incl2}
X^\rho_{t_0, \tau} \subset L^{\infty}(I_{t_0,\tau}; Y^{\rho, \epsilon}),
\end{equation}
and for all $v \in X^\rho_{t_0, \tau}$
\begin{equation}\label{incl3}
 \| v\|_{L^{\infty}(I_{t_0,\tau}; Y^{\rho, \epsilon})} \leq C_0 \| v\|_{X^\rho_{t_0,\tau}}.
\end{equation}
\end{lem}

\begin{proof} 

$\bullet$ Case $ p > 2$.  The bound 
$$\|e^{-itH}u\|_{L^{\infty}((-\pi, \pi);\W^{\eta-\eps,p+1})} \leq C  \| u\|_{\H^\rho}$$
is a direct consequence of \eqref{sobo}. 
On the other hand, the couple $(8,4)$ is admissible which implies 
\begin{equation*}
\| \frac 1 {\langle x \rangle ^{\rho/2}} e^{-itH}u\|_{L^{8}((-  \pi,  \pi); \W^{\rho,4})}\leq C \|  e^{-itH}u\|_{L^{8}((- \pi,  \pi); \W^{\rho,4})}\leq C \| u\|_{\H^\rho}.
\end{equation*}
To bound the last term, we use the  Strichartz inequality \eqref{stri0} and get 
\begin{equation}\label{lastB}
 \| e^{-itH}u\|_{ L^{8(p-1)}((-\pi, \pi);\W ^{\rho,r}) }\leq C \|u\|_{\H^\rho}, \qquad \frac 1 r = \frac 1 2 - \frac 1{4(p-1)}.
\end{equation}
It remains to check that by Sobolev embeddings, we have 
\begin{equation*}
\W ^{\rho,r}  \subset \W ^{\frac {\rho} {2(p-1)}, 4 (p-1)}, 
\end{equation*}
which follows from 
\begin{equation}\label{sob}
\rho (1- \frac {1} {2(p-1)} )\geq \frac 1 r - \frac 1 {4(p-1)}  =  \frac 1 2 - \frac{1} {2(p-1)} \quad \Leftarrow \quad  \rho\geq  \frac 1 2 - \frac 1 {2(2p-3)}=\frac{p-2}{2p-3}.
\end{equation}~

$\bullet$ Case $ \frac32 < p \leq 2$.  The proof follows the same lines. To bound the last term, it is enough to check here the Sobolev embedding
\begin{equation}\label{sob3}
\W ^{\rho,r}  \subset \W ^{\frac{2p-3}{4(p-1)}, 4 (p-1)}
\end{equation}
which holds true under the condition
\begin{equation*} 
\rho -\frac{2p-3}{4(p-1)}  \geq \frac 1 r - \frac 1 {4(p-1)}  =  \frac 1 2 - \frac{1} {2(p-1)} \quad \Leftarrow \quad  \rho>  \frac{4p-7}{4(p-1)}.
\end{equation*}
 We conclude with \eqref{lastB}.

$\bullet$ Case $1< p \leq \frac 3 2$. Recall that $(q,r)= (\frac{4}{p-1-\kappa}, \frac{ 2} {2-p + \kappa})$ is an admissible couple.    From the Sobolev embedding \eqref{sobo} we obtain 
$$ \|e^{-itH}u\|_{L^{\infty}((-  \pi,  \pi ); \W^{\eta-\eps,p+1}) } \leq \|e^{-itH}u\|_{L^{\infty}((-  \pi,  \pi ); \W^{\rho- \rho_0,p+1}) } \leq C  \| u\|_{\H^\rho}.
$$
The bound 
\begin{equation*} 
 \| e^{-itH} u\|_{L^q ((-\pi, \pi); \W^{\rho, r} )}  \leq C  \| u\|_{\H^\rho}, 
 \end{equation*}
 is given by the Strichartz estimate \eqref{stri0}.

Finally, the embedding \eqref{incl2} follows from the fact that $X^\rho_{t_0, \tau}   \subset L^{\infty}(I_{t_0,\tau}; \H^\rho)$ and~\eqref{incl}.
\end{proof}

As a consequence of \eqref{prop.cont}, we can show 

\begin{lem}\label{Stone}
Assume that $\eps, \rho >0$ and $\sigma> 0$. There exists $C>0$ such that for any $N\geq 1$, and any $u \in Y^{\rho, \epsilon}$ and any $v\in X^\s_{t_0, \tau}$,
$$ \| S_N u \| _{Y^{\rho, \epsilon}} \leq C \| u\|_{Y^{\rho, \epsilon}}, \qquad \| S_N v\|_{X^\s_{t_0, \tau}} \leq C \| v\|_{X^\s_{t_0, \tau} }.
$$
\end{lem}

We also have the following statement.

 \begin{lem}\label{lem-borne}
 Let  $\rho>0$ and $ \epsilon>0 $.   Then for any    $u\in Y^{\rho, \epsilon}$  and $N \geq 1$,
$$\|(1-S_N)u\|_{Y^{\rho, \epsilon}} =  o(1)_{N\rightarrow +\infty}.$$
Assume moreover that $0<\rho<\rho'$ and $0< \epsilon'< \epsilon $.  Then for any $u\in Y^{\rho', \epsilon'}$ and $N \geq 1$,
\begin{equation}\label{saving}
\|(1-S_N)u\|_{Y^{\rho, \epsilon}}\leq C N^{- \min(\rho'-\rho, \epsilon-\epsilon')}\|u\|_{Y^{\rho', \epsilon'}} .
\end{equation}
\end{lem}
One can easily see that the analysis above 
implies that $S_N u$ 
is a Cauchy sequence in $L^2(\Omega;Y^{\rho, \epsilon})$ and thus we may see the measures $\mu_0$ and $\nu_t$ 
as finite Borel measures on $Y^{\rho, \epsilon}$.  

The property \eqref{saving} will be used to obtain uniform bounds for the approximate flow (see Lemma~\ref{lem.limite.bis}), which is a key ingredient in the  proof of the quasi-invariance result (Proposition~\ref{lemeqlim-bis}).   Notice in particular that \eqref{saving} implies that,  for  $0<\rho<\rho'$ and $0< \epsilon'< \epsilon $,  the embedding $Y^{\rho', \epsilon'}\subset Y^{\rho, \epsilon}$ is compact.

\begin{proof}
Notice that the assumptions $\rho<\rho'$ and $\epsilon ' < \epsilon$ gain some positive power in $H$ (hence positive powers of $N \geq 1$), and consequently compactness in the space variable because powers of $H$ control both powers of $D_x$ and of~$x$.  On the other hand, since the second and the third term in the definition of $Y^{\rho, \epsilon}$ are defined in terms of the free evolution, we may exchange some saving derivatives  in $H$
for some  time derivatives and hence some compactness in time. We omit the details.
\end{proof}

%%%%%%%%%%%%%%%%%%%%%%%%%%%%%%%%%%%%%%%%%%%%%%%%%%%%%%%%%%%%%%%
%%%%%%%%%%%%%%%%%%%%%%%%%%%%%%%%%%%%%%%%%%%%%%%%%%%%%%%%%%%%%%%

\section{Large deviation bounds} \label{sec.7}
We start by recalling the following  large deviation bound, which is a  variation around  results obtained in \cite{Th09,BTT}, leading to an improvement in the time variable.
\begin{prop}\label{proplarge}
Assume that $(\sigma, q)$ satisfy $\sigma < \sigma_q $, where 
\begin{equation*}
\sigma_q = \begin{cases} \frac 1 2 - \frac 1 q  &  \text{ if } \; 2\leq q\leq 4 \\[5pt]
\frac 1 6 + \frac 1 { 3q} & \text{ if } \; q \geq 4.
\end{cases}
\end{equation*}
 Then 
\begin{equation}\label{probab1}
 \mu_0\big(\big\{u_0 \in X^0(\R) :   \|  \e^{-it H}u_0\|_{L^\infty{((-\pi, \pi)}; \W^{\sigma,q})}  \geq R\big\}\big)\leq C \e^{-cR^2}.
 \end{equation}
  \end{prop}
 \begin{proof} 
From~\cite[Theorem 2.1, $ (\infty,2,1)$]{Th09} and \cite[Lemma A.8]{BTT}, 
\begin{equation}\label{borne}
\| e_n \|_{L^2} =1, \qquad  \|  e_n \|_{L^4} \leq C \log^{\frac 1 4}( \lambda_n)\lambda_n^{- \frac 14}, \qquad  \|  e_n \|_{L^\infty} \leq C\lambda_n^{- \frac 1 6}.
\end{equation}
We deduce by H\"older inequality 
\begin{equation}\label{bornebis}
\forall \sigma <\sigma_q ,\quad  \exists \epsilon = \frac{ \sigma_q - \sigma} 2; \quad \| H^{\frac \sigma 2} e_n \|_{L^q} =\| \lambda_n ^{ \sigma } e_n \|_{L^q}\leq C \lambda_n^{- \epsilon}.
\end{equation}
We now revisit the proof of~\cite[Proposition 4.4]{BT08}, see also~\cite[Section 3]{BT08-1} and~\cite[Proposition 6.2]{Th09}, with a slightly different treatment of the time variable. To begin with, observe that 
$$D_t^\alpha e^{-itH}u_0= H^{\alpha} e^{-itH}u_0,$$
as can be checked by a decomposition in Hermite series together with a Fourier transform in time. 

Next, let $s\geq r> \frac 4 \epsilon >q$, let  $\chi\in C^\infty(\mathbb{R})$ equal to $1$ on $[- \frac \pi 2, \frac \pi 2]$. By the Minkowski inequality we have  
\begin{multline*}
\big\| \chi(t) \langle D_t\rangle^{\frac \epsilon 4} H^{\frac \sigma 2} e^{-itH} \Bigl( \sum_{n\geq 0} \frac{ g_n} { \lambda_n} e_n (x)\Bigr) \big\|_{L^s(\Omega);  L^r(\mathbb{R}_t); L^q( \mathbb{R}_x) }\leq \\
\begin{aligned}
&\leq \big\| \sum_{n\geq 0}  \langle \lambda_n \rangle^{\frac \epsilon 2} \lambda_n^{  \sigma  -1 } \chi(t) e^{-it \lambda_n } e_n (x){ g_n} \big\|_{ L^q( \mathbb{R}_x); L^r(\mathbb{R}_t) ; L^s(\Omega)}\\
&\leq C \sqrt{s} \big\| \Bigl(\sum_{n\geq 0} \langle \lambda_n \rangle ^\epsilon \lambda_n^{ 2\sigma  -2 } |\chi(t)|^2 |e_n|^2(x) \Bigr)^{1/2} \big\|_{ L^q( \mathbb{R}_x); L^r(\mathbb{R}_t) }\\
&\leq C \sqrt{s} \Bigl(\sum_{n\geq 0}  \langle\lambda_n\rangle ^{ \epsilon+ 2\sigma  -2 } \|\chi(t) e_n(x)\| _{L^q( \mathbb{R}_x); L^r(\mathbb{R}_t) }^2\Bigr)^{1/2}\\
&\leq C \sqrt{s} \sum_{n\geq 0}  \langle \lambda_n\rangle ^{ \epsilon+ 2\sigma  -2 } \lambda_n^{- 2\sigma - 2\epsilon} <C'\sqrt{s} 
\end{aligned}
\end{multline*}
where we used that for i.i.d normalised Gaussian random variables 
$$ \|\sum_{n\geq 0} \beta_n g_n \|_{L^s(\Omega)}\leq C \sqrt{s}(\sum_{n\geq 0} |\beta_n|^2)^{1/2}, 
$$ 
and in the last line we used~\eqref{bornebis} and the fact that $ \lambda_n = \sqrt{2n+1}$.

The Bienaym\'e-Tchebychev inequality and an optimisation with respect to the parameter $s\geq 1$ (see {\it e.g.}~\cite[(4.5) \& (4.6)]{BT08}) gives 
\begin{equation}\label{LD}
 \mu_0\big(\{u_0 \in X^0(\R) :   \|  \chi(t) \langle D_t\rangle^{\frac \epsilon 4} H^{\frac \sigma 2} \e^{-it H}u_0\|_{L^r(\mathbb{R}_t) ; L^q( \mathbb{R}_x)}  \geq R\big\}\big)\leq C \e^{-cR^2}.
 \end{equation}
Finally, to get~\eqref{probab1}, we just remark using Sobolev embedding in the time variable (recall that $\frac \epsilon 4 > \frac 1 r$)
\begin{multline*}
\|  \e^{-it H}u_0\|_{L^\infty({(-\pi, \pi)}; \W^{\sigma,q})} = \|  H^{\frac \sigma 2} \e^{-it H}u_0\|_{L^\infty{((-\pi, \pi)}; L^{q}_x)}\leq  \\
  \leq \|  H^{\frac \rho 4} \e^{-it H}u_0\|_{L^{q}_x; L^\infty{(-\pi, \pi)} } \leq \| \chi(t) D_t^{\frac \epsilon 4} H^{\frac \rho 4} \e^{-it H}u_0\|_{L^{q}_x; L^r{(\R_t)} } \;,
\end{multline*}
which together with \eqref{LD} yields the result.
\end{proof}

We shall also need the following result
\begin{prop}\label{propnew}
Assume that $0< \gamma< \frac 1 4$. 
 Then for any $\rho< \frac 1 4 + \gamma$, there exist $c,C>0$ such that 
\begin{equation*}
 \mu_0\big(\big\{u_0 \in X^0(\R) :   \big\|  \frac{ 1} {|x|^\gamma} \e^{-it H}u_0\big\|_{L^\infty({(-\pi, \pi)}; \W^{\rho,4})}  \geq R\big\}\big)\leq C \e^{-cR^2}.
 \end{equation*}
  \end{prop}
 \begin{proof} 
 The proof follows the same lines as the proof of Proposition~\ref{proplarge} after replacing the bound~\eqref{borne} by the bound~\eqref{borne.F.1} in Appendix~\ref{app.E}.
 \end{proof} 
We can now proceed to show that $\mu_0-$almost every function is in $Y^{\rho, \epsilon}$.
 \begin{prop} \label{large-dev}
Assume that $\rho, \epsilon>0$ satisfy the assumptions in Sections~\ref{sec.6.1},~\ref{sec.6.2},~\ref{sec.6.3}. Then 
$$ \exists \, c,C>0;\quad  \mu_0 \big( \big\{  u \in X^0(\R) :  \;\; \| u \|_{Y^{\rho, \epsilon}} > R \} \big) \leq C e^{-cR^2}.
$$
\end{prop}

\begin{proof} $\bullet$ Case $p>2$.
Recall that
\begin{multline*}
Y^{\rho, \epsilon} =\Big\{u\in {\mathcal H}^{-\varepsilon}: \;\; e^{-itH}u\in L^{8(p-1)}\big((-\pi, \pi);\W ^{\frac{\rho}{2(p-1)}, 4 (p-1)}\big), \\
 e^{-itH}u\in C^0([-  \pi,   \pi]; \W^{ \eta-\eps, p+1}),\;\;  \frac 1 {\langle x \rangle ^{\rho/2}}e^{-itH}u\in   L^{8}\big((- \pi,   \pi); \W^{\rho,4}\big)\Big\}\;.
\end{multline*}
Since  $\rho <\frac 1 2$, we have $\rho < \frac 1 4 + \frac \rho 2$. Therefore, we can apply  Proposition~\ref{propnew} with $\gamma=\frac{\rho}2<\frac14$, which allows to control the last term in the definition of $Y^{\rho, \epsilon}$.  To control the other terms,   it is enough to check 

\begin{gather} \label{cond1}
\frac {\rho} {2(p-1)} < \sigma_{4 (p-1)}= \frac{p+1}{6(p-1)} \\
\eta-\eps< \sigma_{p+1}.\label{cond2}
\end{gather}
The inequality \eqref{cond2} holds true because we made the  choice  $\eta= \min(\frac{\rho-\rho_0}2, \sigma_{p+1})>0$ and $\epsilon>0$ is small enough. Condition~\eqref{cond1} follows from the fact that $p>2$ and $\rho<\frac12$.

$\bullet$ Case $\frac 3 2<p \leq 2$.
Recall that
\begin{multline*}
Y^{\rho, \epsilon} =\Big\{u\in {\mathcal H}^{-\varepsilon}: \;\; e^{-itH}u\in L^{8(p-1)}\big((-\pi, \pi);\W ^{\frac{2p-3}{4(p-1)}-\eps, 4 (p-1)}\big), \\
 e^{-itH}u\in C^0([-  \pi,   \pi]; \W^{ \eta-\eps, p+1}),\;\;  \frac 1 {\langle x \rangle ^{(2p-3)/4}}e^{-itH}u\in   L^{8}\big((- \pi,   \pi); \W^{\rho,4}\big)\Big\}\;.
\end{multline*}
In this case we apply  Proposition~\ref{propnew} with $\gamma=\frac{2p-3}{4}<\frac14$ and $\rho <\frac14+ \frac{2p-3}{4}=\frac{p-1}2$. Then, we have to check the conditions
\begin{gather*} 
\frac{2p-3}{4(p-1)}-\eps < \sigma_{4 (p-1)}= \frac{2p-3}{4(p-1)}\\
\eta-\eps< \sigma_{p+1} \;, 
\end{gather*}
which both hold true.

$\bullet$ Case {$1<p\leq \frac 3 2$}.
Recall that 
$$
Y^{\rho, \epsilon}=\big\{u\in {\mathcal H}^{-\varepsilon}\,:\, e^{-itH}u\in L^\infty(( -  \pi,   \pi); \mathcal{H}^{- \epsilon})\cap L^{q}((-   \pi,  \pi ); \W ^{\rho, r}) \cap C^0( [-  \pi,  \pi]; \W^{\eta-\epsilon, p+1})\big\},
$$
with  $ 0<\rho <\frac {p-1} 2$, and  where $(q,r)= (\frac{4}{p-1-\kappa}, \frac{ 2} {2-p + \kappa})$ is a Strichartz-admissible couple. Here, the conditions to check are 
\begin{gather*} 
\rho< \sigma_{r}= \frac 1 2 - \frac 1 r = \frac{ p-1- \kappa} 2\\
\eta-\eps< \sigma_{p+1} \;, 
\end{gather*}
and there are both satisfied if $\eps, \kappa>0$ are small enough.
\end{proof}

\begin{rem} \label{uniform-large}As a consequence of the proof (namely the proof of Proposition~\ref{proplarge}), all the results in this section remain true if we replace $\H^{- \epsilon}$ by $E_N$ and $\mu_0$ by $\mu_N$, with the same constants (hence uniform with respect to $N\geq 1$).
\end{rem}

%%%%%%%%%%%%%%%%%%%%%%%%%%%%%%%%%%%%%%%%%%%%%%%%%%%%%%%%%%%%%%%
%%%%%%%%%%%%%%%%%%%%%%%%%%%%%%%%%%%%%%%%%%%%%%%%%%%%%%%%%%%%%%%

\section{The local  Cauchy theory}\label{sec.8}

We consider the equation

 \begin{equation}  \label{eq_u}
  \left\{
      \begin{aligned}
         &i \partial _t u-Hu = \cos^{\frac {p-5}2}(2t)  |u|^{p-1} u,\quad (t,x) \in  (-\frac{\pi}{4},\frac{\pi}{4})\times \R,
       \\  &  u_{|t= t_0}=u_0.
      \end{aligned}
    \right.
\end{equation}
Recall that $F(u)=|u|^{p-1}u$, then~\eqref{eq_u} admits the Duhamel formula
$$u = e^{-i(t-t_0) H} u_0 - i \int_{t_0}^t \cos^{\frac{p-5}{2}}(2s)  e^{-i(t-s)H}F\big(e^{-i(s-t_0)H} u_0+v(s)\big)ds,
$$ and consequently setting $u = e^{-i(t-t_0) H} u_0 +v$, the function $v$ must satisfy
$ v = K(v)$ with 
\begin{equation}\label{defK}
K(v):=-i\int_{t_0}^t \cos^{\frac{p-5}{2}}(2s) e^{-i(t-s)H}F\big(e^{-i(s-t_0)H} u_0+v(s)\big)ds.
\end{equation} 
For  $t_0,t\in (-\frac{\pi}{4},\frac{\pi}{4})$, thanks to a fixed point argument, we will prove that equation~\eqref{eq_u} admits, on the interval $I_{t_0, \tau}= (t_0 - \tau, t_0 + \tau)$, a   solution of the form $u=e^{-i(t-t_0)H}u_0+v$, where   $v \in X^{\rho}_{t_0, \tau}$. 

The main result we shall need to prove for   the a.s. global existence   is the following.
\begin{prop}\label{proplocal2}  
Let $p>1$ and $t_0\in (-\frac{\pi}{4},\frac{\pi}{4})$.  Let 
\begin{equation*}
\rho< \begin{cases} \frac{ p-1} 2 &\text{ if }\; 1<p \leq   2 \\[5pt]
 \frac 1 2 &\text{ if } \; p\geq   2   
\end{cases}
\end{equation*}
chosen sufficiently close to $\frac{ p-1} 2$ or $\frac 1 2$   respectively so that the assumptions in Section~\ref{sec.fonc} are satisfied.  Let  $\eps>0$ chosen small enough in the definition of the space $Y^{\rho, \epsilon}$. There exist  $c>0$ and $\kappa, \delta \geq 1$  such that for any $R> 0$, setting 
\begin{equation}\label{prop-tau}
  \tau\leq \begin{cases}  c R^{-\kappa}(\frac{\pi}4-|t_0|)^{\delta}  &\text{ if } 1< p < 5\\
c R^{-\kappa}& \text{ if } p \geq 5
\end{cases}
\end{equation}
 for any $u_0\in Y^{\rho, \epsilon}$ such that  $\| u_0 \|_{Y^{\rho, \epsilon}} \leq R$, there exists a  unique solution   
 $u= e^{-itH} u_0 + v$, with $v\in  X^\rho_{t_0, \tau}$, 
  to the equation\;\eqref{eq_u} on the interval 
 $$I_{t_0, \tau} = (t_0 - \tau, t_0 + \tau),$$
  which  satisfies
 \begin{equation*}
  \| v\big\|_{X^\rho_{t_0, \tau}}\leq  \frac{1}{C_0}, \qquad    \| v\|_{L^{\infty}(I_{t_0,\tau}; Y^{\rho, \epsilon})} \leq 1,
\end{equation*}
where $C_0>0$ is the constant in \eqref{incl3}.

Furthermore, for two such initial data $u_0, \widetilde{u}_0\in Y^{\rho, \epsilon}$ such that $ \| u_0\|_{Y^{\rho, \epsilon}}, \| \widetilde{u}_0 \|_{Y^{\rho, \epsilon}} \leq R$, we have 
\begin{equation}\label{continuity.1}
\begin{cases}
\| u- \widetilde{u}\| _{X^{0}_{t_0, \tau}}  \leq C \| u_0- \widetilde{u}_0 \|_{Y^{\rho, \epsilon}} & \text{ if $1<p\leq 2$}\\[5pt]
\| u- \widetilde{u}\| _{X^\rho_{t_0, \tau}}  \leq C \| u_0- \widetilde{u}_0 \|^\theta_{Y^{\rho, \epsilon}}& \text{ for some $\theta >0$  if $1<p\leq 2$} \\[5pt]
 \| u- \widetilde{u}\| _{X^\rho_{t_0, \tau}} \leq C \| u_0- \widetilde{u}_0 \|_{Y^{\rho, \epsilon}} & \text{ if $p>2$}
\end{cases}
\end{equation}
and for any $t \in (t_0-\tau, t_0 + \tau)$, 
\begin{equation}\label{continuity.2}
\begin{cases} 
\| u(t)- \widetilde{u}(t)\| _{Y^{\rho, \epsilon}} \leq C \| u_0- \widetilde{u}_0 \|^\theta_{Y^{\rho, \epsilon}}& \text{ for some $\theta >0$  if $1<p\leq 2$}\\[5pt]
\| u(t)- \widetilde{u}(t)\| _{Y^{\rho, \epsilon}} \leq C \| u_0- \widetilde{u}_0 \|_{Y^{\rho, \epsilon}} & \text{ if $p>2$}.
\end{cases}
\end{equation}
In addition, after possibly taking smaller $c>0$, and larger $\delta, \kappa\geq 1$,  the solution  satisfies \begin{equation}\label{Lp+1}
\sup_{t \in (t_0 - \tau, t_0 + \tau)} \| u(t) - u_0\|_{L^{p+1}} \leq 1.
\end{equation}

Finally, let  $\rho'>\rho$,  $\epsilon'<\epsilon$ and assume in addition that $u_0\in Y^{\rho', \epsilon'}$,  then there exists $M>0$ such that
\begin{equation}\label{persistance}
 \| u \|_{L^\infty ((t_0 - \tau, t_0 + \tau); Y^{\rho', \epsilon'})} \leq M\| u_0\|_{Y^{\rho', \epsilon'}}\,,\quad  u = e^{-i(t-t_0)H} u_0 + v\,,  \quad \| v\|_{X^{\rho'}_{t_0, \tau} } \leq \| u_0\|_{Y^{\rho', \epsilon'}}.
 \end{equation}
 \end{prop}
 \begin{rem}\label{rem8.5}
Proposition~\ref{proplocal2},   holds for the equation 
 \begin{equation*}  
  \left\{
      \begin{aligned}
         &i\partial_t u-Hu=\cos^{\frac{p-5}{2}}(2t)S_N\big(|S_Nu|^{p-1}S_Nu\big),\quad (t,x) \in  (-\frac{\pi}{4},\frac{\pi}{4})\times \R,
       \\  &  u_{|t= t_0}=u_0 \in E_N,
      \end{aligned}
    \right.
\end{equation*}
with uniform estimates with respect to the parameter $N \geq 1$.
 This is a direct consequence of the proof of Proposition~\ref{proplocal2} and    the boundedness of $S_N$ on $L^q(\R)$ spaces for $1\leq q\leq +\infty$ (see \eqref{prop.cont} and Lemma~\ref{Stone}).
\end{rem}

  \begin{proof}[Proof of Proposition \ref{proplocal2}]
 Consider the constant $C_0>0$ in \eqref{incl3}. We will show that  the operator $K$, defined in \eqref{defK}, has a unique fixed point  $v$ in the closed ball of radius $1/C_0$ centered around $u_0^f:= e^{-i(t-t_0)H} u_0$ in~$X^\rho_{t_0, \tau}$. 

\underline{Step 1: the operator $K$ maps a ball of $X^\rho_{t_0, \tau}$ into itself.} We have to distinguish several cases with respect to $p>1$. 

$\bullet$  Case  $p>2$. From~\eqref{5.3.0} applied with $\sigma=\rho/2$,  we get
\begin{equation*} 
 \big\| F\big(u_0^f+ v\big)\big\|_{{\mathcal H}^{\rho}}\leq  
C \big(  \|\frac 1 {\langle x \rangle ^{\rho/2} }u_0^f\|_{\W^{\rho, 4}} + \| \frac 1 {\langle x \rangle ^{\rho/2} }v\|_{\W^{\rho, 4}} \bigr) \Big(\|  \langle x \rangle ^{\frac{\rho} {2( p-1)}} u_0^f\|_{L^{4(p-1)}}^{p-1} + \| \langle x \rangle ^{\frac{\rho} {2( p-1)}} v\|_{L^{4(p-1)}}^{p-1} \Big).
\end{equation*}
Taking the $L^1_t$ norm and using H\"older inequality gives,
\begin{eqnarray}\label{borne3.2ter}
 \big\| F\big(u_0^f+ v\big)\big\|_{L^1(I_{t_0,\tau};{\mathcal H}^{\rho})} 
& \leq& C \tau ^{\frac12- \frac 1 8} \big(  \|\frac 1 {\langle x \rangle ^{\rho/2} } u_0^f\|_{L^8(I_{t_0,\tau} ;\W^{\rho, 4})} + \| \frac 1 {\langle x \rangle ^{\rho/2} }v\|_{L^8(I_{t_0,\tau}; \W^{\rho, 4})} \bigr)  \nonumber\\
&&  \bigl(\| \langle x \rangle ^{\frac{\rho} { 2(p-1)}}  u_0^f\|_{L^{8(p-1)}(I_{t_0,\tau}; L^{4(p-1)})} ^{p-1} + \|  \langle x \rangle ^{\frac{\rho} { 2(p-1)}}  v\|_{L^{8(p-1)}(I_{t_0,\tau}; L^{4(p-1)})}^{p-1} \bigr) \nonumber \\
 & \leq & C \tau ^{\frac 3 8} \big(  \| u_0\|_{Y^{\rho, \epsilon}} + \| v\|_{X^\rho_{t_0, \tau} }\big)^p .
\end{eqnarray}

\quad  -- Subcase $2 <p < 5$. First, notice that for all $|t|< \frac{\pi}4$
\begin{equation*}
0<\cos^{\frac{ p-5} 2} (2t)\leq  C \big(\frac \pi 4 - |t|\big)^{\frac {p-5} 2}.
  \end{equation*}
Therefore, since $\tau \leq \frac12({\frac \pi 4 - |t_0| })$,    for all $t\in I_{t_0,\tau}$, we have $\frac{\pi}4-|t|\geq\frac{\pi}4-  |t_0|-\tau \geq \frac12({\frac \pi 4 - |t_0| })$, hence 
$$ \forall t \in I_{t_0,\tau}, \quad  \cos^{\frac { p-5} 2}( 2t) \leq C \big(\frac \pi 4 - |t_0|\big)^{\frac {p-5} 2} ,$$
 and together with \eqref{borne3.2ter} we deduce
\begin{equation*}
 \big\|\cos^{\frac{p-5}{2}}(2s)F(u_0^f+v)\big\|_{L^1(I_{t_0,\tau}; \H^\rho)} 
 \leq C \tau ^{\frac 38}  \big(\frac \pi 4 - |t_0|\big) ^{\frac {p-5} 2}\big(  \| u_0\|_{Y^{\rho, \epsilon}} + \| v\|_{X^\rho_{t_0, \tau} }\big)^p   .
 \end{equation*}
Thus  we get, using the inhomogeneous Strichartz estimates \eqref{stri1}
 \begin{equation}\label{K1}
 \| K(v)  \| _{X^\rho_{t_0, \tau}}\leq C \tau ^{\frac 3 8}  \big(\frac \pi 4 - |t_0|\big) ^{\frac {p-5} 2}\big(  \| u_0\|_{Y^{\rho, \epsilon}} + \| v\|_{X^\rho_{t_0, \tau} }\big)^p .
 \end{equation}
As a consequence,  in~\eqref{prop-tau}, taking $\kappa, \delta >1$ large enough and $c>0$ small enough shows that the operator~$K$ maps the unit ball of $X^\rho_{t_0, \tau}$ into itself.
 
\quad  -- Subcase $p\geq  5$.  In this case, we simply  use that $0<\cos^{\frac{p-5} 2}(2t) \leq 1$ and get 
 \begin{equation}\label{K2}
 \| K(v)  \| _{X^\rho_{t_0, \tau}}\leq C \tau ^{\frac 38} \big(  \| u_0\|_{Y^{\rho, \epsilon}} + \| v\|_{X^\rho_{t_0, \tau} }\big) ^p
 \end{equation}
 and conclude similarly.

 $\bullet$  Case  $\frac32<p \leq 2$. We can proceed as in the previous case, but here we apply~\eqref{5.3.0}  with $\sigma=\frac{2p-3}4$. In particular we get the estimate
 \begin{equation*} 
 \big\|\cos^{\frac{p-5}{2}}(2s)F(u_0^f+v)\big\|_{L^1(I_{t_0,\tau}; \H^\rho)} 
 \leq C \tau ^{\frac 38}  \big(\frac \pi 4 - |t_0|\big) ^{\frac {p-5} 2}\big(  \| u_0\|_{Y^{\rho, \epsilon}} + \| v\|_{X^\rho_{t_0, \tau} }\big)^p   .
 \end{equation*}
 With  the inhomogeneous Strichartz estimates \eqref{stri1} we deduce
  \begin{equation}\label{K3}
 \| K(v)  \| _{X^\rho_{t_0, \tau}}\leq C \tau^{\frac{ 3} 8}  \big(\frac \pi 4 - |t_0|\big) ^{\frac {p-5} 2}\big(  \| u_0\|_{Y^{\rho, \epsilon}} + \| v\|_{X^\rho_{t_0, \tau} }\big)^p .
 \end{equation}
Hence we get, as before,  taking in~\eqref{prop-tau} the parameters $\kappa, \delta>1$ large enough and $c>0$ small enough,   that the operator~$K$ maps the ball of radius $1/C_0$ of $X^\rho_{t_0, \tau}$ into itself.

$\bullet$ Case $1<p \leq \frac 3 2$. We now use ~\eqref{5.4.1},  and get
\begin{equation*} 
 \big\| F(u_0^f+ v)\big\|_{{\mathcal H}^{\rho}(\R)} 
 \leq C \big(  \|u_0^f\|_{\W^{\rho, r}} + \| v\|_{\W^{\rho, r}} \big) \big(\|u_0^f\|_{L^{s}} + \| v\|_{L^{s}} \big)^{p-1},
\end{equation*}
and using the H\"older inequality in time, the bounds~\eqref{inter714}  and~\eqref{inter715}, we get 
\begin{multline} \label{7.22}
\big\| \cos^{\frac{p-5}{2}}(2s) F(u_0^f+ v)\big\|_{L^1(I_{t_0,\tau};{\mathcal H}^{\rho}(\R))} \leq  \\
\begin{aligned} 
&\leq C \tau^{\frac{ 8-p} 8}\big(\frac \pi 4 - |t_0|\big)^{\frac {p-5} 2} \big(  \|u_0^f\|_{L^8(I_{t_0,\tau}; \W^{\rho, r})} + \| v\|_{L^8(I_{t_0,\tau}; \W^{\rho, r})} \big) \big(\|u_0^f\|_{L^8(I_{t_0,\tau}; L^{s})} + \| v\|_{L^8(I_{t_0,\tau}; L^{s})} \big)^{p-1} \\
& \leq C \tau^{\frac{ 8-p} 8}\big(\frac \pi 4 - |t_0|\big)^{\frac {p-5} 2} \big( \| u_0 \|_{Y^{\rho, \epsilon}} + \| v\|_{X^\rho_{t_0, \tau} }\big)^p.
 \end{aligned}
 \end{multline} 
With  the inhomogeneous Strichartz estimates \eqref{stri1} we deduce
  \begin{equation}\label{K4}
 \| K(v)  \| _{X^\rho_{t_0, \tau}}\leq C \tau^{\frac{ 8-p} 8}  \big(\frac \pi 4 - |t_0|\big) ^{\frac {p-5} 2}\big(  \| u_0\|_{Y^{\rho, \epsilon}} + \| v\|_{X^\rho_{t_0, \tau} }\big)^p ,
 \end{equation}
and conclude similarly.

  \underline{Step 2: the operator $K$ is a contraction.} Let $v,w \in X^\rho_{t_0, \tau}$ be such that $ \Vert v\Vert_{X^\rho_{t_0, \tau}}, \Vert w\Vert_{X^\rho_{t_0, \tau}}\leq  1/C_0$.
  The contraction argument depends on whether $p>2$ or $p\leq 2$.  
  
   $\bullet$ Case $p> 2$. We  follow the same lines as in Step 1, but using~\eqref{5.5.0}  instead of~\eqref{5.4.1}. We get
 \begin{multline*} 
 \big\|\cos^{\frac{p-5}{2}}(2s)\big(F(u_0^f+v)- F(u_0^f+w)\big)\big\|_{L^1(I_{t_0,\tau}; \H^\rho)} \leq  \\
 \leq C \tau ^{\frac 3 8}  \big(\frac \pi 4 - |t_0|\big) ^{\frac {p-5} 2}  \| v- w\|_{X^\rho_{t_0, \tau} } \bigl( \|  u_0\|_{Y^{\rho, \epsilon}} + \| v\|_{X^\rho_{t_0, \tau} }\bigr) ^{p-1}  
 \end{multline*}
 which implies according to the inhomogeneous Strichartz estimates~\eqref{stri1}, 
    \begin{equation*}
 \| K(v) - K(w) \| _{X^\rho_{t_0, \tau}}\leq C \tau ^{\frac 3 8}  (\frac \pi 4 - |t_0|) ^{\frac {p-5} 2} \| v- w\|_{X^\rho_{t_0, \tau}}\bigl(\|  u_0\|_{Y^{\rho, \epsilon}} + \| v\|_{X^\rho_{t_0, \tau} }\bigr) ^{p-1}
 \end{equation*}
 and we  get the contraction by choosing  $\kappa, \delta>1$ large enough and $c>0$ small enough  in~\eqref{prop-tau}.  
 
   $\bullet$ Case    $\frac 32<p \leq 2$.  We shall prove that $K$ it is a contraction {\em in a weaker norm}.  Recall that  $X_{t_0,\tau}^0=L^{\infty}(I_{t_0, \tau}; L^2) \cap  L^{4}(I_{t_0, \tau}; L^\infty)$.
  From Proposition~\ref{prop.5.4}  we get
   \begin{equation}\label{Linf}
 \big\| F\big(u_0^f+ v\big)- F\big(u_0^f+ w \big)\big\|_{L^2}
  \leq C \| v-w\|_{L^2}   \big(   \| u_0 \|_{L^\infty}+\|v \|_{L^\infty} +  \|w \|_{L^\infty} \big)^{p-1}
  \end{equation}
 and  as previously we get
 \begin{multline*}
 \big\|\cos^{\frac{p-5}{2}}(2s)\big( F(u_0^f+v)- F(u_0^f+w)\big)\big\|_{L^1(I_{t_0,\tau} ; L^2)} \leq \\
 \leq   C \tau ^{\frac 3 8}(\frac \pi 4 - |t_0|)^{\frac {p-5} 2} \| v-w\|_{X^{0}_{t_0, \tau}}   \big(\|  u_0\|_{Y^{\rho, \epsilon}}  + \| v\|_{X^{\rho}_{t_0, \tau}} + \| w\|_{X^{\rho}_{t_0, \tau}}\big)^{p-1}.
 \end{multline*}
Taking in~\eqref{prop-tau} the values $\kappa, \delta>1$ large enough, $c>0$ small enough  we get (since $\| v\|_{X^{\rho}_{t_0, \tau}}\leq 1/C_0$ and $\| w\|_{X^{\rho}_{t_0, \tau}}\leq 1/C_0$),
$$ \| K(v) - K(w)\| _{X^0_{t_0, \tau}} \leq \frac 1 2 \| v-w\| _{X^0_{t_0, \tau}}. $$
The map $K$ sending the ball of radius $1/C_0$ of $X^\rho_{t_0, \tau}$ into itself and being a contraction for the $X^0_{t_0, \tau}$ topology has a unique fixed point in the  ball of radius $1/C_0$ of $X^{\rho}_{t_0, \tau}$.  
   
      $\bullet$ Case    $1<p \leq \frac32$.  The proof is similar to the previous case, but using \eqref{5.7.1} with the parameter~$r>2$ appearing in the definition of $Y^{\rho,\eps}$ instead of \eqref{Linf}.

    \underline{Step 3: Regularity of the fluctuation.}   To conclude the proof of the first part of Proposition~\ref{proplocal2}, it remains to prove that  
   \begin{equation*} 
v \in C\big((t_0-\tau,  t_0+\tau ); \H^\rho(\R)\big).
 \end{equation*}
  Since $v= K(v)$, we get  
$$v(t_2) - v(t_1) =-i  \int_{t_1}^{t_2}   \cos ^{\frac {p-5} 2} (2s)  e^{-i(t-s) H} F\big( u^f_0(s) + v(s)\big) ds.
$$ Then, as in \eqref{K1} (resp. \eqref{K2},  \eqref{K3} and \eqref{K4}) we get 
$$\| v(t_2) - v(t_1)\|_{\H^\rho} \leq C |t_2 - t_1| ^{\frac 3 8} (\frac \pi 4 - |t_0|)^{-\eta}\big(\|  u_0\|_{Y^{\rho, \epsilon}}  + \| v\|_{X^{\rho}_{t_0, \tau}} \big)^{p},$$
with $\eta=\max \big( \frac {5-p} 2, 0\big)$, hence the result.  

This  concludes the proof of the first part in Proposition~\ref{proplocal2} (existence, uniqueness, and regularity of the solution).

  \underline{Step 4: proof of~\eqref{continuity.1} and~\eqref{continuity.2}.} Consider two solutions $u =  u^f_0 + v$ and $\widetilde{u} =  \widetilde{u}^f_0 + \widetilde{v}$ of the equation~\eqref{eq_u}. 
  
$\bullet$ Case $p>2$.   
To prove~\eqref{continuity.1}, we again use~\eqref{5.5.0} and get 
 \begin{multline*}
 \big\|\cos^{\frac{p-5}{2}}(2s)\Big(F(u_0^f+v)- F(  \widetilde{u}^f_0+\widetilde{v})\Big)\big\|_{L^1(I_{t_0,\tau}; \H^\rho)} \leq \\
 \leq C \tau ^{\frac 3 8}  (\frac \pi 4 - |t_0|) ^{\frac {p-5} 2} \| (u_0^f+v) - (  \widetilde{u}^f_0+\widetilde{v})\|_{X^\rho_{t_0, \tau} } 
 \bigl( \|u_0\|_{Y^{\rho, \epsilon}} + \|\widetilde{u}_0\|_{Y^{\rho, \epsilon}} + \| v\|_{X^\rho_{t_0, \tau} }+\| \widetilde{v}\|_{X^\rho_{t_0, \tau} }\bigr) ^{p-1}  
 \end{multline*}
which implies, choosing again $\kappa, \delta>1$ large enough and $c>0$ small enough, 
    \begin{multline}\label{7.17}
 \| K_{u_0}(v) - K_{\widetilde{u}_0}(\widetilde{v}) \| _{X^\rho_{t_0, \tau} }  \leq \\
 \leq C \tau ^{\frac 3 8}  (\frac \pi 4 - |t_0|) ^{\frac {p-5} 2} \bigl( \| u_0 - \widetilde{u}_0 \|_{Y^{\rho, \epsilon}} +\| v- \widetilde{v}\|_{X^\rho_{t_0, \tau}}\bigr)\bigl(\|u_0\|_{Y^{\rho, \epsilon}} + \|\widetilde{u}_0\|_{Y^{\rho, \epsilon}} + \| v\|_{X^\rho_{t_0, \tau} }+\| \widetilde{v}\|_{X^\rho_{t_0, \tau} }\bigr) ^{p-1}\\
 \leq \frac 1 2  \bigl( \| u_0 - \widetilde{u}_0 \|_{Y^{\rho, \epsilon} } +\| v- \widetilde{v}\|_{X^\rho_{t_0, \tau}}\bigr).
 \end{multline}  
  Since for the two solutions $K_{u_0}(v) =v$, $K_{\widetilde{u}_0}(\widetilde{v}) = \widetilde{v}$, we get 
  \begin{equation*}
   \| v - \widetilde{v} \| _{X^\rho_{t_0, \tau} } = \| K_{u_0}(v) - K_{\widetilde{u}_0}(\widetilde{v}) \| _{X^\rho_{t_0, \tau} }\leq \| u_0 - \widetilde{u}_0 \|_{Y^{\rho, \epsilon} } .
  \end{equation*}
  Coming back to the solutions $u = e^{-i(t-t_0) H} u_0 + v$, $\widetilde{u}= e^{-i(t-t_0) H} \widetilde{u}_0 + \widetilde{v}$, we get
   \begin{equation*}
  \| u - \widetilde{u} \| _{X^\rho_{t_0, \tau}}\leq C \| u_0 - \widetilde{u}_0 \|_{Y^{\rho, \epsilon} } .
  \end{equation*}
  Finally, to prove~\eqref{continuity.2}, we only use that, since $v- \widetilde{v} =  K_{u_0} (v) - K_{\widetilde{u}_0} (\widetilde{v})$, we have from~\eqref{7.17}
  $$ \| v- \widetilde{v} \|_{L^\infty(I_{t_0,\tau}; \H^\rho)} \leq \| u_0 - \widetilde{u}_0 \|_{Y^{\rho, \epsilon} } 
  $$
  together with 
   \begin{equation*}
   \| u (t) - \widetilde{u} (t) \|_{Y^{\rho, \epsilon}} \leq \| e^{-i(t-t_0)H} ( u_0 - \widetilde{u}_0) \|_{Y^{\rho, \epsilon}} + \| v(t)- \widetilde{v} (t)\| _{\H^\rho} \\
  \leq 2\| u_0 - \widetilde{u}_0 \|_{Y^{\rho, \epsilon}} .
  \end{equation*}~

 $\bullet$ Case $\frac32<p\leq 2$.   Let $u =  u^f_0 + v$ and $\widetilde{u} =  \widetilde{u}^f_0 + \widetilde{v}$.
With the arguments of Step 2, we get the bound  $ \|v- \widetilde{v} \| _{X^{0}_{t_0, \tau}}\leq C \| u_0 - \widetilde{u}_0\|_{Y^{\rho, \epsilon}}$. Then by interpolation of this bound with $ \| v\|_{X^\rho_{t_0, \tau}}, \Vert \widetilde{v}\Vert_{X^\rho_{t_0, \tau}}\leq  1/C_0$, we deduce that for all $0<\rho'' < \rho$
\begin{equation}\label{interp0}
 \|v- \widetilde{v} \| _{X^{\rho''}_{t_0, \tau}}\leq C \| u_0 - \widetilde{u}_0\|_{Y^{\rho, \epsilon}}^\theta, \qquad \theta = 1- \frac {\rho''}  {\rho}.
\end{equation}
We are going to  follow a strategy inspired from~\cite{CaFaHa11}.   According to~\eqref{5.6.1bis} with $\sigma=\frac{2p-3}4$, we get that for all   $ \| u_0\|_{Y^{\rho, \epsilon}}, \| \widetilde{u}_0 \|_{Y^{\rho, \epsilon}} \leq R$
\begin{multline*}
\big\| F\big(u_0^f+ v\big) -F\big(  \widetilde{u}^f_0+ \widetilde{v}\big)\big\|_{L^{1} (I_{t_0,\tau} ; \H^\rho)}\leq \\
\leq C \big(R^{p-1}+ \| v\|_{X^\rho_{\tau, t_0}}^{p-1}+ \| \widetilde{v}\|_{X^{\rho}_{\tau, t_0}}^{p-1}\big)  \big( \|    u_0 - \widetilde{u}_0\|_{Y^{\rho, \epsilon}}+ \|v-\widetilde{v}\|_{X^{\rho}_{t_0, \tau}} \big)+ \\
\qquad +C (R+ \| v\|_{X^\rho_{\tau, t_0}})\| v-\widetilde{v}\|_{L^{8(p-1)}(I_{t_0,\tau}; \W^{\frac{ 2p-3}{4(p-1)},{4(p-1)}})}^{p-1} \leq \\
\leq C R^{p-1}\Bigl(\|   u_0 - \widetilde{u}_0\|_{Y^{\rho, \epsilon}}+ \|v-\widetilde{v}\|_{X^{\rho}_{t_0, \tau}}\Bigr)  + C R \| v-\widetilde{v}\|_{X^{\rho''}_{t_0, \tau}}^{p-1} ,
\end{multline*}
for some $\frac{4p-7}{4(p-1)}<\rho''<\rho$, by \eqref{sob3}. Now, we use  \eqref{interp0} together with Proposition~\ref{prop.5.1}, and get 
\begin{multline*}
   \| v- \widetilde{v}\|_{X^\rho_{t_0, \tau}} \leq \\
  \leq C R^{p-1} ( \frac \pi 4 - |t_0|)^{\frac{p-5} 2} \tau ^\kappa   \big( \|   u_0 - \widetilde{u}_0\|_{Y^{\rho, \epsilon}}+ \|v-\widetilde{v}\|_{X^{\rho}_{t_0, \tau}}   \big)
 + CR( \frac \pi 4 - |t_0|)^{\frac{p-5} 2} \tau ^\kappa    \| u_0 - \widetilde{u}_0\|_{Y^{\rho, \epsilon}}^{\theta(p-1)}.
\end{multline*}
Thus, choosing $C R ( \frac \pi 4 - |t_0|)^{\frac{p-5} 2} \tau ^\kappa < \frac 1 2$, gives 
$$ \| v- \widetilde{v}\|_{X^\rho_{t_0, \tau}} \leq  \|   u_0 - \widetilde{u}_0\|_{Y^{\rho, \epsilon}} +   \| u_0 - \widetilde{u}_0\|_{Y^{\rho, \epsilon}}^{\theta(p-1)},
$$
which proves~\eqref{continuity.1} when $\frac 3 2 < p \leq 2$.

 $\bullet$ Case $1<p <\frac 3 2$.      The proof in this case  is similar to the previous one,  by replacing~\eqref{5.6.1bis} in Proposition~\ref{prop.besov} by~\eqref{5.7.1bis}. We do not write the details. 

  \underline{Step 5: proof of~\eqref{persistance}.}  Finally, to prove~\eqref{persistance}, we  revisit the first step. We detail for example the case $1<p\leq \frac32$.  Starting from~\eqref{7.22}, with $\rho$ replaced by $\rho'>\rho$, and the same choice of $\tau>0$, we get
\begin{multline*}  
\big\| F\big(u_0^f+ v\big)\big\|_{L^1(I_{t_0,\tau};{\mathcal H}^{\rho'}(\R))} \leq \\
\begin{aligned}  
& \leq C \tau^{\frac{ 8-p} 8} \big(  \|u_0^f\|_{L^8(I_{t_0,\tau}; \W^{\rho', r})} + \| v\|_{L^8(I_{t_0,\tau}; \W^{\rho', r})} \bigr) \bigl(\|u_0^f\|_{L^8(I_{t_0,\tau}; L^{s})}  + \| v\|_{L^8(I_{t_0,\tau}; L^{s})} \bigr)^{p-1} \\
&\leq  C \tau^{\frac{ 8-p} 8} \big( \| u_0 \|_{Y^{\rho', \epsilon'}} + \| v\|_{X^{\rho'}_{t_0, \tau} }\big)\big( \| u_0 \|_{Y^{\rho, \epsilon}} + \| v\|_{X^\rho_{t_0, \tau} }\big)^{p-1}.
\end{aligned}  
 \end{multline*}
We deduce with the previous choice of $\tau>0$ (possibly choosing a smaller $c>0$ in $\tau$), 
\begin{eqnarray*} 
 \| v\|_{X^{\rho'}_{t_0, \tau}}
& \leq &C \tau^{\frac{ 8-p} 8}(\frac \pi 4 - |t_0|) ^{\frac {p-5} 2} \big( \| u_0 \|_{Y^{\rho', \epsilon'}} + \| v\|_{X^{\rho'}_{t_0, \tau} }\big)\big( \| u_0 \|_{Y^{\rho, \epsilon}} + \| v\|_{X^\rho_{t_0, \tau} }\big)^{p-1}\\
& \leq & \frac 1 2 \big( \| u_0 \|_{Y^{\rho', \epsilon'}} + \| v\|_{X^{\rho'}_{t_0, \tau} }\big)
 \end{eqnarray*}
which implies 
\begin{equation*} 
 \| v\|_{X^{\rho'}_{t_0, \tau}}
 \leq 2 \| u_0 \|_{Y^{\rho', \epsilon'}} .
 \end{equation*}
 Coming back to $u$ gives~\eqref{persistance}.

 \underline{Step 6: proof of~\eqref{Lp+1}.}  Again, we detail for example the case $1<p\leq \frac32$.  We study the contributions of $u= e^{-i(t-t_0)H} u_0$ and $v= K(v)$.
 Since the $Y^{\rho,\epsilon}$ norm controls the $L^\infty; \W^{{\eta-\epsilon}, p+1}$ norm of $e^{-itH} u_0$,  we have 
 \begin{equation}\label{interpol1}
  \|e^{-i(t-t_0)H} u_0-  u_0 \|_{\W^{{\eta-\epsilon}, p+1} }\leq 2R.
  \end{equation}
 On the other hand, since 
 \begin{equation}\label{interpol2}
 \big\|e^{-i(t-t_0)H} u_0-  u_0\big\|_{\W^{{\eta-\epsilon}-2, p+1} }= \big\| \int_{t_0}^t \partial_s(e^{-i(s-t_0)H} u_0 )ds\big\|_{\W^{{\eta-\epsilon}-2, p+1} } \leq R|t-t_0| \leq R \tau.
 \end{equation}
 Interpolating between~\eqref{interpol1} and~\eqref{interpol2} gives
\begin{equation*}
  \big\|e^{-i(t-t_0)H} u_0-  u_0 \big\|_{L^ {p+1}} \leq 2R \tau^{\epsilon_0}
  \end{equation*}
 for some $\epsilon_0>0$, and a suitable choice of $c, \delta, \kappa>0$ ensures that
  $$ \big\|e^{-i(t-t_0)H} u_0-  u_0 \big\|_{L^ {p+1}} \leq \frac 1 2.
  $$
 Let us now turn to the analysis of the contribution of $v$. 
 Replacing in~\eqref{7.22} the interval $I_{t_0,\tau}$ by~$(t',t)$, we get
 \begin{eqnarray*}
  \| v(t) - v(t')\| _{\mathcal{H}^\rho} &\leq& \big\|\cos^{\frac{p-5}{2}}(2s)F\big(e^{-i(s-t_0)H} u_0+v(s)\big)\big\|_{L^1((t',t) ; \H^\rho)} \\
 &\leq &C |t-t'| ^{\frac{ 8-p} 8}  (\frac \pi 4 - |t_0|) ^{\frac {p-5} 2}R^p  .
 \end{eqnarray*}
 Now according to the Sobolev embeddings $\H^\rho\subset L^{p+1}$, which implies 
 $$ \| v(t) - v(t') \| _{L^{p+1}} \leq   C |t-t'| ^{\frac{ 8-p} 8}  (\frac \pi 4 - |t_0|) ^{\frac {p-5} 2}R^p  
 $$
 and a suitable choice of $c, \delta, \kappa>0$ ensures that
  $$|t-t'|\leq \tau \;\; \Rightarrow  \;\; \| u(t) - u(t') \|_{L^{p+1}} \leq \frac 1 2\;,$$
  hence the result.
      \end{proof}

%%%%%%%%%%%%%%%%%%%%%%%%%%%%%%%%%%%%%%%%%%%%%%%%%%%%%%%%%%%%%%%
%%%%%%%%%%%%%%%%%%%%%%%%%%%%%%%%%%%%%%%%%%%%%%%%%%%%%%%%%%%%%%%

\section{Global existence for \texorpdfstring{$p>1$}{p larger than one}}\label{sec.9} In this section, we assume $\q_0= (0,1,1,0)$, {\it i.e.} $\mu_{\q_0} = \mu_0$. We show that the problem \eqref{C1} is globally well-posed on a set of full $\mu_0-$measure (Theorem~\ref{thmglobal}). We start with a result on the harmonic oscillator side.

\begin{prop}\label{prop-global} Let $p>1$ and consider
\begin{equation*}
\rho< \begin{cases} \frac{ p-1} 2 &\text{ if }\; 1<p \leq  2 \\[5pt]
 \frac 1 2 &\text{ if } \; p\geq   2   
\end{cases}
\end{equation*}
chosen sufficiently close to $\frac{ p-1} 2$ or $\frac 1 2$   respectively so that the assumptions in Section~\ref{sec.fonc} are satisfied.  Let  $\eps>0$ chosen small enough in the definition of the space $Y^{\rho, \epsilon}$.
There exists a set 
$  \Sigma $ of full $\mu_0-$measure such that the local solution~$u$ of~\eqref{eq_u} 
with initial condition $u_0\in \Sigma$ is  defined on
\begin{equation*}
I = \begin{cases} (- \frac \pi 4, \frac \pi 4) & \text{ if } 1< p < 5\\
\mathbb{R} & \text{ if } p \geq 5
\end{cases}
\end{equation*}
and we shall denote it by 
$$ u= \Phi(t,t_0 ) u_0=e^{-i(t-t_0)H} u_0 + v.$$
Moreover, for every $u_0\in  \Sigma$  and for any $ \eta >0$,  there exists $C>0$ such that 
\begin{equation*} 
  \| \Phi(t,  t_0 )u_0\|_{Y^{\rho, \epsilon}}\leq \begin{cases} C\Big( 1+ (\frac \pi 4 -|t|)^{\frac{p-5} 4 - \eta}\Big), \quad \forall t \in ( - \frac \pi 4, \frac \pi 4) & \text{ if } 1<p<5\\[5pt]
C\Big( 1+ \log^{\frac12}(1+|t|) \Big), \quad \forall t\in \R & \text{ if } p\geq 5.
\end{cases}
\end{equation*}
 
 Furthermore, we have an additional smoothness property: there exist $K, \gamma>0$ such that    for all $u_0\in \Sigma$, there exists $C>0$  such that 
\begin{equation}\label{borne-sol}
  \| v(t)\|_{\H^\rho} \leq \begin{cases} C \big(1+( \frac \pi 4 - |t|)^{-K}\big),\quad  \forall t \in ( - \frac \pi 4, \frac \pi 4)  &\text{ if } 1<p<5\\[8pt]
 C (1+ |t|) \big(1+  \log^{\frac 1 2}(1+ |t|)\big)^{ \gamma},\quad  \forall t\in \R  & \text{ if }  p \geq 5.
 \end{cases}
\end{equation}
\end{prop}

 {We emphasize that the constants $K,\gamma>0$ which appear in the previous statement are deterministic, and depend only on $p>1$ and $\rho>0$. They are obtained by an iteration of the local well-posedness result on small time intervals (see the proof of Proposition~\ref{215}).}

 We proceed in three steps. First we prove  bounds (independent of $N\geq 1$) on the solution of the approximate equation \eqref{C3_N}, then we pass to the limit $N\rightarrow + \infty$ to get  well-posedness for~\eqref{eq_u} on the interval $I$. Finally, we prove the quasi-invariance result.  For simplicity, in the proofs,  we only address the existence  for positive times. 
\subsection{Uniform estimates}\label{sec.9.1}

\begin{prop}\label{215}
Let $p>1$ and consider
\begin{equation*}
\rho< \begin{cases} \frac{ p-1} 2 &\text{ if }\; 1<p \leq  2 \\[5pt]
 \frac 1 2 &\text{ if } \; p\geq   2   
\end{cases}
\end{equation*}
chosen sufficiently close to $\frac{ p-1} 2$ or $\frac 1 2$   respectively so that the assumptions in Section~\ref{sec.fonc} are satisfied.  Let  $\eps>0$ chosen small enough in the definition of the space $Y^{\rho, \epsilon}$. Let  $\eta >0$, then 
 for all $i,N\in \N^{*}$, there exists a $\mu_N-$measurable set $\widetilde{\Sigma}^{i}_N\subset E_N$ so that  there exist $c,C, \eps_0>0$ with 
\begin{equation*}
\mu_N \big(E_N\backslash \widetilde{\Sigma}^{i}_N\big)\leq Ce^{-c i^{\epsilon_0}}\,,
\end{equation*}
and for all $u_0\in \widetilde\Sigma^i_N$  
\begin{equation}\label{bonn1}
\| \Phi_{N}(t,t_0)u_0\|_{Y^{\rho, \epsilon}}\leq \begin{cases} i+2+ (\frac \pi 4 -|t|)^{\frac{p-5} 4 - \eta}, \quad \forall t \in ( - \frac \pi 4, \frac \pi 4) & \text{ if } 1<p<5\\[5pt]
  i+2+ \log^{\frac12}(1+|t|) , \quad \forall t\in \R & \text{ if } p\geq 5.
\end{cases}
\end{equation}
 Moreover, there exist $ K_p, \gamma_p>0$ such that for all  $u_0\in \widetilde\Sigma^i_N$, setting $  v_N := \Phi_{N}(t,t_0)u_0 - e^{-i(t-t_0)H} u_0 $, there exists $C>0$ such that  
\begin{equation}\label{borne4bis}
\| v_N(t)\|_{\H^\rho} \leq \begin{cases}  C \bigl(i+2+( \frac \pi 4 - |t|)^{-K_p}\bigr)^{\gamma_p},\quad  \forall t \in ( - \frac \pi 4, \frac \pi 4)  &\text{ if } 1<p<5\\[5pt]
 C (1+ |t|) \big(i + 2+  \log^{\frac 1 2}(1+ |t|)\big)^{ \gamma_p},\quad  \forall t\in \R  & \text{ if }  p \geq 5.
 \end{cases}
\end{equation}
\end{prop}
\subsubsection{The case $1<p<5$}  It is enough to consider the case $t_0=0$.  We set, for $i,j \geq 1$ integers,  
\begin{equation*} 
B_N^{i,j}:=
\big\{u\in E_{N}\,:\,\|u\|_{Y^{\rho, \epsilon}}\leq i+j\big\}.
\end{equation*}
Let $0<  \alpha < \frac {4} { 5-p}$.  For any $t' \in (-\frac \pi 4 + \frac 2 {j^\alpha}, \frac \pi 4 - \frac{2} {j^\alpha})$, thanks to Proposition~\ref{proplocal2}, there exist $i_0 \in \N$ and  $\gamma>1$,    which only depends on $\rho>0$ and $p>1$,     such that 
\begin{equation} \label{tau}
\tau =c(i+j)^{-\gamma} 
\end{equation}
 for every $t\in (t'- \tau, t' + \tau)$, $i \geq i_0$,
\begin{equation}\label{flotN}
{\Phi}_{N}(t, t')\big(B_N^{i,j}\big)\subset 
\big\{u\in E_{N}\,:\,
\|u\|_{Y^{\rho, \epsilon}}
\leq i+j+ 1\big\} .\end{equation}
 Namely, the time of existence of Proposition~\ref{proplocal2} is $c(i+j)^{-\kappa}(\frac{\pi}4-t')^{\delta} \geq c(i+j)^{-\gamma}$ for $\gamma=\kappa+\delta \alpha$. In the sequel, $[x]$ stands for the integer part of $x\in \R$.  Remark  that for $|k| \leq [(\frac \pi 4 - \frac 2 { j^\alpha})/\tau]$, and since for $i\geq i_0$ large enough we have $\alpha< \gamma$, then 
$$(k+2) \tau \leq \frac \pi 4 - \frac 2 { j^\alpha} + 2c(i+j)^{-\gamma} \leq \frac \pi 4 \quad \Rightarrow\quad   k\tau \leq \frac \pi 4 - 2 \tau,
$$
and similarly 
$$ -\frac \pi 4 + 2 \tau \leq k \tau.$$
Let
$$
 B_{N}^{i,j, k}= {\Phi}_{N}( k\tau, 0)^{-1}(B_N^{i,j} ), \qquad  {\Sigma}_{N}^{i,j}:=
\bigcap_{k=-[(\frac \pi 4 - \frac 2 { j^\alpha})/\tau]}^{[(\frac \pi 4 - \frac 2 { j^\alpha})/\tau]} B_{N}^{i,j, k}\, .
$$
Notice that thanks to \eqref{flotN}, we obtain that the solution of \eqref{C3_N} with data $u_0\in{\Sigma}_{N}^{i,j}$ satisfies
\begin{equation}\label{eqflot}
\big\|{\Phi}_{N}(t, 0)u_0\big\|_{Y^{\rho, \epsilon}}
\leq i+j+1,\quad t\leq \frac{ \pi} 4 - \frac 2{j^\alpha}\,.
\end{equation}
Indeed, for $t\leq \frac{ \pi} 4 - \frac 2{j^\alpha}$, we can find an integer $|k| \leq [(\frac \pi 4 - \frac 2 { j^\alpha})/\tau]$, and 
$\tau_1\in [-\tau,\tau]$ so that $t=k\tau
+\tau_1$ and thus 
$u(t)= \Phi_{N}(k\tau+ \tau_1, k\tau) \big( \Phi_{N}(k\tau, 0)u_0\big)$.
Since $u_0\in{\Sigma}_{N}^{i,j}$ implies that $ \Phi_{N}(k\tau,0)u_0\in 
 B_{N}^{i,j}$, we can apply~\eqref{flotN} and get~\eqref{eqflot}.

By Proposition \ref{lemeq}, the measure $\nu_{N,t}$ is quasi-invariant  by the flow $\Phi_{N}$ and more precisely from~\eqref{abso-cont} with~$s=0$
  \begin{eqnarray*}
\nu_{N,0} \big(E_{N}\backslash{B}_{N}^{i,j,k}\big)
&\leq &\nu_{N,k \tau} \Big( \Phi_N ( k\tau, 0) \big(E_{N}\backslash{B}_{N}^{i,j,k}\big)\Big)^ { \cos^{\frac{ 5-p} 2}(2k\tau)}\\
&\leq &\nu_{N,k \tau} \big(E_{N}\backslash{B}_{N}^{i,j}\big)^ { \cos^{\frac{ 5-p} 2}(2k\tau)} \\
&\leq &\mu_{N} \big(E_{N}\backslash{B}_{N}^{i,j}\big)^ { \cos^{\frac{ 5-p} 2}(2k\tau)}.
\end{eqnarray*}
Now, since
$$ k\tau \le \frac \pi 4 - \frac 2 { j^\alpha}, \qquad  \cos(2k\tau) \geq  \frac 4 { j^{\alpha}}, \quad  \alpha < \frac {4} { 5-p}\;,$$
by the large deviation bound of  Proposition~\ref{large-dev} and Remark~\ref{uniform-large}, we have  
\begin{equation*}
 \mu_{N} \big(E_{N}\backslash{B}_{N}^{i,j}\big)^ { \cos^{\frac{ 5-p} 2}(2k\tau)}\leq
Ce^{-c (i+j)^2 j^{\frac{ (p-5)\alpha } 2 }}\leq C e^{-c' (i+j)^{\epsilon_0}}\;,
\end{equation*}
where   ${\epsilon_0} = 2 +\alpha  \frac{ p-5} 2 >0$. We deduce
\begin{equation}\label{zvez}
 \nu_{N,0} \big(E_{N}\backslash{\Sigma}_{N}^{i,j} \big)\leq C(i+j) ^\gamma  e^{-c' (i+j)^{\epsilon_0}}\leq C e^{-c'' (i+j)^{\epsilon_0}}.
\end{equation}

Next, we set
\begin{equation}\label{def2}
\widetilde{\Sigma}^i_N=\bigcap_{j= 1}^{+\infty}{\Sigma}_{N}^{i,j}\,.
\end{equation}
Thanks to \eqref{zvez},
\begin{equation}\label{mestilde}
\nu_{N,0}(E_{N}\backslash \widetilde{\Sigma}^i_N)\leq  C\sum_{j \leq i}e^{-c'' (i+j)^{\epsilon_0}} + C \sum_{j > i} e^{-c'' (i+j)^{\epsilon_0}} \leq C e^{-c i^{\epsilon_0}}  .
\end{equation}
In addition, using \eqref{eqflot}, we get that   for every $i \geq i_0$, every
$N\geq 1$, every $u_0\in \widetilde{\Sigma}_N^i$, every $0<t< \frac{\pi}4$,
\begin{equation*} 
\big\| \Phi_{N}(t,0 )u_0\big\|_{Y^{\rho, \epsilon}}
\leq i+2+{( \frac \pi 4 - t)^{-\frac 1 \alpha}}.
\end{equation*}
 Indeed, for $0<t< \frac{\pi}4$ there exists $j\geq 2 $ such that $\frac 1 {(j-1)^\alpha} < \frac \pi 4 - t < \frac{ 1 } {j^\alpha}$ and we apply  
\eqref{eqflot} with this~$j\geq 2$.
Choosing $\alpha < \frac{ 4} {5-p} $ but arbitrarily close to $\frac 4 {5-p}$ proves \eqref{bonn1}. 
To prove~\eqref{borne4bis}, for $u_0\in \Sigma^{i,j}_N$, we apply Proposition~\ref{proplocal2} with Remark~\ref{rem8.5} which implies that on each interval $[k\tau, (k+1) \tau]$, we have 
$$ u= e^{-i(t- k\tau)H} u\mid_{t= k\tau}+ v_N, \quad \| v_N(t) \|_{\H^\rho} \leq 1.$$
Iterating this estimate between $0$ and $t$ leads to (recall that $\tau = c(i+j)^{- \gamma}$)
$$ \Phi_{N}(t,0 )u_0= e^{-it H} u_0 + v_N, \quad \| v_N(t) \|_{\H^\rho} \leq C(i+j)^\gamma,
$$ 
and for $u_0 \in \widetilde{\Sigma}^i_N$, choosing again  $\frac 1 {(j-1)^\alpha} < \frac \pi 4 - t < \frac{ 1 } {j^\alpha}$ gives 
$$ \Phi_{N}(t,0 )u_0= e^{-it H} u_0 + v_N,  \quad \| v_N(t) \|_{\H^\rho} \leq C\Big(i+(\frac \pi 4 -t)^{-\frac 1 {\alpha}}\Big)^{\gamma },
$$ 
which proves~\eqref{borne4bis}. 

\subsubsection{The case $ p \geq 5$}
We revisit the proof above, taking benefit from the better estimate in Proposition~\ref{proplocal2} and~\eqref{abso-cont}.
 Thanks to Proposition~\ref{proplocal2}, there exist $i_0\geq 1$  and $\gamma=\kappa>0$ (only depending on $\rho>0$ and $p\geq 5$)
such that if we set 
\begin{equation*}
\tau =c(i+j)^{-\gamma},\qquad \gamma=\kappa,
\end{equation*}
for every $t_0 \in \R$ and for every $t_1\in(t_0 - \tau, t_0 + \tau)$, $i \geq i_0$,
\begin{equation}\label{flotNbis}
{\Phi}_{N}(t_1, t_0)\big(B_N^{i,j}\big)\subset 
\big\{u\in E_{N}\,:\,
\|u\|_{Y^{\rho, \epsilon}}
\leq i+j+ 1\big\} .\end{equation}
Let ${\epsilon_0} >0$ to be fixed later and
$$
 B_{N}^{i,j, k}= {\Phi}_{N}( k\tau, 0)^{-1}(B_N^{i,j} ), \qquad  {\Sigma}_{N}^{i,j}:=
\bigcap_{k=-[2^{{\epsilon_0} j^2}/\tau]}^{[2^{{\epsilon_0} j^2}/\tau]} B_{N}^{i,j, k}\, ,
$$
where $[2^{{\epsilon_0} j^2}/\tau]$ stays for the integer part of $2^{{\epsilon_0} j^2}/\tau$. 
Notice that thanks to \eqref{flotNbis}, we obtain that the solution of \eqref{C3_N} with data $u_0\in{\Sigma}_{N}^{i,j}$ satisfies
\begin{equation}\label{eqflotbis}
\big\|{\Phi}_{N}(t, 0)u_0\big\|_{Y^{\rho, \epsilon}}
\leq i+j+1,\quad t\leq 2^{{\epsilon_0} j^2}\,.
\end{equation}
Indeed, for $t\leq 2^{{\epsilon_0} j^2}$, we can find an integer $ k= [t/\tau] \leq  2^{{\epsilon_0} j^2}/\tau $, and 
$\tau_1\in [-\tau,\tau]$ so that $t=k\tau
+\tau_1$ and thus 
$u(t)= \Phi_{N}(k\tau+ \tau_1, k\tau) \big( \Phi_{N}(k\tau, 0)u_0\big)$.
Since $u_0\in{\Sigma}_{N}^{i,j}$ implies that $ \Phi_{N}(k\tau, 0)u_0\in 
 B_{N}^{i,j}$, we can apply~\eqref{flotNbis} and get~\eqref{eqflotbis}.

Now we apply   Proposition \ref{lemeq} with $s=0$, and from~\eqref{abso-cont} we deduce
\begin{eqnarray*}
\nu_{N,0} \big(E_{N}\backslash{B}_{N}^{i,j,k}\big) &\leq& \nu_{N,k \tau} \big( \Phi_N ( k\tau, 0) \big(E_{N}\backslash{B}_{N}^{i,j,k}\big)\big)\\
&\leq &\nu_{N,k \tau} \big(E_{N}\backslash{B}_{N}^{i,j}\big)\\
&\leq &\mu_{N} \big(E_{N}\backslash{B}_{N}^{i,j}\big).
\end{eqnarray*}
Then, by Proposition~\ref{large-dev} and Remark~\ref{uniform-large} we have
\begin{equation*}
\mu_{N} \big(E_{N}\backslash{B}_{N}^{i,j}\big)\leq Ce^{-c (i+j)^2},
\end{equation*}
which in turn implies
\begin{equation}\label{zvez1} 
\nu_{N,0} \big(E_{N}\backslash{\Sigma}_{N}^{i,j} )\leq C(i+j) ^\kappa  2^{{\epsilon_0} j^2} e^{-c (i+j)^2}\leq C e^{-c' (i+j)^2},
\end{equation}
if ${\epsilon_0} < c/2$.
Next, we set
\begin{equation}\label{def3}
\widetilde{\Sigma}^i_N=\bigcap_{j= 1}^{+\infty}{\Sigma}_{N}^{i,j}\,.
\end{equation}
Thanks to \eqref{zvez1},
$$
\nu_{N,0}(E_{N}\backslash \widetilde{\Sigma}^i_N)\leq  C\sum_{j \geq 1}e^{-c' (i+j)^2} \leq C e^{-c' i^2} .
$$
In addition, using \eqref{eqflotbis},   for every $i\geq i_0$, every
$N\geq 1$, every $u_0\in \widetilde{ \Sigma}^i_N$, every $t >0$,
\begin{equation*} 
\big\| \Phi_{N}(t,0 )u_0\big\|_{Y^{\rho, \epsilon}} \leq  i+2+ \log^{\frac 1 2} (1+ t).
\end{equation*}
Indeed for $t>1$ there exists $j\geq 1 $ such that $2^{{\epsilon_0} (j-1)^2} \leq t <2^{{\epsilon_0} j^2}$ and we apply 
\eqref{eqflotbis} with this $j\geq 1$.
This proves \eqref{bonn1}. 
To prove~\eqref{borne4bis}, for $u_0\in \Sigma^{i,j}_N$, we apply Proposition~\ref{proplocal2} with Remark~\ref{rem8.5} which gives that on each interval $[k\tau, (k+1) \tau]$, we have 
$$ u= e^{-i(t- k\tau)H} u\mid_{t= k\tau}+ v_N, \quad \| v_N(t) \|_{\H^\rho} \leq 1.$$
Iterating this estimate between $0$ and $t$  on each interval $[k\tau, (k+1) \tau]$ leads to (recall that $\tau = c(i+j)^{- \kappa}$)
$$ \Phi_{N}(t,0 )u_0= e^{-it H} u_0 + v_N, \quad\| v_N(t) \|_{\H^\rho} \leq C(1+t)(i+j)^\kappa,
$$ 
and for $u_0 \in \widetilde{\Sigma}^i_N$, choosing again $2^{{\epsilon_0} (j-1)^2} \leq t <2^{{\epsilon_0} j^2}$ gives 
$$ \Phi_{N}(t,0 )u_0= e^{-it H} u_0 + v_N, \quad\| v_N(t) \|_{\H^\rho} \leq C(1+ t)\big(i+ \log^{\frac 1 2} ( 1 + t)\big)^{\kappa},
$$ 
which proves~\eqref{borne4bis}.

 %%%%%%%%%%%%%%%
\subsection{Passing to the limit \texorpdfstring{$N\rightarrow + \infty$} {N tends to infinity}}\label{paraclosed}
For integers $i\geq i_0$ and $N \geq 1$, we define the
cylindrical sets
\begin{equation}\label{def1}
\Sigma_{N}^{i}:=\big\{u\in X^0(\R)\,:\, \Pi_{N}u\in \widetilde{\Sigma}^{i}_N\big\}\;,
\end{equation}
where $\widetilde{\Sigma}^{i}_N$ is defined in~\eqref{def2} or in~\eqref{def3}. Next, for $i\geq i_0$, we define
\begin{equation*}
\Sigma^{i}:=\big\{ u\in X^0(\R) :\exists N_{k}, \lim_{k\rightarrow + \infty} N_k = +\infty, \;
\exists\, u_{N_{k}}\in  {\Sigma}^{i}_{N_{k}}, \lim_{k\rightarrow + \infty} \|S_{N_k}u_{N_{k}}- u\|_{Y^{\rho, \epsilon}}=0\big\}.
\end{equation*}

Let us prove that  $\Sigma^{i}$ is a closed subset of $Y^{\rho, \epsilon}$. The closedness property is clear, it is enough to show that   $\Sigma^{i} \subset Y^{\rho, \epsilon}$. Assume that there exists $u_{N_k} \in {\Sigma}^{i}_{N_{k}}$ such that $ \lim_{k\rightarrow + \infty} \|S_{N_k}u_{N_{k}}- u\|_{Y^{\rho, \epsilon}}=0$. Then for any $P\in \mathbb{N}$, as soon as $N_k \gg P$, we have 
$$ \|S_P(u_{N_{k}}- u)\|_{Y^{\rho, \epsilon}}= \|S_P(S_{N_k}u_{N_{k}}- u)\|_{Y^{\rho, \epsilon}}\leq C  \|S_{N_k}u_{N_{k}}- u\|_{Y^{\rho, \epsilon}}\rightarrow 0.$$
As a consequence, using~\eqref{bonn1} (with $t=t_0=0$)  and Lemma \ref{Stone}, we deduce 
\begin{equation*}
 \|S_Pu\|_{Y^{\rho, \epsilon}}  \leq  \limsup_{k\rightarrow + \infty} \|S_P( u_{N_k})\|_{Y^{\rho, \epsilon}}  
 \leq C  (i+1)
 \end{equation*}
and passing to the limit $P\rightarrow + \infty$, we deduce that $u\in Y^{\rho, \epsilon}$ and 
$$  \|u\|_{Y^{\rho, \epsilon}} \leq C(i+1) .$$

Next, we prove   we have the following inclusions
\begin{equation}\label{inclusion}
\limsup_{N\to +\infty}\Sigma_{N}^{i}=\bigcap_{N=1}^{+\infty}\bigcup_{N_{1}=N}^{+\infty}\Sigma_{N_{1}}^{i}\subset \Sigma^{i}.
\end{equation}
Indeed, if $\dis u\in \limsup_{N\to +\infty}\Sigma_{N}^{i}$, there exists $N_k\rightarrow +\infty$ such that 
$$ \Pi_{N_k}u \in \widetilde \Sigma^i_{N_k},
$$ and the same proof as above shows that   
$$ u\in Y^{\rho, \epsilon}, \quad  \|u\|_{Y^{\rho, \epsilon}} \leq C(i+1).$$
Now, we clearly have 
$$ 
\|S_n u - u \|_{Y^{\rho, \epsilon}}= o(1) _{n\rightarrow + \infty},
$$
and since $S_n (\Pi_n u) = S_n u$,  the sequence $u_{N_k}:= \Pi_{N_k}u$ is the one ensuring that $u\in \Sigma^i$. 
This proves~\eqref{inclusion}. 

Consider $G_{N}(u)=\exp(-\frac1{p+1}\|S_Nu\|^{p+1}_{L^{p+1}(\R)})$, $G(u)=\exp(-\frac1{p+1}\|u\|^{p+1}_{L^{p+1}(\R)})$ and recall that 
 $$d\nu_0 =G(u) d\mu_0, \qquad d\nu_{N,0} =G_N(u) d\mu_0.$$
As a consequence of \eqref{inclusion},  we get the inequality
\begin{equation}\label{kr1}
\nu_0\big(\Sigma^{i}\big)\geq \nu_0\big(\limsup_{N\to +\infty}\Sigma_{N}^{i}\big).
\end{equation}
Using Fatou's lemma, we get
\begin{equation}\label{kr2}
\nu_0(\limsup_{N\rightarrow +\infty}\Sigma_{N}^{i})
\geq 
\limsup_{N\rightarrow \infty}\nu_0(\Sigma_{N}^{i})\,.
\end{equation}
 We have (because the set $\Sigma^i_N$ is cylindrical),
$$
\nu_0(\Sigma_{N}^{i})=\int_{\Sigma_{N}^{i}}G(u)d\mu_0(u),
$$
and
$$
{\nu}_{N,0}(\widetilde{\Sigma}_{N}^{i})=\int_{\widetilde{\Sigma}_{N}^{i}}G_{N}(u)d{\mu} _{N}(u)=\int_{\Sigma_{N}^{i}}G_{N}(u)d{\mu} _{0}(u).$$
We deduce 
$$ \Bigl| \nu_0(\Sigma_{N}^{i})- {\nu}_{N,0}(\widetilde\Sigma_{N}^{i})\Bigr| \leq \int_{\Sigma_{N}^{i}} |G(u)- G_N(u)|  d\mu_0(u)= o(1)_{N\rightarrow+ \infty},
$$ 
where we used the dominated convergence theorem and the fact that $\mu_0-$a.s. $u \in L^{p+1}$ and consequently  $\mu_0-$a.s 
$$ \lim_{N\rightarrow + \infty} G_N (u) = G(u).$$
Therefore, using Proposition~\ref{215}, we obtain
\begin{equation*}
\begin{aligned}
\limsup_{N\rightarrow \infty}\nu_0(\Sigma_{N}^{i})
& = 
\limsup_{N\rightarrow \infty}{\nu}_{N,0}(\widetilde\Sigma_{N}^{i})
\\
& \geq 
\limsup_{N\rightarrow \infty}\big(\nu_{N,0}(E_N)-Ce^{-c i^{\epsilon_0}}\big)
=
\nu_0(X^0(\R))-Ce^{-c i^{\epsilon_0}}.
\end{aligned}
\end{equation*}
Collecting the last estimate,  \eqref{kr1}, and \eqref{kr2}, we obtain that
\begin{equation}\label{kr5}
\nu_0\big(\Sigma^{i}\big) \geq \nu_0(X^0(\R))-Ce^{-c i^{\epsilon_0}}.
\end{equation}
%%%%%%%%%%%%%%%%%%%%%%%%%%%%%%%%%%%%%%%%%%%%%%%%%%%%%%%%%%%%%%%%%%%%%%%%%%%%%%%%%%%%%%%%%%%%%%%%%%%%%%%%%%%

\medskip

 Now, we set
\begin{equation*} 
\Sigma:=\bigcup_{i=i_0}^{\infty}\Sigma^{i}.
\end{equation*}
Then, by \eqref{kr5}, the set $\Sigma$ is of full $\nu_0-$measure (and hence also of full $\mu_0-$measure).
It turns out that one has global existence for any initial condition $u_0\in \Sigma$.
 Recall  that $\widetilde{\Phi}_N$ stands for the extension of the flow~${\Phi}_N$ on $\H^{-\eps}(\R)$, see \eqref{def-phitilde}. We now state the global existence results.
%%%%%%%%%%%%%%%%%%%%%%%%%%%%%%%%%%%%%%%%%%%%%%%%%%%%%%%%%%%%%%%%%%%%%%%%%%%%%%%%%%%%%%%%
\begin{prop}\label{prop.sigmabis}
Let $p>1$ and consider
\begin{equation*}
\rho< \begin{cases} \frac{ p-1} 2 &\text{ if }\; 1<p \leq  2 \\[5pt]
 \frac 1 2 &\text{ if } \; p\geq   2   
\end{cases}
\end{equation*}
chosen sufficiently close to $\frac{ p-1} 2$ or $\frac 1 2$   respectively so that the assumptions in Section~\ref{sec.fonc} are satisfied.  Let  $\eps>0$ chosen small enough in the definition of the space $Y^{\rho, \epsilon}$. For  any initial condition $u_0\in \Sigma$,  there exists a unique  global solution~$u$ of \eqref{eq_u} 
in the class
$$u = e^{-i(t-{ t_0})H}u_0 + C(I; \H^\rho(\R)), $$
where 
\begin{equation*}
I = \begin{cases} (- \frac \pi 4 , \frac \pi 4) & \text{ if } 1<p<5\\
\R & \text{ if } p\geq5.
\end{cases}
\end{equation*}
 We shall denote this solution by $u= \Phi(t,0) u_0$. 
Moreover, every integer $i\geq i_0$ and  for any $ \eta >0$,  there exists $C>0$ such that for every $u_0\in \Sigma^{i}$ \begin{equation}\label{borne5quart}
  \| \Phi(t,{ t_0})u_0\|_{Y^{\rho, \epsilon}}\leq \begin{cases} C\big( i+2+ (\frac \pi 4 -|t|)^{\frac{p-5} 4 - \eta}\big), \quad  \forall t \in ( - \frac \pi 4, \frac \pi 4) & \text{ if } 1<p<5\;,\\[5pt]
C\big( i+2+ \log^{\frac12}(1+|t|) \big),\quad  \forall t\in \R & \text{ if }  p \geq 5,
\end{cases}
\end{equation}
and if $ \Phi(t,{ t_0})u_0 = e^{-i(t-{ t_0})H} u_0 + v$ we have the bounds
\begin{equation}\label{borne5bis}
  \| v(t)\|_{\H^\rho} \leq  \begin{cases} C \bigl(i+( \frac \pi 4 - |t|)^{-K_p}\bigr), \quad \forall t \in ( - \frac \pi 4, \frac \pi 4)  &\text{ if } 1<p<5\\[5pt]
 C (1+ |t|) \big(i + 2+  \log^{\frac 1 2}(1+ |t|)\big)^{\gamma_p},\quad  \forall t\in \R & \text{ if } p\geq 5.
 \end{cases}
 \end{equation}
  Furthermore, if $(u_{0,N_k})_{k\geq 0}\in \Sigma^{i}_{N_{k}}$, $N_{k}\to +\infty$ are so that 
$$
 \lim_{k\rightarrow + \infty}\|S_{N_k}u_{0,N_k}- u_0\|_{Y^{\rho,\epsilon}}=0, 
$$ 
 then for all $t\in I$ 
\begin{equation}\label{eq.limite}
\forall  \sigma<\rho,\quad  \lim_{k \rightarrow+ \infty}\big \|\widetilde{\Phi}_{N_{k}}(t,0)u_{0,N_k}-u(t)\big \|_{Y^{\sigma, \epsilon}}=0.
\end{equation}
\end{prop}

The key point in the proof of Proposition~\ref{prop.sigmabis} is the following lemma. Recall the notation $I_{t_0,\tau}=( t_0 - \tau, t_0 + \tau)$ and recall that the set $\Sigma_N^i$ is defined in \eqref{def1}, \eqref{def2}, and \eqref{def3}. 
\begin{lem}\label{lem.limite}
There exist $\kappa, \delta\geq 1$ and $c>0$ such that the following holds true. Let $\rho$ satisfying the assumptions of Proposition~\ref{prop.sigmabis}. For any $R \geq 1$, consider a sequence $ u_{0,N_k}\in {\Sigma}^i_{N_k}$ and $u_0\in Y^{\rho, \epsilon}$ 
such  that 
\begin{equation}\label{hyp9.4} 
 \|u_{0,N_k}\|_{Y^{\rho, \epsilon}}\leq R,\quad \|u_0\|_{Y^{\rho, \epsilon}} \leq R,\quad \forall \sigma < \rho, \quad \lim_{k\rightarrow+ \infty}\|u_{0, N_k} -u_0\|_{Y^{\sigma, \epsilon}}=0.
\end{equation}
Then if we set 
\begin{equation}\label{def-tau1}
  \tau\leq \begin{cases}  c R^{-\kappa}(\frac{\pi}4-|t_0|)^{\delta}  &\text{ if } 1< p < 5\\
cR^{-\kappa}& \text{ if } p \geq 5
\end{cases}
\end{equation}
 the quantities
$\widetilde{\Phi}_{N_k}(t, t_0 ) u_{0, N_k}$ and $\Phi(t,t_0)u_0$ exist for $|t-t_0| \leq \tau$ and satisfy 
$$
\|\widetilde{\Phi}_{N_k}(t, t_0 ) u_{0, N_k}\|_{L^\infty(I_{t_0,\tau};Y^{\rho, \epsilon})}\leq 
R +1, \qquad \|\Phi(t, t_0) u_{0}\|_{L^\infty(I_{t_0,\tau};Y^{\rho, \epsilon})}\leq R +1.
$$
Furthermore,  
\begin{equation}\label{eq-diff}
\forall \sigma< \rho, \quad \lim_{k\rightarrow+ \infty} \| \widetilde{\Phi}_{N_k}(t, t_0) u_{0, N_k}- \Phi(t, t_0) u_0\|_{L^\infty(I_{t_0,\tau};Y^{\sigma, \epsilon})} =0.
\end{equation}
\end{lem}
\begin{proof}[Proof of Lemma~\ref{lem.limite}]
 The first part of this lemma is a direct consequence of our local 
well-posedness results of Proposition~\ref{proplocal2}  and Remark~\ref{rem8.5}. 
It remains to prove \eqref{eq-diff}. For that, let us write
$$
\Phi(t,t_0) u_0= u= e^{-i(t-t_0)H} u_0 + v, 
\qquad \widetilde{\Phi}_{N_k}(t,t_0) u_{0,N_k} = u_k = e^{-i(t-t_0)H} u_{0,N_k} + v_k.
$$
 Let us now remark that from the first part in the lemma, 
$$
\|\widetilde{\Phi}_{N_k}(t,t_0) u_{0,N_k}-\Phi(t,t_0) u_0 \|_{L^\infty(( t_0 - \tau, t_0 + \tau);Y^{\rho, \epsilon})}\leq 2R +2,
$$
and consequently it is enough to prove~\eqref{eq-diff} for some $\sigma > 0$.
We have 
$$ 
u-u_k = e^{-i(t-t_0)H}(u_0 - u_{0,N_k})+ (v- v_k).
$$
By assumption, for any $0<\sigma<\rho$
$$
\| e^{-i(t-t_0)H} (u_0 -  u_{0,N_k})\|_{Y^{\sigma, \epsilon}}= 
\|u_0 -  u_{0,N_k}\|_{Y^{\sigma, \epsilon}} = o(1)_{k\rightarrow + \infty}.
$$
  Therefore it remains to show that some $\sigma > 0$ we have 
\begin{equation}\label{conv00}
\|v-v_k\|_{L^\infty(I_{t_0,\tau};Y^{\sigma, \epsilon})}=o(1)_{k\rightarrow + \infty},
\end{equation}
for $\tau>0$ chosen as in the statement of the lemma. Set $w_k=v-S_{N_k}v_k$ and let us prove that 
\begin{equation}\label{conv0}
\|w_k\|_{L^\infty(I_{t_0,\tau};Y^{\sigma, \epsilon})}=o(1)_{k\rightarrow + \infty},
\end{equation}
which will imply \eqref{conv00} by Lemma~\ref{lem-borne} and \eqref{persistance}. Observe that $w_k$ solves the problem
\begin{eqnarray}\label{eq.diff.bis}
(i\partial_t - H) w_k
&=& \cos^{\frac{ p-5} 2} (2t) \big(|u|^{p-1} u - S^2_{N_k}( |S_{N_k} u_k|^{p-1} S_{N_k}u_k)\big)\nonumber 
\\
&=&\cos^{\frac{ p-5} 2} (2t) (1- S^2_{N_k} ) ( |u|^{p-1} u) + 
 \cos^{\frac{ p-5} 2} (2t) S^2_{N_k} \big( |u|^{p-1} u - |S_{N_k} u_k|^{p-1} S_{N_k}u_k\big)
\end{eqnarray}
with initial condition $w_k \mid_{t=t_0} =0$.
Standard estimates now show 
$$ 
 \big\|  \cos^{\frac{ p-5} 2} (2t)  |u|^{p-1} u\big \|_{L^{1} (I_{t_0,\tau}; L^2)} 
\leq C \tau^\kappa\|u\|_{X_{t_0,\tau}^0}^{p}
\leq C\tau^\kappa(R+1)^p
$$
and consequently, by dominated convergence,
\begin{equation}\label{eq.estim1} 
\big\| \cos^{\frac{ p-5} 2} (2t) (1- S^2_{N_k})( |u|^{p-1} u)\big \|_{L^{1} (I_{t_0,\tau}; L^2)} 
\rightarrow 0\;\; \text{ as } \;\; k\rightarrow + \infty\,.
\end{equation}
Thus from the  Strichartz estimates \eqref{stri1} in $Z^0$, we deduce that the contribution of this term to $w_k$ is bounded by $o(1)$ in $L^\infty; L^2$. 
We estimate the second term in the r.h.s. of~\eqref{eq.diff.bis} by using a direct manipulation
on the expression $|z_1|^{p-1}z_1-|z_2|^{p-1}z_2$.  Recall that $X_{t_0,\tau}^0=L^{\infty}(I_{t_0, \tau}; L^2) \cap  L^{4}(I_{t_0, \tau}; L^\infty)$ and   denote by    $\widetilde{X}_{t_0,\tau}^0=L^{8}(I_{t_0, \tau}; L^4) \cap  L^{4}(I_{t_0, \tau}; L^\infty)$. Then by Proposition~\ref{prop.5.4}.  
\begin{multline}\label{eq.estim2}
\big\| \cos^{\frac{ p-5} 2} (2t)S_{N_k}\Big( |u|^{p-1} u - |S_{N_k} u_k|^{p-1} S_{N_k}u_k\Big)\big \|_{L^{1} ( I_{t_0,\tau} ;L^2)} \leq \\
\begin{aligned}
&\leq C\tau^{\kappa}\| u- S_{N_k} u_k\|_{\widetilde{X}_{t_0,\tau}^0} \big(\| u\|_{X_{t_0,\tau}^\rho}+ \| S_{N_k}u_k\|_{X_{t_0,\tau}^\rho}\big)^{p-1}
\\
&\leq C\tau^{\kappa}(R+1)^{p-1}\big( \| e^{-i(t-t_0)H}(u_0 -S_{N_k}u_{0,N_k}) \|_{L^{\infty}(I_{t_0,\tau}; Y^{\rho,\eps})} + \| w_k\|_{X_{t_0,\tau}^0}\big)
\\
&\leq C\tau^{\kappa}(R+1)^{p-1}\big( \| u_0 -S_{N_k}u_{0,N_k} \|_{Y^{\rho,\eps}} + \| w_k\|_{X_{t_0,\tau}^0}\big)\\
&\leq o(1) _{k\rightarrow + \infty} + C\tau^{\kappa}(R+1)^{p-1} \| w_k\|_{X_{t_0,\tau}^0}\,.
\end{aligned}
\end{multline}
 We deduce from~\eqref{eq.estim1}  and \eqref{eq.estim2}, 
$$
\|w_k\|_{X_{t_0,\tau}^0} 
\leq C\tau^\kappa(R+1)^{p-1} 
\| w_k\|_{X^0_{t_0,\tau}} + o(1)_{k\rightarrow + \infty}\,.
$$
  By taking $C\tau^\kappa( R+1)^{p-1} \leq 1/2$, we infer that 
$
\|w_k\|_{X^0_{t_0,\tau}}=o(1)_{k\rightarrow + \infty}.
$
Next, by interpolation with~\eqref{persistance}, we deduce that for any $0<\sigma<\rho$, we have $\|w_k\|_{X^\sigma_{t_0,\tau}}=o(1)_{k\rightarrow + \infty}$. Finally we choose $\sigma<\rho$ large enough such that we can apply~\eqref{incl2}  which implies \eqref{conv0}. This completes the proof of Lemma~\ref{lem.limite}.
\end{proof}

\begin{proof}[Proof of Proposition~\ref{prop.sigmabis}]
The local existence result follows from Proposition~\ref{proplocal2}.  Here the main points are the globalisation and the limit~\eqref{eq.limite}.  We only consider  the case $1<p<5$, the case $p\geq 5$ being similar. We assume in the sequel that $t_0=0$ in~\eqref{eq_u}. Let $u_0 \in \Sigma^i$. By assumption, we know that there exist sequences 
$N_k \in \mathbb{N}$, $u_{0,N_k}\in {\Sigma}^i_{N_k}$ 
such that 
$$\lim_{k\rightarrow + \infty}\|S_{N_k}u_{0,N_k}- u_0\|_{Y^{\rho, \epsilon}}=0.$$
 Consequently,
by Proposition~\ref{215}, we know that for any $\eta>0$
\begin{equation}\label{eq.estapriori}
\big\|\widetilde{\Phi}_{N_k}(t,0)(\Pi_{N_k}u_{0,N_k})\big\|_{Y^{\rho, \epsilon}}\leq   i+2+ (\frac \pi 4 -t)^{\frac{p-5} 4 - \eta}.
\end{equation}
The strategy of proof consists in proving that as long as the solution to~\eqref{eq_u} exists, 
we can pass to the limit in~\eqref{eq.estapriori} 
and there exists a constant $C'>0$ independent of $i \geq i_0$ such that 
\begin{equation}\label{eq.estaprioribis}
\big\|\Phi(t,0)u_0\big\|_{Y^{\rho, \epsilon}}
\leq C' \big( i+2+ (\frac \pi 4 -t)^{\frac{p-5} 4 - \eta}\big)
\end{equation}
which (taking into account that the norm in $Y^{\rho, \epsilon}$ controls the local existence time), 
implies that the solution is global and satisfies~\eqref{eq.estaprioribis} for all times.

Let us fix $T\in (- \frac \pi 4, \frac \pi 4)$.  We assume 
\begin{equation}\label{eq.estaprioriter}
\big\|\widetilde{\Phi}_{N_k}(t,0)(\Pi_{N_k}u_{0,N_k})\big\|_{Y^{\rho, \epsilon}}
\leq \Lambda, \quad \text{ for } |t| \leq T
\end{equation}
and we want to show 
\begin{equation}\label{eq.estaprioriquar}
\big\|\Phi(t,0)u_{0}\big\|_{Y^{\rho, \epsilon}}
\leq C'\Lambda, \quad \text{ for } |t| \leq T.
\end{equation}
As a first step, let us fix $t=0$. For $Q\in \mathbb{N}$, if 
$N_k \geq Q$, $\Pi_{N_k} S_Q = S_Q$ and consequently, using Lemma~\ref{Stone}
and the definition of $\Sigma^i$, we obtain 
$$
\|S_Q u_0\|_{Y^{\rho, \epsilon}}= \lim_{k\to + \infty} \|S_Q (\Pi_{N_k}u_{0,N_k})\|_{Y^{\rho, \epsilon}} \leq C'\Lambda$$
and passing to the limit $Q \rightarrow + \infty$, we deduce 
$$\|u_0\|_{Y^{\rho, \epsilon}}= \lim_{Q\rightarrow + \infty}\|S_Q u_{0}\|_{Y^{\rho, \epsilon}}\leq C'\Lambda.$$
This implies that the sequences $\Pi_{N_k} u_{0,N_k}$ and $u_0$ satisfy the assumptions 
of Lemma~\ref{lem.limite} (with $R=C'\Lambda$). As a consequence, we know that 
$$\forall \sigma< \rho, \quad  \lim_{k\rightarrow+ \infty} \| \widetilde{\Phi}_{N_k} (t,0)( \Pi_{N_k} u_{0,N_k})-\Phi(t,0) u_0\|_{L^\infty((0,\tau); Y^{\sigma, \epsilon})} =0$$ 
for $\tau= c_T \Lambda^{-\kappa}$ given in \eqref{def-tau1}. 
Now we show that this convergence allows to pass to the limit  in~\eqref{eq.estaprioriter} for $t= \tau$, using Lemma~\ref{Stone} again.
Indeed, fix $Q$, then for $N_k\gg 2Q$, the sequence 
$ S_Q  \big(\Phi_{N_k} ( \tau,0) (\Pi_{N_k} u_{0,N_k})\big)$ is bounded in $Y^{\rho, \epsilon}$ by $C' \Lambda$, and converges to  
$S_Q\big( \Phi(\tau,0) u_0\big)$ in $Y^{\sigma,\epsilon}$ for all $0<\sigma<\rho$.  Here, the constant $C'>0$ is given by Lemma~\ref{Stone}.  We deduce that $S_Q \big(\Phi(\tau,0) u_0\big)\in Y^{\rho, \epsilon}$ and 
$$\|S_Q \big(\Phi(\tau,0) u_0\big)\|_{ Y^{\rho, \epsilon}} \leq C' \Lambda.$$
Next, passing to the limit $Q\rightarrow + \infty$, we deduce that $\Phi(\tau,0)u_0\in {Y^{\rho, \epsilon}}$ and 
\begin{equation*}
\|\Phi(\tau,0)u_0\|_ {Y^{\rho, \epsilon}}\leq C' \Lambda.
\end{equation*}
Now, we can apply the results in Lemma~\ref{lem.limite}, with the same $\Lambda$ as in 
the previous step (remark that the assumption~\eqref{hyp9.4} is now  true for any  $\sigma<\rho$) which implies that ~\eqref{eq.estaprioriquar} holds for $t\in[0, 2\tau]$, 
and so on and so forth. 

Notice here that at each step the {\it a priori} bound does not get worse, because we only use the results in 
Lemma~\ref{lem.limite} to obtain the convergence of $\big\|\widetilde{\Phi}_{N_k} ( t,0) (\Pi_{N_k} u_{0,N_k})- \Phi(t,0) u_0\big\|_{Y^{\rho, \epsilon}} $ to 
$ 0$, and then obtain the estimates on the norm $ \|\Phi(t,0) u_0\|_{Y^{\rho, \epsilon}}$ by passing to the 
limit in~\eqref{eq.estaprioriter} (applying first $S_Q$, passing to the limit $k\rightarrow + \infty$, then to the limit $Q\rightarrow + \infty$). A completely analogous argument holds for the negative times~$t$.
\end{proof}

%%%%%%%%%%%%%%%%%%%%%%%%%%%%%%%%%%%%%%%%%%%%%%%%%%%%%%%%%%%%%%%
%%%%%%%%%%%%%%%%%%%%%%%%%%%%%%%%%%%%%%%%%%%%%%%%%%%%%%%%%%%%%%%

\section{Quasi-invariance of the measures}\label{sec.10}
\subsection{Passing to the limit \texorpdfstring{$N\rightarrow + \infty$}{N to infinity} in Proposition~\ref{lemeq}} Recall that the measure~$\nu_{t}$ is defined in \eqref{defnutilde} by 
$$d\nu_{t}= e^{-\frac{\cos^{\frac{p-5}{2}}(2t)}{p+1}\|u\|_{L^{p+1}(\R)}^{p+1}}d\mu_0.$$
In particular 
$$\nu_{t} \ll \mu_0 \qquad \text{and}  \qquad   \mu_0 \ll \nu_{t}.$$

The purpose of this section is to show the following result.
\begin{prop}\label{lemeqlim-bis}
For all $t, t'\in (- \frac \pi 4, \frac \pi 4)$,
\begin{equation*} 
  \Phi(t,t') _{\#} \mu_0 \ll \mu_0 \ll \Phi(t,t') _{\#}\mu_0. 
\end{equation*}
More precisely, for all $0\leq |t'| \leq |t| < \frac \pi 4 $  and all $A \subset \Sigma$,
\begin{align}\label{equival-bis}
  \nu_{t}\big(\Phi(t,0)A\big) & \leq \begin{cases}  \nu_{t'}\big(\Phi(t',0)A\big)&\text{ if } 1\leq p \leq 5 \\[5pt]
  \Big[\nu_{t'}\big(\Phi(t',0)A\big)\Big] ^{\smash{\bigl( \frac{\cos(2t)}{ \cos( 2t')}\bigr) ^{\frac{ p-5} 2}}} &\text{ if } p \geq 5\\ \end{cases}\\
\intertext{and} 
\nu_{t'}\bigl(\Phi(t',0) A \bigr)&\leq \begin{cases}\Bigl[\nu_{t} \big(\Phi(t,0) A\big)\Bigr] ^{\smash{\bigl( \frac{\cos(2t)} {\cos(2t')}\bigr) ^{\frac{ 5-p} 2}}}&\text{ if } 1\leq p \leq 5\\[5pt] \nu_{t} \big(\Phi(t,0) A\big) &\text{ if }  p \geq 5. \end{cases}
 \label{abso-cont-bis}
 \end{align}
\end{prop}

We will need   the following statement comparing $\Phi(t, t_0)$ and $\widetilde{\Phi}_{N}(t,t_0)$ for small  $|t-t_0|$
\begin{lem}\label{lem.limite.bis}
 Let   $t_0, t\in (- \frac \pi 4, \frac \pi 4)$. Let $\rho'>\rho$ and $0<\eps'<\eps$ which satisfy the assumptions of Proposition~\ref{prop.sigmabis}.  There exists $\kappa, \delta\geq 1$ and $c>0$ such that if we set 
\begin{equation*} 
  \tau\leq \begin{cases}  c R^{-\kappa}(\frac{\pi}4-|t_0|)^{\delta}  &\text{ if } 1< p < 5\\
cR^{-\kappa}& \text{ if } p \geq 5\;,
\end{cases}
\end{equation*}
the following holds true:  there exist $C, \delta'>0$ such that for every $R>0$, every $R'>0$,  every $u_0 \in Y^{\rho', \epsilon'}$  such that  $\| u_0 \|_{Y^{\rho, \epsilon}} \leq R$ and  $\| u_0 \| _{Y^{\rho', \epsilon'}}\leq R'$, if $|t-t_0| \leq \tau$, then
$$
\|\Phi(t,t_0)u_0-\widetilde{\Phi}_N(t,t_0)u_0\|_{Y^{\rho, \epsilon}}<CR' N^{-\delta'}.
$$
\end{lem}
\begin{proof}[Proof of Lemma~\ref{lem.limite.bis}]    Recall the notation $u_0^f= e^{-i(t-t_0)H} u_0$. First, we write
$$ \Phi(t,t_0)u_0 = u_0^f + v, \qquad \widetilde{\Phi}_N(t,t_0)u_0 = u_0^f + v_N $$
with 
$$v = K(v), \qquad v_N = S_N \big(K(v_N)\big),$$
where the operator $K$ is defined in~\eqref{defK}. We deduce
$$
\Phi(t,t_0)u_0-\widetilde{\Phi}_N(t,t_0)u_0= v- v_N = w_N,
$$ where 
\begin{equation}\label{first}
 (i\partial_t + H) w_N = (1- S_N) \big(F(u_0^f + v)\big) 
 + S_N \big( F(u_0^f + v)- F(u_0^f + S_Nv_N)\big).
\end{equation}
From~\eqref{persistance}, we know 
\begin{equation}\label{persistance2}
\|v\|_{X^{\rho'}_{t_0, \tau}} + \|v_N\|_{X^{\rho'}_{t_0, \tau}}\leq 2R'
\end{equation} 
\begin{equation*} 
\| F(u_0^f + v)\|_{L^1((t_0,t); \H^{\rho'})}\leq R'.
\end{equation*}
We get
$$\| w_N\|_{X^0_{t_0, \tau}} \leq C R'N^{-\rho'}+ \|  F(u_0^f + v)- F\big(u_0^f + S_Nv_N\big)\|_{L^1((t_0,t); L^2)}.$$
From the H\"older inequality we get easily 
$$ \|  F\big(u_0^f + v\big)- F\big(u_0^f + S_Nv_N\big)\|_{L^1((t_0,t); L^2)}\leq C \tau^{\delta'} R^{p-1}\| v- v_N\|_{X^{0}_{t_0, \tau}},
$$
where $X_{t_0,\tau}^0=L^{\infty}(I_{t_0, \tau}; L^2) \cap  L^{4}(I_{t_0, \tau}; L^\infty)$. Then taking $c>0$ small enough and $\kappa>0$ large enough in the definition of $\tau$, and using the Strichartz estimate \eqref{stri1} in~\eqref{first} gives
$$\| w_N\|_{X^0_{t_0, \tau}} \leq C R'N^{-\rho'}+ \frac 1 2\| w_N\|_{X^0_{t_0, \tau}}\quad  \Rightarrow \quad \| w_N\|_{X^0_{t_0, \tau}} \leq 2C R'N^{-\rho'} .$$
Interpolation between this bound and~\eqref{persistance2}, we get that for all $0<\rho<\rho'$
$$ \| w_N\|_{X^\rho_{t_0, \tau}} \leq 2C R'N^{-(\rho'-\rho)}
$$
which implies Lemma~\ref{lem.limite.bis}.
\end{proof}

\begin{proof}[Proof of Proposition~\ref{lemeqlim-bis}]  Let us prove for example~\eqref{abso-cont-bis} in the case $1<p\leq 5$. By the regularity properties of the measures $\nu_t$, it is enough to prove this result if the set $A$ is closed. Now we assume that $0\leq t' < t < \frac \pi 4$. 
Recall that our measures~$ \nu_t$ are seen as  finite Borel measures on $Y^{\rho, \epsilon}$.
Let $\Sigma$ \big(resp. $\Sigma_t = \Phi(t,0) \Sigma$\big) be the set of full $\nu_0$ (resp. $\nu_t$) measure constructed
in the previous section.  Clearly, for all $0\leq t' < t < \frac \pi 4$,
$$ \dis\Sigma_t = \Phi(t,0) \Sigma, \qquad  \Sigma_t =  \Phi(t,t') \Sigma_{t'}, \qquad \Sigma = \bigcup_{i=1}^{+\infty} \Sigma^i,\qquad  \Sigma_t = \bigcup _{i=1}^{+\infty} \Sigma^i_t, \qquad \Sigma^i_t = \Phi(t,t') \Sigma^i_{t'}$$
and by Fatou's Lemma, 
$$ \forall B \in \Sigma_t, \quad \nu_t (B) = \lim_{i\rightarrow + \infty} \nu_t ( B \cap \Sigma^i_t).$$

As a consequence, we can replace $A$ by $ A\cap \Sigma^i$ (which is also closed). Let $T<\frac{\pi}4$. From Section~\ref{paraclosed}, we know that   $\Sigma^i$ are closed in~$Y^{\rho, \epsilon}$, and from Proposition~\ref{prop.sigmabis}  the set  $\Sigma^i_t$ is bounded in $Y^{\rho, \epsilon}$ uniformly with respect to $t \in [0, T]$ by
$$  C\big( i+ 2 + (\frac \pi 4 - T) ^{\frac{p-5} 4 - \eta}\big).$$
Let $\rho'>\rho, \epsilon'<\epsilon$ sufficiently close to each other. Now from the large deviation bounds in Section~\ref{sec.7}, if $B_n$ is the ball of radius $n$ in $Y^{\rho', \epsilon'}$, we have, for $B\in \Sigma_t$, 
$$  \mu_0 (B) = \mu_0( B \cap Y^{\rho', \epsilon'}) = \lim_{k\rightarrow + \infty} \mu_0 (B \cap B_k), 
$$ 
and the same relation holds with $\mu_0$ replaced by $\nu_t$.  As a consequence, we can replace $A$ by $A\cap \Sigma^i \cap B_k$.  Dividing the interval $[0,T]$ by a finite number $P$ of intervals of size $\tau$, applying~\eqref{persistance} we get that for any $t\in [0,T]$, $\Phi(t,0) \big(A\cap \Sigma^i \cap B_k\big)$ is bounded in $Y^{\rho', \epsilon'} $ by $M^{P_i} k :=C_{i,k}$.  Hence we can assume that $\Phi(t,0) A$ is closed in $Y^{\rho, \epsilon}$ and bounded in $Y^{\rho',\epsilon'}$ uniformly with respect to $t\in [0, T]$. 
Now, let $0\leq t' \leq t \leq T<\frac{\pi}4$. For $A \in \Sigma^i$, we have
$$  \nu_{t'} \big( \Phi(t',0) A\big) = \nu_{t'} \big( \cup_{k} \Phi(t',0) (A \cap B_k)\big)= \lim_{k\rightarrow + \infty } \nu_{t'} \big( \Phi(t',0) (A \cap B_k)\big) ,$$
thus 
\begin{multline*}
   \nu_t \big( \Phi(t,0) A\big)= \nu_t \big( \Phi(t,t') \Phi(t',0) A\big) \geq \\
  \geq  \nu_t \Big( \cup_{k} \Phi(t,t') \big((\Phi(t',0) A) \cap B_k\big)\Big)\geq  \limsup_{k\rightarrow + \infty} \nu_t \Big(  \Phi(t,t') \big((\Phi(t',0) A) \cap B_k\big)\Big).
\end{multline*}

As a consequence, it is enough to prove~\eqref{abso-cont-bis} with  $\widetilde{A}= \Phi(t',0) A$ replaced by $   \widetilde{A}_k=(\Phi(t',0)A )\cap B_k$ and $ \Phi(t,t')\widetilde{A}$ replaced by $\Phi(t,t') \widetilde{A}_k$. Notice that according to~\eqref{persistance}, since $\widetilde{A}_k$ is bounded by $k$ in $Y^{\rho',\epsilon'}$, we know that  $ \Phi(t,t') \widetilde{A}_k$ is bounded uniformly with respect to $t\in [0, T]$ in $Y^{\rho',\epsilon'}$ by the constant $C_{i,k}>0$.

We now proceed and prove~\eqref{abso-cont-bis} by time increments $|t-t'| \leq  \tau$ as defined in Lemma~\ref{lem.limite.bis}. 
Let $\eps_0 >0$ and $N\geq 1$ large enough such that 
$$CR' N^{-\delta'}<\eps_0.$$ 
 From Lemma~\ref{lem.limite.bis} between $t'$ and $t'+\tau$ we have,  $B_{\epsilon_0}$ being the ball of radius $\epsilon_0$ in $Y^{\rho,\epsilon}$, with $A_k^{t'}= \Phi(t',0)\widetilde{A}_k$,
\begin{equation*}
\nu_t \big(\Phi(t'+ \tau, t') A_k^{t'} + B_{\epsilon_0}\big) = \lim_{N\rightarrow + \infty} \widetilde{\nu}_{N,t} \big(\Phi(t'+ \tau, t') A_k^{t'} + B_{\epsilon_0}\big) \geq \limsup_{N\rightarrow+ \infty} \widetilde{\nu}_{N,t} \big( \widetilde{\Phi}_{N} (t'+\tau, t') A_k^{t'}\big) ,
\end{equation*}
 where the first limit above is simply obtained by the Lebesgue dominated convergence theorem.  From ~\eqref{abso-cont} combined with Remark~\ref{extension}, we have 
\begin{equation*}
{\widetilde{\nu}}_{N,t'} (A_k^{t'}) 
\leq \begin{cases} \smash{\Bigl({\widetilde{\nu}}_{N,t} ( \widetilde{\Phi}_{N} (t'+\tau, t') A_k^{t'})\Bigr)^{\bigl(\frac{ \cos(2( t'+ \tau))}{\cos(2t')}\bigr) ^{\frac {5-p} 2}}}&\text{ if } 1\leq p \leq 5\\[5pt]
  {\widetilde{\nu}}_{N,t} \big( \widetilde{\Phi}_{N} (t'+\tau, t') A_k^{t'}\big) &\text{ if }  p \geq 5 \end{cases}
\end{equation*}
which implies 
\begin{equation}\label{approchee}
 \nu_{t'}( A_k^{t'}) = \lim_{N\rightarrow + \infty} {\widetilde{\nu}}_{N,t'} (A_k^{t'}) \leq \begin{cases} \smash{\Bigl(\nu_t \big(\Phi(t'+ \tau, t')A_k^{t'} + B_{\epsilon_0}\big)\Bigr)^{\bigl(\frac{ \cos(2( t'+ \tau))}{\cos(2t')}\bigr) ^{\frac {5-p} 2}}} &\text{ if } 1\leq p \leq 5\\[5pt]
 \nu_t \big(\Phi(t'+ \tau, t')A_k^{t'} + B_{\epsilon_0}\big) &\text{ if }  p \geq 5. \end{cases}
 \end{equation}
 We have
$$ \lim_{\epsilon_0 \rightarrow 0} \nu_{t }\big( \Phi(t'+ \tau, t')A_k^{t'} + B_{\epsilon_0}\big)= \nu_t \big(\overline{\Phi(t'+ \tau, t')A_k^{t'}}\big),$$
and since $\Phi(t'+ \tau, t')A_k^{t'}$ is closed in $Y^{\rho, \epsilon}$, 
passing to the limit $\epsilon_0 \rightarrow 0$ in~\eqref{approchee} gives

\begin{equation*} 
 \nu_{t'}( A_k^{t'})    \leq \begin{cases} \smash{\Bigl(\nu_t \big(\Phi(t'+ \tau, t')A_k^{t'}  \big)\Bigr)^{\bigl(\frac{ \cos(2( t'+ \tau))}{\cos(2t')}\bigr) ^{\frac {5-p} 2}}} &\text{ if } 1\leq p \leq 5\\[5pt]
 \nu_t \big(\Phi(t'+ \tau, t')A_k^{t'}  \big) &\text{ if }  p \geq 5. \end{cases}
 \end{equation*}
Applying this estimate between $t'$ and $t$ by increments smaller than $\tau$ gives ~\eqref{abso-cont-bis} for all $0\leq |t'| \leq |t| \leq T$. Since $T<\frac \pi 4$ is arbitrary, this proves~\eqref{abso-cont-bis}.
The proof of~\eqref{equival-bis} is similar.
The proof of Proposition~\ref{lemeqlim-bis} is therefore completed.
\end{proof}

\subsection{Global existence for \texorpdfstring{$(NLS_p)$}{NLS-p}: proof of Theorem~\texorpdfstring{\ref{thm2}}{2.2}}\label{Sect7}

We are now ready to prove Theorem~\ref{thm2}, in the particular case $\q_0= (0,1,1,0)$ ($\mu_{\q_0} = \mu_0$). 
For this we use the inverse of the lens transform~\eqref{lens}--\eqref{lensbis}. From Proposition~\ref{prop.sigmabis} and Proposition~\ref{lemeqlim-bis}, we know that we can solve~\eqref{eq_u} for    every initial data in the set~$\Sigma$, the solution takes the form $u = \Phi(t,0) u_0$, and we have full $\mu_0-$measure sets $\Sigma_t= \Phi(t,0) \Sigma$. Applying the inverse lens transform and \eqref{conjnl}, we define the sets 
$$ \mathbf{S}_s:= \Psi(s,0)\Sigma= \mathscr{L}_{t(s)} ^{-1} ( \Sigma_{t(s)}).
$$
We now  check that these sets are of full $\mu_{\q_s}-$measure with $\q_s= (s,1,1,0)$. Actually  by~\eqref{compol} we have
\begin{equation*} 
\mu_{\q_s}(\mathbf{S}_s)=\mu_{0}\big( \mathscr{L}_{t(s)}  \mathscr{L}_{t(s)} ^{-1}   \Sigma_{t(s)} \big)= \mu_{0}\big(   \Sigma_{t(s)} \big)=1.
 \end{equation*}

 The first part of Theorem~\ref{thm2} is just the fact that the lens transform conjugates the flows of~\eqref{C1} and~\eqref{C3}, and the  second part follows from Proposition~\ref{prop-global} and Lemma~\ref{lemA}.

%%%%%%%%%%%%%%%%%%%%%%%%%%%%%%%%%%%%%%%%%%%%%%%%%%%%%%%%%%%%%%%%%%%%%%

\section{Decay estimates and scattering}\label{sec.11}
In this section we are going to exploit the quasi-invariance properties of the measures  $\Phi(t,0) _{\#} \nu_0$   to get almost sure estimates for the evolution of the $L^{p+1}$ norms, and prove the scattering results in Theorem~\ref{thm1}.
\subsection{Decay estimates}
The first step is to prove the following estimates on our solutions on the harmonic oscillator side, obtained in Proposition~\ref{prop-global}.
\begin{prop}\label{croissance}
 There exists a set $\widetilde{\Sigma}$ of full $\mu_0-$measure such that for all  $u_0\in \widetilde{\Sigma}$,  there exists $C>0$ such that the  global  solution of~\eqref{C3}, given by $\Phi(t,0) u_0$, satisfies  
\begin{equation*}
 \| \Phi(t,0) u_0 \|_{L^{p+1} }\leq \begin{cases}
 C |\log^{\frac 1 {p+1}} (\frac\pi 4 -|t|)|, \quad \forall t\in ( - \frac \pi 4, \frac \pi 4), &\text{ if } 1<p<5\\[5pt]
  C\big(  1+\log^{\frac 1 2} (1+|t|)\big) ,  \quad\forall t \in \R, &\text{ if }  p \geq 5.
  \end{cases}
  \end{equation*}
  \end{prop}
By applying the inverse lens transform~\eqref{lens}, $t\in (- \frac \pi 4, \frac \pi 4) \rightarrow s(t) \in \R$ using that 
\begin{equation}\label{estlq}
 \|\mathscr{L}_t(G) \|_{L^q} = \cos^{\frac 1 q - \frac 1 2 }(2t) \| G \|_{L^q}, 
\end{equation}
 we can translate this result into

\begin{cor}\label{coro-lp}
 There exists a set $\widetilde{\Sigma}$ of full $\mu_0-$measure such that for all  $u_0\in \widetilde{\Sigma}$,  there exists $C>0$ such that the  global  solution of~\eqref{C1}, given by $\Psi(s,0) u_0$, satisfies   for all $s\in \R$ 
\begin{equation*}
 \| \Psi(s,0) u_0 \|_{L^{p+1} }\leq \begin{cases}
C\frac{(1+  \log\<s\> )^{1/(p+1)}} {\langle s\rangle ^{ \frac 1 2 -\frac 1 {p+1}}}&\text{ if } 1<p<5\\[5pt]
 C\frac{1} {  \langle s \rangle ^{ \frac 1 2 - \frac 1 {p+1}}} &\text{ if } p \geq 5. 
\end{cases}
\end{equation*}
\end{cor}

For $p\geq 5$, Proposition~\ref{croissance} follows from~\eqref{borne5quart} and the fact that the $Y^{\rho, \epsilon}$ norm controls the $L^{p+1}$ norm. For $1<p<5$, the starting point is the following observation. 

\begin{lem}\label{brique}  Assume that $1<p<5$.   Let $\Lambda>0$ and 
$$\mathcal{K}_\Lambda= \big\{ u \in X^0(\R)\;: \;\; \| u\|_{L^{p+1}} >\Lambda\big\}.$$
 For any $ 0\leq |t| <\frac \pi 4$
$$ \nu_0 \big( \Phi(t,0)^{-1}(\mathcal{K}_\Lambda)\big) \leq C e^{- \frac{ \Lambda^{p+1}}{p+1}}.$$
\end{lem}

\begin{proof}
The proof is straightforward.  According to~\eqref{abso-cont-bis} with $t'=0$ and  $A= \Phi(t,0)^{-1}(\mathcal{K}_\Lambda)$ we have
\begin{multline*}
 \nu_0 (A) \leq \bigl(\nu_t (\mathcal{K}_\Lambda)\bigr) ^{\cos^{\frac{5-p} 2}(2t) }= \Bigl(\int_{\mathcal{K}_\Lambda} e^{ - \frac{\cos(2t)^{\frac{ p-5} 2}}{p+1} {\| u\|_{L^{p+1}}^{p+1}}} d\mu_0 \Bigr) ^{\cos^{\frac{5-p} 2}(2t) } \leq \\
 \leq \Bigl( e^{ - \frac{\cos^{\frac{ p-5} 2}(2t)}{p+1} {\Lambda^{p+1}} }\int_{\mathcal{K}_\Lambda} d\mu_0 \Bigr) ^{\cos(2t) ^{\frac{5-p} 2}} = e^{ - \frac{\Lambda^{p+1}} {p+1}} (\mu_0 (\mathcal{K}_\Lambda)) ^{\cos^{\frac{5-p} 2}(2t) } \leq e^{ - \frac{\Lambda^{p+1}} {p+1}}  \;,
 \end{multline*}
 which was the claim.
 \end{proof}
 
 \begin{proof}[Proof of Proposition~\ref{croissance}]
We detail the case $1<p< 5$ (the case $p\geq 5$ is similar). It is enough to consider the case $t>0$. Let   $M>0$ large enough, to be fixed later,  and 
$$ A^{i,j} _N= \big\{ u \in E_N;\;\; \| S_N u \|_{L^{p+1} } \leq M \log^{\frac 1 {p+1}}(i+j)\big\}.$$
As in the proof of Lemma~\ref{brique}, we get
\begin{equation}\label{zvez2}
 \nu_{N,0} \Big(\Phi_N(t,0) ^{-1} \big(E_N \backslash A^{i,j}_N\big)\Big) \leq (i+j) ^{- \frac {M^{p+1}} {p+1}}
\end{equation}
Then with $\tau= (i+j)^{- \gamma}$ as in~\eqref{tau}, define 
$$
 A_{N}^{i,j, k}= {\Phi}_{N}( k\tau, 0)^{-1}(A_N^{i,j} ), \qquad  {S}_{N}^{i,j}=
\bigcap_{k=-[(\frac \pi 4 - \frac 2 { j^\alpha})/\tau]}^{[(\frac \pi 4 - \frac 2 { j^\alpha})/\tau]} A_{N}^{i,j, k}\, ,
$$
and from~\eqref{zvez2}, we get
$$ \nu_{0,N} (E_N \backslash A^{i,j,k}_N) \leq  (i+j) ^{- \frac {M^{p+1}} {p+1}} 
$$ 
therefore
$$ \nu_{0,N} (E_N \backslash S^{i,j}_N) \leq \frac{C} {\tau} (i+j) ^{- \frac {M^{p+1}} {p+1}} \leq C (i+j) ^{\gamma - \frac {M^{p+1}} {p+1}}.
$$
Now let 
$$ \widetilde{S}^i_N = \bigcap_{j=1}^{+\infty} S^{i,j}_N, 
$$
and choose $M>0$ large enough such that if $\gamma - \frac {M^{p+1}} {p+1}< -1$, so that 
$$
 \nu_{0,N} (E_N \backslash \widetilde{S}^{i}_N) \leq C \sum_{j=0}^{+\infty}   (i+j) ^{\gamma - \frac {M^{p+1}} {p+1}}\leq C i ^{1+\gamma - \frac {M^{p+1}} {p+1}}.
$$
Then from \eqref{mestilde} we deduce
$$\nu_{0,N} \big(E_N \backslash (\widetilde{S}^{i}_N\cap \widetilde{\Sigma}^i_N)\big)\leq C i ^{1+\gamma - \frac {M^{p+1}} {p+1}} + C e^{-ci^{\epsilon_0}}.
$$
We now claim that for any  $u \in \widetilde{\Sigma}^i_N \cap \widetilde{S}^i_N$, 
we have 
\begin{equation}\label{controle}
 \|\Phi_N(t,0) u \|_{L^{p+1}} \leq M \log^{\frac 1 {p+1}}\big( i+1+(\frac \pi 4-t)^{-\gamma}\big) +1.
 \end{equation}
 Indeed,  for $0 \leq t < \frac \pi 4$, let $j\geq 2$ be such that
 $$ t\in \big[\frac \pi 4 - 2(j-1)^{- \gamma}, \frac\pi 4 - 2j^{- \gamma}\big],$$
which implies $j\leq 1+ (\frac \pi 4 - t) ^{-1/\gamma}$. With $\tau = c(i+j) ^{- \gamma}$, we can find an integer $|k| \leq [(\frac \pi 4 - \frac 2 { j^\gamma})/\tau]$, and 
$\tau_1\in [0,\tau]$ so that $t=k\tau
+\tau_1$ and thus since from the definition of $S^{i,j}_N$ we have 
$$  \|\Phi_N(k \tau,0) u \|_{L^{p+1} }\leq M \log^{\frac 1 {p+1}}( i+j) \leq M \log^{\frac 1 {p+1}} \big(i +1+ (\frac \pi 4 - t) ^{-1/\gamma}\big)  .
$$
As a consequence~\eqref{controle} follows from~\eqref{Lp+1} in Proposition~\ref{proplocal2}. 

For integers $i\geq i_0$ and $N \geq 1$, we now define the
cylindrical sets
$$
 S_{N}^{i}:=\big\{u\in X^0(\R)\,:\, \Pi_{N}(u)\in \widetilde{S}_{N}^{i}\big\}.
$$
Next, for $i\geq i_0$, we set
\begin{equation*}
  S^{i}=\big\{ u\in X^0(\R): \exists N_{k}, \lim_{k\rightarrow + \infty} N_k = +\infty, 
\exists\, u_{N_{k}}\in  {\Sigma}^{i}_{N_{k}}\cap S^i_{N_k}, \lim_{k\rightarrow + \infty} \|S_{N_k}u_{N_{k}}- u\|_{Y^{\rho, \epsilon}}=0\big\},
\end{equation*}
so that, as in~\eqref{kr5},
$$ \nu_{0} \big( X^0(\R) \backslash (S^i\cap {\Sigma}^i)\big)\leq C i ^{1+\gamma - \frac {M^{p+1}} {p+1}} + C e^{-ci^{\epsilon_0}}.
$$
Therefore, combining the {\em a priori} bound~\eqref{controle} with~\eqref{eq.limite} we get for all $u \in S^i \cap \Sigma^i$
$$\| \Phi(t,0) u\|_{L^{p+1}} \leq M \log^{\frac 1 {p+1}}( i+j) \leq M \log^{\frac 1 {p+1}} \big(i +1+ (\frac \pi 4 - t) ^{-1/\gamma}\big).
$$
Finally, define 
$$\widetilde{\Sigma} = \bigcup_{i=1}^{+\infty} \big(\Sigma^i \cap  S^i\big),$$
which is a set of full $\mu_0-$measure (since $\nu_0$ and $\mu_0$ have the same $0$ measure sets), and this concludes the proof of Proposition~\ref{croissance}. 
\end{proof}

 \subsection{Proof of Theorem \ref{thm1}}

We are now able to prove the following result which will imply Theorem~\ref{thm1}.

\begin{thm}\label{thm1.1.bis}
Assume that $p>3$. Then the solutions to~\eqref{C1} constructed above scatter almost surely when $s \longrightarrow \pm \infty$. There exist $\epsilon _0, \epsilon _1, \eta_0,\eta_1 >0$ and for $\mu_{\q}-$almost every initial data $U_0$,  there exist ${W_\pm\in \mathcal{H}^\sigma(\R)}$ such that 
   \begin{equation}\label{scat1} 
  \| \Psi(s,0) U_0- e^{is\partial_y^2} (U_0 +W_{\pm})\|_{\mathcal{H}^{\epsilon_0}(\R)}  \leq C \langle s\rangle ^{- \eta_0},  \quad s\longrightarrow \pm\infty,
   \end{equation}
and 
   \begin{equation} \label{scat2} 
  \|  e^{-is \partial^2_y}  \Psi(s,0) U_0- (   U_0+W_{\pm})   \|_{\H^{\epsilon_1}(\R)}  \leq C \langle s\rangle ^{- \eta _1}, \quad s\longrightarrow \pm\infty.
  \end{equation}
  When $p\geq 5$, we can precise the result: for all $\delta< \frac 1 2$, 
    \begin{equation} \label{scat3}
  \| \Psi(s,0) U_0- e^{is\partial_y^2} (U_0 +W_{\pm})\|_{{H}^{\delta}(\R)}  \leq C \langle s\rangle ^{- \eta_0},  \quad s\longrightarrow \pm\infty.
   \end{equation}
For all $\delta < \frac{ p+1} { 4p+2}$,     
\begin{equation}\label{scat3bis} 
  \| \Psi(s,0) U_0- e^{is\partial_y^2} (U_0 +W_{\pm})\|_{{\H}^{\delta}(\R)}  \leq C \langle s\rangle ^{- \eta_0},  \quad s\longrightarrow \pm\infty.
   \end{equation}
\end{thm}

\begin{rem}
Remark that since $e^{is\partial_y^2} $ does not act on $\H^{\epsilon_0}$, hence~\eqref{scat1} and~\eqref{scat2} are different.
\end{rem}
\begin{rem}
Recall that $U_0$ is essentially $L^2$ (actually $\mathcal{B}^{0}_{2,\infty}$, see Section~\ref{sec.besov}). Theorem~\ref{thm1.1.bis} shows that the scattering operators,
$$\mathfrak{S}^\pm: U_0 \mapsto U_0 + W_\pm,$$
which associate to the initial data $U_0$ the asymptotic profiles, are the sum of  the identity and smoothing operators, almost surely defined from $\mathcal{B}^{0}_{2,\infty}$ to $\H^{\epsilon _0}$.
\end{rem}
In the following we give the argument in the particular case $\q_0=(0,1,1,0)$, and thus $\mu_{\q_0}=\mu_0$. We refer to Section~\ref{mesq} where we explain how to treat the case of a general Gaussian measure~$\mu_{\q}$, as it is stated in Theorem~\ref{thm1.1.bis}.  We only  treat the case $s \longrightarrow  +\infty$ (the case $s \longrightarrow  -\infty$ is similar). 
The first step is 
\begin{lem}\label{prescattering} Let $3<p<5$. There exist $\epsilon_0, \eta >0$ such that for $\mu_0-$almost every initial data~$u_0$, there exists an asymptotic state $v_+ \in \H^{\epsilon _0}$ such that  the solution to~\eqref{C3} satisfies 
$$ u = \Phi(t,0) u_0= e^{-itH} u_0 + v, $$
where for all $0 \leq t <\frac{\pi}4$
\begin{equation}\label{poly3}
  \| v(t)-v_+\|_{\H^{\eps_0}(\R)} \leq C \big(\frac{\pi}4-t\big)^{\eps_0+\eta}.
\end{equation} 
\end{lem}

 \begin{proof}
In the sequel, we use the notation $u_0^f=e^{-itH} u_0$. The function~$v$ satisfies
\begin{equation*}
v(t)=-i\int_{0}^t \cos^{\frac{p-5}{2}}(2s)  e^{-i(t-s)H} F\big(u_0^f(s)+v(s)\big)  ds.
\end{equation*} 
Let $\sigma =\frac12-\frac{1}{p+1} \leq \frac13$. Let us show that there exists $\delta>0$ such that 
\begin{equation}\label{integral}
\int_{t}^{\frac{\pi}4}  \cos^{\frac{p-5}{2}}(2s)\big\| F\big(u_0^f(s)+v(s)\big) \big\|_{\H^{-\sigma}}ds  \leq C (\frac{\pi}4-t)^{\delta},
\end{equation} 
this will imply that there exists $v_+ \in \H^{-\sigma}(\R)$ such that $v \longrightarrow v_+$ in $\H^{-\sigma}(\R)$ when $t \longrightarrow \frac{\pi}4$, with the rate
\begin{equation*} 
\| v(t) -v_+\|_{\H^{-\frac13}} \leq  \| v(t) -v_+\|_{\H^{-\sigma}} \leq C (\frac{\pi}4-t)^{\delta}.
\end{equation*} 
By Sobolev, $\H^{\s}(\R) \subset L^{p+1}(\R)$ and therefore by duality $L^{\frac{p+1}p}(\R)  \subset  \H^{-\s}(\R)$. Thanks to Proposition~\ref{croissance}, we compute  for $0 \leq t <\frac{\pi}4$
\begin{eqnarray}
\int_{t}^{\frac{\pi}4}  \cos^{\frac{p-5}{2}}(2s)\big\|F\big(u_0^f(s)+v(s)\big) \big\|_{\H^{-\sigma}}ds \nonumber
&\leq& C\int_{t}^{\frac{\pi}4}  \cos^{\frac{p-5}{2}}(2s)\big\| u_0^f(s)+v(s) \big\|^p_{L^{p+1}}ds  \nonumber\\
&\leq &C\int_{t}^{\frac{\pi}4}  \cos^{\frac{p-5}{2}}(2s) \big|\log^{\frac p { p+1}} \big( \frac \pi 4 -s\big)\big|ds \nonumber \\
 &\leq&  C\int_{t}^{\frac{\pi}4}   (\frac \pi 4 -s)^{\frac{p-5}{2}}\big|\log \big( \frac \pi 4 -s\big)\big|ds \nonumber \\
&\leq &C (\frac \pi 4 -t)^{1+ \frac{p-5}{2}}\big|\log \big(  \frac \pi 4 -t\big)\big|,\label{exe}
\end{eqnarray}
where we used that $p>3$. As a consequence we get \eqref{integral}.  Let us prove that for all $\kappa>0$, there exists $\eps>0$ such that 
 \begin{equation}\label{8.9}
 \| v(t)\|_{\W^{\eps,p+1}} \leq C \big(\frac{\pi}4-t\big)^{-\kappa}.
\end{equation} 
 By Sobolev, for $\delta= \rho +\frac1{p+1}- \frac 1 2>0$  if we choose $\rho$ sufficiently close to $\frac 1 2$, we get, using~\eqref{borne5bis},
 \begin{equation}\label{premier} 
 \| v(t)\|_{\W^{\delta, p+1}} \leq C \| v(t)\|_{\H^{\rho}} \leq C\big(\frac{\pi}4-t\big)^{-K}.
\end{equation}
The estimate \eqref{8.9} then follows from an interpolation between~\eqref{premier} and Proposition~\ref{croissance}.

 Now let $\eps>0$ to be fixed later,  and compute
\begin{eqnarray*}
\frac{d}{ds} \int_{\R} |H^{\eps/2} v(s)|^2&=& 
2 \Im \int_{\R}  i\partial_s v \ov{H^{\eps} v}\\
&= &2 \cos^{\frac{p-5}{2}}(2s)\Im \int_{\R}  F\big(u_0^f(s)+v(s)\big) \ov{H^{\eps} v}\\
&= & 2 \cos^{\frac{p-5}{2}}(2s)\Im \int_{\R} H^{\eps/2} \Big( F\big(u_0^f(s)+v(s)\big) \Big)\ov{H^{\eps/2} v}\\
&\leq & C \cos^{\frac{p-5}{2}}(2s) \big\| F\big(u_0^f(s)+v(s)\big) \big\|_{\W^{\eps,(p+1)/p}}  \big\| v\big\|_{\W^{\eps,p+1}}.
\end{eqnarray*}
From Proposition~\ref{prop.C.W.1},  we get that  for any $1<q_1,q_2<+\infty$ such that $1/q_1+1/q_2=1/q$,
  \begin{equation*}
  \big\| |u|^{p-1}u \big\|_{\W^{\eps,q}} \leq C  \| u \|^{p-1}_{L^{(p-1)q_1}}  \|  u \|_{\W^{\eps,q_2}},
    \end{equation*}
hence, with the choices $q_1=\frac{p+1}{p-1}$, $q_2=p+1$, and $q=\frac{p+1}{p}$ we get 
    \begin{equation*}
\frac{d}{ds} \int_{\R} |H^{\eps/2} v(s)|^2 \leq C \big(\frac{\pi}4-s\big)^{\frac{p-5}{2}} \Big(\| u_0^f \|^{p-1}_{L^{p+1}} + \| v\|^{p-1}_{L^{p+1}} \Big)\Big( \| u_0^f\|_{\W^{\eps,p+1}}+ \| v\|_{\W^{\eps,p+1}}\Big)  \| v\|_{\W^{\eps,p+1}}.
\end{equation*}
By time integration, for $0 <t\leq \pi/4$, thanks to \eqref{8.9}, we get 
 \begin{equation} \label{bornev}
 \| v(t)\|^2_{\H^{\eps}} \leq  \| v_0\|^2_{\H^{\eps}}+C_R \int_0^t \big(\frac{\pi}4-s\big)^{\frac{p-5}{2}-3\kappa}ds \leq C_R,
\end{equation} 
provided $\kappa>0$ (and hence $\epsilon >0$)  is small enough (depending only on $p>3$).

The estimate \eqref{bornev} shows indeed that $v_+ \in \H^{\eps}(\R)$. Now, let $0 \leq \theta \leq 1$, and set $\sigma(\theta)=-\theta/3+(1-\theta)\eps$, then by interpolation
 \begin{equation*} 
 \| v(t)-v_+\|_{\H^{\s(\theta)}} \leq  \| v(t)-v_+\|^{1-\theta}_{\H^{\eps}}\| v(t)-v_+\|^{\theta}_{\H^{-1/3}}  \leq C \big(\frac{\pi}4-t\big)^{\theta\delta}.
\end{equation*} 
Next, choosing $\theta=\theta_0-\theta_1$, with $\theta_0=\frac{3\eps}{1+3\eps}$ and $\theta_1<\theta_0$. Then we set   $\epsilon _0:=\sigma(\theta)=(\frac13+\eps)\theta_1 >0$ and $\eta:=\theta \delta-\eps_0>0$ for $\theta_1>0$ small enough, which implies~\eqref{poly3}. 
 \end{proof}
 
 \begin{proof}[Proof of Theorem~\ref{thm1.1.bis}]
Now we need to come back to the NLS side. Denote by $U(s,y)$ the solution to~\eqref{C1}, and by $ u= \mathscr{L} (U)$ the solution to~\eqref{C3}. By Lemma~\ref{prescattering}, there exists $v_+ \in \H^{\eps_0}$ such that $v(t)=u(t)-e^{-itH}u_0 \longrightarrow v_+$ in $\H^{\eps_0}$ when $t \longrightarrow \pi/4$, with the rate~\eqref{poly3}.  Observe that from \eqref{conjlin}
 \begin{equation}\label{conj00} 
   \mathscr{L}_{t(s) } (e^{is \partial_y^2} v_0) = e^{-i t(s) H} v_0.
  \end{equation} 
  
Denote by~$W_+:=e^{i {\frac{\pi}4}H}v_+ \in \H^{\eps_0}$ and  let $0<\eps_1<\eps_0$. By Lemma~\ref{lemA}, \eqref{conj00} and \eqref{conjnl} we have, with   $t(s)=\frac{\arctan(2s)}2$.
  \begin{eqnarray} 
  \|    U(s)-e^{is \partial^2_y}\big(   u_0+W_+\big)   \|_{\H^{\eps_1}} &\leq & C\big(\frac{\pi}4-t(s)\big)^{-\eps_1}    \|   \mathscr{L}_{t(s)}U(s)- \mathscr{L}\big(e^{is \partial^2_y}(   u_0+W_+)   \big)   \|_{\H^{\eps_1}}, \nonumber\\
 & \leq &C\big(\frac{\pi}4-t(s)\big)^{-\eps_1}    \|     v\big(t(s)\big) -e^{-i(t(s)-\pi/4)H}   v_+   \|_{\H^{\eps_1}}.\label{scatt}
  \end{eqnarray} 
 Then by \eqref{poly3}  
  \begin{equation} \label{prev1}
 \big(\frac{\pi}4-t(s)\big)^{-\eps_1}    \|     v\big(t(s)\big) - v_+   \|_{\H^{\eps_1}}  \leq C \big(\frac \pi 4 -t(s)\big)^ {\eta_1}\sim \langle s\rangle^{- \eta_1}, \qquad s \rightarrow +\infty.
  \end{equation}
  Using the equation we have 
  $$  \big\|\big( 1-  e^{-i(   t(s)- \pi/4)H}\big)v_+  \big\|_{\H^{\epsilon_1}} = \big\| \int_{t(s)} ^{\frac \pi 4} \partial_t (e^{-i(t- \pi/4)H} v_+) dt \big\|_{\H^{\epsilon_1}} \leq \big(\frac \pi 4 - t(s) \big) \| v_+\|_{\H^{{\epsilon _1}+2}}
  $$ 
  while on the other hand,
  $$
  \big\|( 1-  e^{-i(\pi/4-   t(s))H} )v_+\big\|_{\H^{\epsilon_1}} \leq 2 \ \| v_+\|_{\H^{\epsilon _1}}.
  $$ By interpolation this implies, choosing $\epsilon _1< \frac{\epsilon_0} 3$
  \begin{equation}\label{prev2}
      \big\|\big( 1-  e^{-i(   t(s)- \pi/4)H}\big)v_+  \big\|_{\H^{\epsilon_1}}  \leq C \big(\frac \pi 4 - t(s) \big)^{\frac{\eps_0-\eps_1}2} \| v_+\|_{\H^{\epsilon_0}} \leq C \big(\frac \pi 4 - t(s) \big)^{\eps_1+\eta_2} \| v_+\|_{\H^{\epsilon_0}},
    \end{equation}
  for some $\eta_2>0$.
 Putting the estimates \eqref{prev1} and \eqref{prev2} together with \eqref{scatt}, we deduce that 
  \begin{equation*} 
  \|    U(s)-e^{is \partial^2_y}\big(   U_0+W_+\big)   \|_{\H^{\eps_1}}  \leq C \<s\>^{- \eta _0}, \qquad s\rightarrow + \infty\;,
  \end{equation*}
  which is~\eqref{scat1}. 
  
 It remains to prove~\eqref{scat2}.  According to Lemma~\ref{lemA} and \eqref{conj00}, we have 
  \begin{equation*}
  \|e^{-is\partial_y^2} v_0\| _{\H^\eta} \leq C \langle s\rangle ^ \eta \| \mathscr{L}_{-t(s)} (e^{-is\partial_y^2} v_0)\| _{\H ^\eta}\\
  = C \langle s\rangle ^ \eta \|  e^{it(s) H} v_0\| _{\H ^\eta}= C \langle s\rangle ^ \eta \|   v_0\| _{\H ^\eta}.
  \end{equation*}
   Applying this estimate to $v_0 = U(s) - e^{is\partial_y^2} ( u_0 + W_+)$ gives (using~\eqref{scat1}), with 
   $\eta = \frac12 \min(\epsilon_0, \eta_0)$,
  \begin{eqnarray*} 
 \|e^{-is\partial_y^2} U(s) -  ( u_0 + W_+)\| _{\H^\eta} &\leq& C \langle s\rangle ^ \eta \| U(s) - e^{is\partial_y^2} ( u_0 + W_+)\| _{\H ^\eta}\\
& \leq& C \langle s\rangle ^ \eta \| U(s) - e^{is\partial_y^2} ( u_0 + W_+)\| _{\H ^{\epsilon_0}}\\
& \leq  & C \langle s\rangle ^ \eta \langle s\rangle ^ {- \eta_0}\leq  C \langle s\rangle ^ {-\eta_0/2},
   \end{eqnarray*} 
 which proves  Theorem~\ref{thm1} when $3<p<5$.
 
 Let us now revisit the proof above when $p\geq 5$.
 From Proposition~\ref{prop.sigmabis},~\eqref{borne5bis}, with $v_+ = v(\frac \pi 4)$, 
 $$ \forall \rho<\frac 1 2 , \quad \lim_{t\rightarrow \frac \pi 4} \big\| u(t) - e^{-itH} u_0- v_+ \big\| _{\H^\rho}= \lim_{t\rightarrow \frac \pi 4} \big\| v(t) - v( \frac \pi 4) \big\| _{\H^\rho} =0,
 $$ 
which proves \eqref{scat3}. On the other hand, as in \eqref{exe}, and using Proposition~\ref{croissance}, we get
\begin{equation*} 
 \| v- v_+\|_{\H^{- \sigma}} \leq  \int_{t}^{\frac{\pi}4}  \cos^{\frac{p-5}{2}}(2s)\big\||u_0^f(s)+v(s)|^{p-1}\big(u_0^f(s)+v(s)\big) \big\|_{\H^{-\sigma}}ds  \leq C \big(\frac{\pi}4-t\big)^{\frac{p-3} 2}.
\end{equation*} 
Interpolating between these two estimates gives 
$$\| u(t) - e^{-itH} u_0- v_+ \| _{\H^\delta} \leq C\big(\frac{\pi}4-t\big)^{\frac{(p-3) \theta} 2}, \quad  \delta = - \sigma \theta + (1-\theta) \rho,
$$
Taking $\rho$ arbitrarily close to $\frac 1 2$ and using $\sigma = \frac 1 2 - \frac 1 {p+1} $ gives
$$  \| u(t) - e^{-itH} u_0- v_+ \| _{\H^\delta} \leq C\big(\frac{\pi}4-t\big)^{\frac{p+1} {p} ( \frac 1 2 - \delta) + \epsilon}
$$
where $\epsilon>0$  can be taken arbitrarily small. 
We deduce that for any $\delta < \frac{ p+1} {2( 2p+1)}$, there exist $ \eta>0$ and $C>0$ such that 
$$\| u(t) - e^{-itH} u_0- v_+ \| _{\H^\delta} \leq C\big(\frac{\pi}4-t\big)^{ \delta + \eta}.$$
Applying the inverse of the lens transform and using Lemma~\ref{lemA} gives~\eqref{scat3bis}.
 \end{proof}

      %%%%%%%%%%%%%%%%%%%%%%%%%%%%%%%%%%%%%%%%%%%%%%%%%%%%%%%%
%%%%%%%%%%%%%%%%%%%%%%%%%%%%%%%%%%%%%%%%%%%%%%%%%%%%%%%%%%%%
\subsection{Nonlinear evolution of measures and proof of Theorem~\ref{thm0}} 
We first consider the particular case $\q_0= (0,1,1,0)$. Estimates~\eqref{equival-ter} and~\eqref{abso-cont-ter} are just consequences of~\eqref{equival-bis} and~\eqref{abso-cont-bis}, and Lemma~\ref{lemA}.  Let us prove for instance~\eqref{abso-cont-ter} in the case $1 < p \leq 5$. The bound \eqref{abso-cont-bis} gives \medskip
\begin{equation}\label{ineqmu}
\nu_{t(s')}\bigl(\Phi(t(s'),0) A \bigr)\leq \Bigl[\nu_{t(s)} \big(\Phi(t(s),0) A\big)\Bigr] ^{\smash{\bigl( \frac{\cos(2t(s))} {\cos(2t(s'))}\bigr) ^{\frac{ 5-p} 2}}}.
\end{equation}
Then, since $t(s)=\frac{\arctan(2s)} 2$, we have 
\begin{equation} \label{eqcos}
\cos(2t(s))=\cos(\arctan(2s))=\frac{1}{\sqrt{1+4s^2}}.
\end{equation}
 Next, by \eqref{conjnl}, we have $ \mathscr{L}_{t(s)} \circ  \Psi(s,0)=\Phi(t(s),0) $, therefore if we denote by $\rho_s= \mathscr{L}^{-1}_{t(s)} \# \nu_{t(s)}$, from~\eqref{ineqmu} we obtain

\begin{equation*} 
\rho_{s'}\bigl(\Psi(s',0) A \bigr)\leq \Bigl[\rho_{s}\bigl(\Psi(s,0) A \big)\Bigr] ^{\smash{\bigl( \frac{1+ 4 (s')^2}{1+ 4 s^2}\bigr) ^{\frac{ 5-p} 4}}} .
\end{equation*}
It remains to compute the measure $\rho_s= \mathscr{L}^{-1}_{t(s)} \# \nu_{t(s)}$. Let $F: \H^{-\eps}(\R) \longrightarrow \R$ be a measurable function. We compute
 \begin{eqnarray*}
 \int_{\H^{-\eps}(\R)}F(v)d\rho_{s}(v)&=&  \int_{\H^{-\eps}(\R)}F(v)  e^{-\frac{\cos^{\frac{p-5}{2}}(2t)}{p+1}\|\mathscr{L}_{t(s)}v\|_{L^{p+1}(\R)}^{p+1}}d\big(\mu_0( \mathscr{L}_{t(s)}v)\big)   \\
 &=&  \int_{\H^{-\eps}(\R)}F(v)  e^{-\frac{\cos^{\frac{p-5}{2}+1-\frac{p+1}2}(2t(s))}{p+1}\|v\|_{L^{p+1}(\R)}^{p+1}}d\big( \mathscr{L}^{-1}_{t(s)}  \# \mu_0(  v)\big)\\
  &=&  \int_{\H^{-\eps}(\R)}F(v)  e^{-\frac{(1+4s^2)}{p+1}\|v\|_{L^{p+1}(\R)}^{p+1}}d\mu_{\q_s}(v),
\end{eqnarray*}
where we used  \eqref{estlq}, \eqref{eqcos}, and \eqref{compol}. Therefore 
$$   \rho_s =  \mathscr{L}^{-1}_{t(s)} \# \nu_{t(s)}= e^{- \frac{(1+ 4s^2)}{p+1} {\| u\|_{L^{p+1}}^{p+1}} }\mu_{\q_s},$$ 
hence the result.

The bound \eqref{log-bound} is the result of Corollary~\ref{coro-lp}, and the scattering results follow from Theorem~\ref{thm1.1.bis}. In order to treat the general case $\q_0 \in \mathcal{Q}$, we refer to Section~\ref{mesq}.
The first part of Corollary~\ref{Coro-Radon} follows directly from  Theorem~\ref{thm0}, while the second part is an application of Proposition \ref{prop-R-N}.

            %%%%%%%%%%%%%%%%%%%%%%%%%%%%%%%%%%%%%%%%%%%%%%%%%%%%%%%%%%%%%%%%%%
        %%%%%%%%%%%%%%%%%%%%%%%%%%%%%%%%%%%%%%%%%%%%%%%%%%%%%%%%%%%%%%%%%%

  \appendix

  \section{Action of the lens transform on Sobolev spaces}\label{appen.a}

  The next result shows that the  mapping  $\mathscr{L}_t^{-1}$ (defined in~\eqref{lensbis}) is not continuous in $\H^\s$ spaces.
  \begin{lem}\label{lemA}
Let $0 \leq |t| < \pi/4$ and let $u$ and $U$ be related  by 
\begin{equation*} 
u(x)= \mathscr{L}_t (U)(x)=\frac{1}{\cos^{\frac{1}{2}}(2t)}U\big( \frac{x}{\cos(2t)}\big)e^{-\frac{ix^2{\tan}(2t)}{2}}\,.
\end{equation*}
\begin{enumerate}[$(i)$]
\item There exists $C>0$ such that for any $0 \leq \sigma \leq 1$ and $0 \leq |t| < \pi/4$,
\begin{equation*}  
\| U\|_{H^\sigma(\R)} \leq C \| u \|_{\H^\sigma(\R)}, \qquad \| \<x\>^{\sigma}  u \|_{L^2(\R)} \leq C \| U \|_{\H^\sigma(\R)}.
\end{equation*}
 \item There exists $C>0$ such that for any $0 \leq \sigma \leq 1$ and $0 \leq |t| < \pi/4$,
\begin{equation}\label{major} 
\| \<y\>^{\sigma}  U \|_{L^2(\R)} \leq C(\frac{\pi}4-|t|)^{-\sigma} \| u \|_{\H^\sigma(\R)},\qquad  \| u\|_{H^\sigma(\R)} \leq C(\frac{\pi}4-|t|)^{-\sigma}  \| U \|_{\H^\sigma(\R)}.
\end{equation}

\end{enumerate}
        \end{lem}
        The dependence in $t$ of the constant in \eqref{major} is optimal, hence when $\sigma>0$ the term $ \| u \|_{\H^\sigma(\R)}$ does not control~$ \| U \|_{\H^\sigma(\R)}$, uniformly in $t \in [-\pi/4, \pi/4]$.
        
This lemma is a corrected version of \cite[Lemma 10.2]{BTT}.
       \begin{proof}
       Firstly, we write 
      \begin{equation*} 
U(y)= \cos^{\frac{1}{2}}(2t) u\big(y \cos(2t)\big)e^{\frac{iy^2\cos(2t)\sin(2t)}{2}}\,.
\end{equation*} 

$(i)$ We compute 
$$ \partial_y U(y)= \cos^{\frac{3}{2}}(2t)\Big(\partial_x u  \big(y \cos(2t)\big)+i y \sin(2t)  u \big(y \cos(2t)\big)   \Big)  e^{\frac{iy^2\cos(2t)\sin(2t)}{2}},$$
with a change of variables, we get
      \begin{eqnarray*} 
\int_{\R} |\partial_y U(y)|^2 dy &\leq & C \cos^3(2t) \int_{\R} | \partial_xu \big(y \cos(2t)\big) |^2dy+ C \cos^3(2t)\sin^2(2t) \int_{\R} y^2 | u \big(y \cos(2t) \big)|^2dy\\
 & = & C \cos^2(2t) \int_{\R} | \partial_xu ( x) |^2dx+ C \sin^2(2t) \int_{\R} x^2 | u (x) |^2dx \\
 &\leq & C  \| u \|^2_{\H^1(\R)},
\end{eqnarray*} 
which together with the relation $\|U\|_{L^2}=\|u\|_{L^2}$ yields the result for $\sigma=1$. The general result follows by interpolation.

$(ii)$ A direct computation gives for $t\longrightarrow \pi/4$
$$\int_{\R}  |y|^{2\sigma}|U(y)|^2 dy   =\cos^{-2\sigma}(2t)    \int_{\R} |x|^{2\sigma}|u(x)|^2 dx \sim  C (t-\frac{\pi}4)^{-2\sigma}   \int_{\R}  |x|^{2\sigma}|u(x)|^2 dx ,$$
and the estimate for $U$  follows (including the optimality). The proof of the estimates for $u$ is similar.
       \end{proof}

%%%%%%%%%%%%%%%%%%%%%%%%%%%%%%%%%%%%%%%%%%%%%%%%%%%%%%%%
%%%%%%%%%%%%%%%%%%%%%%%%%%%%%%%%%%%%%%%%%%%%%%%%%%%%%%%%%%%%

              \section{Some properties of the Gaussian measures \texorpdfstring{$\mu_{\q}$}{mu-q}}\label{Appendix-B}

 \subsection{A more precise description of the measures \texorpdfstring{$\mu_\q$} {mu-q}}
  In this Section, we prove a result which characterizes the $\mu_\q$ as image measures by an explicit map.
   
   \begin{lem}\label{lem-B1}
   For $\q=(s,\alpha,\beta, \theta) \in \QQ$, we define  
   $$e^{\q}_n= ( \Psi_{lin}(s,0)\circ M_{\alpha} \circ \Lambda_\beta \circ \tau_{\theta} )  e_n.$$
   Then the family $(e_n^{\q})_{n\geq 0}$ is the $L^2-$eigenbasis of a twisted harmonic oscillator 
\begin{eqnarray}\label{def-Hq}
 H_{\q}:&=&( \Psi_{lin}(s,0) \circ M_{\alpha} \circ \Lambda_\beta \circ \tau_{\theta} )  H ( \Psi_{lin}(s,0)  \circ M_{\alpha} \circ \Lambda_\beta \circ \tau_{\theta} ) ^{-1}\nonumber \\
 &=&  -\frac{1+4 \beta^4s^2}{\beta^2} \partial^2_x +4i\beta s(\beta x-\theta) \partial_x +(\beta x-\theta)^2+2i\beta^2s
  \end{eqnarray}
  associated to the eigenvalues $2n+1$ (we have $\|e^{\q}_n\|_{L^2}=|\alpha|$).
   \end{lem}  
     \begin{proof}
     First we have $\tau_{\theta}H \tau_{-\theta}=-  \partial^2_x+  (x-\theta)^2$, then 
 \begin{equation}\label{twH}
  \Lambda_{\beta} \tau_{\theta}H \tau_{-\theta} \Lambda_{\frac1\beta}=-\frac{1}{\beta^2} \partial^2_x+(\beta x-\theta)^2.
  \end{equation}
  Next, in order to show that 
 \begin{equation}\label{conj0}
 e^{is \partial_x^2}\big( \Lambda_{\beta} \tau_{\theta}H \tau_{-\theta} \Lambda_{\frac1\beta}\big) e^{-is \partial_x^2}=  -\frac{1+4 \beta^4s^2}{\beta^2} \partial^2_x +4i\beta^2 sx \partial_x +\beta^2\big(x^2+2is\big) ,
  \end{equation}
  we use the Fourier transform. Define $\mathcal{F}f(\xi)=\int_\R e^{-ix\xi}f(x)dx$. Then for $f \in \mathscr{S}(\R)$, 
   \begin{multline*}
\mathcal{F}\Big[      e^{is \partial_x^2}\big( \Lambda_{\beta} \tau_{\theta}H \tau_{-\theta} \Lambda_{\frac1\beta}\big) e^{-is \partial_x^2} f \Big](\xi)=e^{-is \xi^2}\big(  \frac{\xi^2}{\beta^2}-(i\beta \partial_\xi-\theta)^2   \big) e^{is \xi^2} \hat{f} (\xi) \\
   \begin{aligned}
&=   \Big[ \frac{1+4 \beta^4s^2}{\beta^2}\xi^2- 2i\beta^2 s -4i\beta^2 s\xi \partial_\xi   -\beta^2 \partial^2_\xi  -2i \beta \theta \partial_\xi +4\beta \theta s \xi +\theta^2   \Big]\hat{f}(\xi)\\
&=  \mathcal{F}\Big[\Big(-\frac{1+4 \beta^4s^2}{\beta^2} \partial^2_x +4i\beta^2 sx \partial_x +(\beta x-\theta)^2  -4i\beta \theta s \partial_x    +2i\beta^2s               \Big) {f}\Big](\xi),
   \end{aligned}
 \end{multline*} 
 which implies \eqref{conj0}, since the conjugaison by $M_{\alpha}$ is trivial.
 \end{proof}
 
   \begin{prop}\label{prop-B1}
For $\q=(s,\alpha,\beta, \theta) \in \QQ$, the measure $\mu_{\q}$ defined in \eqref{def-muq} is the image of the probability measure~$\p$ on $\Omega$ by the map 
 \begin{equation*} 
 \begin{array}{rcl}
\Omega&\longrightarrow&\H^{-\eps}(\R)\\[3pt]
\dis  \omega&\longmapsto &\dis\gamma^{\om}_{\q}= \sum_{n=0}^{+\infty} \frac1{\lambda_n}g_{n}(\om)e^{\q}_n.
 \end{array}
 \end{equation*}
  \end{prop}

  \begin{proof}
Let  $\q=(s,\alpha,\beta, \theta) \in \QQ$, and $F: \H^{-\eps}(\R) \longrightarrow \R$ be a measurable function. We compute

\begin{eqnarray*}
 \int_{\H^{-\eps}(\R)}F(v)d\mu_{(s,\alpha,\beta, \theta)}(v)&=&  \int_{\H^{-\eps}(\R)}F(v)d\big(\mu_0 \circ ( \Psi_{lin}(s,0)  \circ M_{\alpha} \circ \Lambda_\beta \circ \tau_{\theta}  )^{-1}\big)(v)\\
&=&  \int_{\H^{-\eps}(\R)}F\big(( \Psi_{lin}(s,0)  \circ M_{\alpha} \circ \Lambda_\beta \circ \tau_{\theta}  )(v)\big)d\mu_0 (v)\\
&=&\int_{\Omega}F\big( ( \Psi_{lin}(s,0)  \circ M_{\alpha} \circ \Lambda_\beta \circ \tau_{\theta}  )(   \gamma^{\om})\big)d\p(\om),
\end{eqnarray*}
then using
  \begin{equation*} 
(\Psi_{lin}(s,0)  \circ M_{\alpha} \circ \Lambda_\beta \circ \tau_{\theta}  )(   \gamma^{\om}) =  \gamma^{\om}_{(s,\alpha,\beta, \theta)}=\gamma^{\om}_{\q},
  \end{equation*} 
  we get the result.
\end{proof}

\begin{rem}
Denote by $\Phi_{lin}(t,t')=e^{-i(t-t')H}$ the linear flow of the harmonic oscillator. Then using that for all $\tau \in \R$, the r.v. $e^{-i\tau \lambda_n}g_n$ and $g_n$ have the same Gaussian distribution $\mathcal{N}_{\C}(0,1)$
\begin{equation}\label{inv-lin}
\Phi_{lin}(t,t') \# \mu_0=\mu_0.
\end{equation}
See also Section~\ref{sect-inv}.
\end{rem}

      \subsection{The measures \texorpdfstring{$\mu_{\q}$}{mu-q} are supported on  \texorpdfstring{$\mathcal{B}^{0}_{2,\infty}$} {the Besov space}}\label{sec.besov}
Recall that $H_{\q}$ is the twisted harmonic defined in \eqref{def-Hq}.  By Lemma~\ref{lem-B1},      $H_{\q} e_n^{\q}=\lambda^2_n e_n^{\q}$ with $\lambda_n=\sqrt{2n+1}$.     
          
 For $j\geq 0$ denote by 
     \begin{equation*} 
     I(j)=\big\{n\in \N,\;2^j\leq \lambda_{n}< 2^{j+1}\big\}.
     \end{equation*}
  We have    $\#I(j)\sim c 2^{2j}$ when $j \longrightarrow +\infty$.  Since the family $(e_n^{\q})_{n\geq 0}$ forms a Hilbert basis of $L^2(\R)$ (but with $\|e^{\q}_n\|_{L^2}=|\alpha|$), for   $\s\in \R$,  any $u\in \H^{\s}(\R)$ can be decomposed 
    \begin{equation}\label{dya}
    u=\sum_{j=0}^{+\infty} u_j, \quad \text{with} \quad    u_j=\sum_{n\in I(j)}c_{n}e_n^{\q}.
     \end{equation}

 Then, following \eqref{def-besov}, we can define the Besov space  $\mathcal B^0_{2,\infty}(\R)$, using the dyadic  decomposition~\eqref{dya},  by the norm
$$
\Vert u\Vert_{\mathcal B^0_{2,\infty}(\R)} :=  \sup_{j \geq 0}\Vert u_j\Vert_{L^2(\R)}  <+\infty.
$$
Observe that for every $\s \in \R$, the norm $\|u\|_{\H^\s(\R)}$ is equivalent to the norm 
$$ \big(\sum_{j=0}^{+\infty} 2^{2j\s}\|u_j\|^2_{L^2(\R)}\big)^{1/2}.$$
Therefore, for $\eps>0$,  we have the   embeddings
$$
  L^2(\R) \subset  {\mathcal B}^{0}_{2,\infty}(\R) \subset X^0(\R) \subset \H^{-\eps}(\R).
 $$

\begin{prop}\label{1.1}
Let $\q \in \QQ$. The measure $\mu_\q$ is supported on ${\mathcal B}^0_{2,\infty}(\R)$, namely   
$$\mu_{\q} \big( {\mathcal B}^0_{2,\infty}(\R)\big)=1.$$  Moreover, there exist $c, K_0>0$ such that for all $K\geq K_0$,
$$\mu_{\q} \big(u \in X^0(\R) :\,\|u\|_{{\mathcal B}^0_{2,\infty}}>K\big) \leq e^{-cK^2}.$$
\end{prop}

 \begin{proof}
 For $j \geq 0$, we set ${\mathcal E}_j = {\rm span}\{e^{\q}_{n},\;n \in I(j)\}$. We define the probability measure $\nu_j$ on $ {\mathcal E}_j $ via the map 
 \begin{equation*}
 \begin{array}{rcl}
\Omega&\longrightarrow&  {\mathcal E}_j \\[3pt]
\dis  \omega&\longmapsto &\dis\gamma_j^{\om}=\sum_{n \in I(j)} \frac1{\lambda_n}g_{n}(\om)e_n^{\q}.
 \end{array}
 \end{equation*}
Let $K>0$ and denote by
$$B^K_j=\big\{ u_j \in \E_j: \|u_j\|_{L^2(\R)} \leq K\big\},$$
and
$$B^K=\big\{ u \in X^0(\R) :  \|u\|_{ {\mathcal B}^0_{2,\infty}} \leq K\big\}.$$
Then $B^K =\bigotimes_{j=0}^{+\infty} B^K_j$, so that 
 \begin{equation}\label{prod1}
 \mu_\q(B^K)=\prod_{j=0}^{+\infty} \nu_j(B^K_j).
   \end{equation}
We now show that there exists $c_0>0$ such that 
 \begin{equation}\label{boule}
 \nu_j(B^K_j) \geq 1-e^{- c_02^{2j}K^2 }.
  \end{equation}
  By definition 
  $$  \nu_j\big(u_j\in \mathcal{E}_j:\,  \|u_j\|_{L^2(\R)}>K\big)   = \p\big( \om \in \Omega:\,  \|\gamma_j\|_{L^2(\R)}>K\big). $$
We have
  \begin{equation*}
   \|\gamma_j\|^2_{L^2(\R)}=\sum_{n \in I(j)} \frac{|g_n|^2}{\lambda_n^2} \leq \frac1{2^{2j}} \sum_{n \in I(j)} {|g_n|^2},
   \end{equation*}
 then by the Markov inequality (recall that the law of a complex normalised Gaussian r.v. is $\frac 1 \pi e^{- |x|^2} dx$, where $dx$ is the $2-$dimensional Lebesgue measure on $\mathbb{C}$), for all $t>0$ 
   \begin{eqnarray*}
    \p\big( \om \in \Omega:\,  \|\gamma_j\|_{L^2(\R)}>K\big) &\leq&   \p\big( \om \in \Omega:\,  \sum_{n \in I(j)} {|g_n|^2}>2^{2j}K^2\big)  \\
&\leq& e^{-t  2^{2j}K^2 }  \prod_{n\in I(j)} \mathbb{E}[  e^{t|g_n|^2 }  ]  = e^{-t  2^{2j}K^2 }  (1-t)^{-\#I(j)} \\
       &  \leq & e^{- 2^{2j}(tK^2 +c \log(1-t))}  \leq e^{- c_02^{2j}K^2 },
       \end{eqnarray*}
  with the choice $t=1/2$ and $K>0$ large enough. This   implies \eqref{boule}.
  Finally, from \eqref{prod1} we get
    \begin{equation*}
   \mu_\q\big(u \in X^0(\R):\,\|u\|_{{\mathcal B}^0_{2,\infty}}>K\big)  \leq  1- \prod_{j=0}^{+\infty} \big(1-e^{- c_02^{2j}K^2 }\big) \leq  e^{- c_1 K^2 },
     \end{equation*}
  which was the claim.   
 \end{proof}
         
        \subsection{Singular measures: proof of Proposition~\ref{prop2.2}}\label{appen.b}
 Let   $\q_0, \q_1 \in \QQ$. In the sequel, for simplicity we shall assume that $\q_0=(0,1, 1,0)$  in such a way that $H_{\q_0}=H$ (the general case is similar). Let  $\q_1=(s, \alpha, \beta, \theta)$.  Consider two sequences $(g_n)_{n\geq 0}$ and $(\ell_n)_{n\geq 0}$ of independent standard complex Gaussian random variables~$\mathcal{N}_{\C}(0,1)$ and define the random variables
 \begin{equation*}
\gamma_{\q_0}(\om,x)=\sum_{n=0}^{+\infty}\frac{g_n(\om)}{\lambda_n}e_n(x), \quad \gamma_{\q_1}(\om,x)= \sum_{n=0}^{+\infty}\frac{\ell_n(\om)}{\lambda_n}e^{\q_1}_n(x).
 \end{equation*}
  Since $(e_n)_{n\geq 0}$ and $(e^{\q_1}_n/\alpha)_{n\geq 0}$ are Hilbert bases of $L^2(\R)$, the random series 
 $$\sum_{n=0}^{+\infty} {g_n(\om)}e_n(x), \quad \frac1{\alpha}\sum_{n=0}^{+\infty} {\ell_n(\om)}e^{\q_1}_n(x),$$
  both define the same measure called the white noise measure. 
 Recall the definition~\eqref{def-Hq} of $H_{\q}$, then by application of $\alpha H_{\q_1}^{-1/2}$ we deduce that 
   $$\chi_{\q_1}(\om,x):=\alpha\sum_{n=0}^{+\infty} {g_n(\om)}H_{\q_1}^{-1/2}e_n(x) \;\; \text{and}\;\; \gamma_{\q_1}(\om,x)=  \sum_{n=0}^{+\infty} \frac{\ell_n(\om)}{\lambda_n}e^{\q_1}_n(x),$$
   both define the same measure,  which we denote      by $\mu_{\q_1}$. Define 
   $$T=\frac1{\alpha}H^{-1/2} H_{\q_1}^{1/2}.$$
    Then we have $T\chi_{\q_1}=\gamma_{\q_0}$, which in turn implies $T_{\#} \mu_{{\q_1}}=\mu_{\q_0}$. Actually, for all measurable set $A$, 
   $$T_{\#} \mu_{\q_1}(A)=  \mu_{{\q_1}} (T^{-1}(A))={\p}(\chi^{-1}_{\q_1} \circ T^{-1}(A))={\p}((T \chi_{\q_1})^{-1}(A)) =\mu_{\q_0}(A).$$
   By \cite[Theorem 6.3.2]{Boga} the measures $\mu_{\q_1}$ and $\mu_{\q_0}$ are equivalent if and only if the map 
    $${K=TT^{\star}-I : \H^1(\R) \to \H^1(\R)}$$
      is Hilbert-Schmidt (here the adjoint is taken with respect to the $\H^1$ scalar product). We compute $H_{\q_1}^{\star}=H^{-1}H_{\q_1}H$, thus $\dis T^{\star}=\frac1{\overline{\alpha}}H^{-1}H_{\q_1}^{1/2}H^{1/2}$ so that  
   $$K=TT^{\star}-I=\frac{1}{|\alpha|^2}H^{-1/2} H_{\q_1}^{1/2}H^{-1}H_{\q_1}^{1/2}H^{1/2}-I.$$

  The operator $K : \H^1(\R) \longrightarrow \H^1(\R)$ is Hilbert-Schmidt if and only if $\|K\|_{HS} <+\infty$ where 
  \begin{equation*}
  \|K\|^2_{HS}=\sum_{n=0}^{+\infty}  \big \|H^{1/2}KH^{-1/2}e_n\big\|^2_{L^2(\R)}
 \end{equation*}
(recall that this norm does not depend on the choice of the Hilbert basis $(e_n)_{n\geq 0}$ in $L^2(\R)$).  We observe that $\dis H^{1/2}KH^{-1/2}=\frac{1}{|\alpha|^2}H_{\q_1}^{1/2}H^{-1}H_{\q_1}^{1/2}-I$, then 
   \begin{eqnarray}\label{B7}
         \|K\|^2_{HS}&=&\frac{1}{|\alpha|^4} \sum_{n=0}^{+\infty}   \big\|(H_{\q_1}^{1/2}H^{-1}H_{\q_1}^{1/2}-|\alpha|^2)e_n\big\|^2_{L^2(\R)}\nonumber\\
         &=&\frac{1}{|\alpha|^4}  \sum_{n=0}^{+\infty}\<(H_{\q_1}^{1/2}H^{-1}H_{\q_1}^{1/2}-|\alpha|^2)^2e_n,e_n\>,
       \end{eqnarray}
     using that the operator $H_{\q_1}^{1/2}H^{-1}H_{\q_1}^{1/2}$ is self-adjoint on $L^2(\R)$.
We now claim that 
\begin{equation}\label{lim}
  \big((H_{\q_1}^{1/2}H^{-1}H_{\q_1}^{1/2}-|\alpha|^2)^2e_n,e_n\big)_{L^2(\R)\times L^2(\R)}     \longrightarrow_{n\rightarrow + \infty} C(\q_1) ,
\end{equation}
with $C(\q_1)>0$ for $(|\alpha|,s,\beta) \neq (1,0,1)$. This  will imply in this case that $ \|K\|_{HS} =+\infty$, and  the measures~$\mu_{\q_0}$ and $\mu_{\q_1}$ are mutually singular, reducing the study to  $\q_0 =(0,1,1,0)$ and $\q_1 = (0,1,1,\theta)$.

We set $h_n=(2n+1)^{-1}$, $z=x\sqrt{h_n}$ and $f^{h_n}(z)=h_n^{-1/4}e^{\q_1}_n(z/\sqrt{h_n})$. Then by Lemma~\ref{lem-B1}
\begin{equation}\label{sc}
 \Big[ - \frac{(1+4\beta^4s^2)}{\beta^2} h_n^2\partial^2_z  +4i\beta s( \beta z-\sqrt{h_n}\theta) h_n \partial_z +( \beta z-\sqrt{h_n}\theta)^2+2i \beta^2 sh_n\Big]f^{h_ n}(z)= f^{h_n}(z).
\end{equation}
 Similarly, we define $\widetilde{e}_n(z)=h_n^{-1/4}e_n(z/\sqrt{h_n})$, which is an $L^2-$normalised solution of 
 \begin{equation}\label{harmonicosc}
  (- h_n ^2 \partial_z^2 + z^2 -1 ) \widetilde{e}_n =0.
 \end{equation}
 Let us now recall a little of semi-classical symbolic calculus adapted to the harmonic oscillator (we refer to~\cite{Robert, Helffer} or to~\cite[Chapter~3]{Parme} for a review of this theory).
For $k \in \R$ we define the symbol class
 $$\Gamma^k =\big\{ a\in C^\infty( \mathbb{R}^{2}; \C); \; \forall j, \ell \in \mathbb{N}, \exists C>0; \;\; |\partial_x^j \partial_\xi^\ell a (x,\xi) | \leq C (1+ |x|+|\xi|)^{k- j-\ell} \big\}.
 $$
  For a symbol $a \in \Gamma^k$ and $0<h<1$, we consider the semi-classical  quantification 
 $$A^hu(z)= Op_h(a)u(z):=\frac{1}{2\pi h}\int_{\R^2} \e^{i(z-y)\xi/h}a(z,\xi)u(y)dyd\xi,$$
which, for any $\sigma \in \R$, defines  a family of  uniformly bounded operators $\mathcal{L}(\H^{k+\sigma}; \H^{\sigma})$,
 $$\forall a\in \Gamma^k, \; \exists C>0; \;\forall h>0,\quad \|Op_h(a)\|_{\mathcal{L}(\H^{k+\sigma}; \H^{\sigma})}\leq C.
 $$
 For $a \in \Gamma^{k_1}$, $b \in \Gamma^{k_2}$,  we have the symbolic calculus
 $$ Op_h(a)Op_h(b)= Op_h(ab) + h Op_h(c) +R, \qquad \exists C>0; \;\forall h>0,\;\; \|R\|_{\mathcal{L}(\H^{k+\sigma}; \H^{\sigma+2})}\leq C h^2
 $$
 with $c\in \Gamma^{k_1+k_2-1}$ given by 
 $$ c(z, \xi) = -i   \partial_{\xi} a(z, \xi) \partial_{z} b(z, \xi).
 $$
 As a consequence we get
 \begin{eqnarray}\label{sccom}
  \frac i h  [ Op_h(a), Op_h(b)] &= &Op_h(a)Op_h(b) - Op_h(b)Op_h(a) \nonumber \\
  &=&   Op_h \bigl( \partial_{\xi} a(z, \xi) \partial_{z} b(z, \xi)- \partial_{\xi} b(z, \xi) \partial_{z} a(z, \xi)\bigr)\nonumber\\
    &=& Op_h \bigl(\big\{ a, b\big\}\big),
 \end{eqnarray}
  where 
 $$ \big\{ a, b\big\}(z, \xi) =   \partial_{\xi} a(z, \xi) \partial_{z} b(z, \xi)- \partial_{\xi} b(z, \xi) \partial_{z} a(z, \xi) \in \Gamma^{k_1+k_2-1}
 $$
 is the Poisson bracket of $a$ and $b$.
Coming back to   \eqref{sc} we have $M_{\q_1,h_n}^{h_n}=Op_{h_n} (m_{\q_1,h_n})$ and 
  $$Op_{h_n} (m_{\q_1,h_n})f^{h_n}=f^{h_n},$$
  with 
  $$m_{\q_1,h}(z,\xi):=  \frac{(1+4\beta^4s^2)}{\beta^2} \xi^2  -4\beta s( \beta z-\sqrt{h}\theta) \xi +( \beta z-\sqrt{h}\theta)^2+ 2ish.$$
Similarly, we define $m(z,\xi):=   \xi^2 + z^2 $ and $M^h:=Op_{h_n} (m)=-h^2\partial^2_z+z^2$. 
   By a change of variables we have
\begin{multline} \label{equal}
 \big ((H_{\q_1}^{1/2}H^{-1}H_{\q_1}^{1/2}-|\alpha|^2)^2e_n,e_n\big)_{L^2(\R)\times L^2(\R)} =  \\
 \begin{aligned} 
 &=  \big( \big((M_{\q_1,h}^{h_n})^{1/2}(M^{h_n})^{-1}(M_{\q_1,h_n}^{h_n})^{1/2}-|\alpha|^2\big)^2 \widetilde{e}_n, \widetilde{e}_n\big)_{L^2(\R)\times L^2(\R)}  \\
& = \big( Op_h\big( (m_{\q_1,h_n}/ m-|\alpha|^2)\big)^2 \widetilde{e}_n, \widetilde{e}_n\big)_{L^2(\R)\times L^2(\R)}  + \mathcal{O}(h_n)_{n\rightarrow +\infty}
  \end{aligned} 
\end{multline}
where we used the symbolic calculus.
To understand the limit when $n\rightarrow +\infty$ of 
$$ \big( Op_h\big( (m_{\q_1,h_n}/ m-|\alpha|^2)\big)^2 \widetilde{e}_n, \widetilde{e}_n\big)_{L^2(\R)\times L^2(\R)}, $$
 let us recall (see {\it e.g.}~\cite{Bu97-1})
\begin{lem}[Semi-classical measures]\label{lem-measure}
For any bounded sequence $u_n \in L^2(\mathbb{R})$, and any sequence $(h_n)$ such that $h_n>0$ and $h_n \rightarrow 0$, there exist  subsequences $(u_{n_k}, h_{n_k})$ and a Radon measure $\mu$ on $\mathbb{R}^{2} $ such that  for any $a \in \Gamma^0$, compactly supported in $(z,\xi)$ 
\begin{equation*} 
 \lim_{k\rightarrow +\infty }\bigl( Op_{h_{n_k}} (a) u_{n_k} , u_{n_k}\bigr)_{L^2( \mathbb{R})\times L^2(\R)} = \langle \mu, a\rangle.
\end{equation*}
\end{lem}
Let us  apply this procedure to the couples $(\widetilde{e}_n, h_n)$.  Using~\eqref{harmonicosc} which implies that for $|z| > 1$, $\widetilde{e}_n$ is exponentially decaying and for $|\xi| >1$, its Fourier transform $\mathcal{F}(\widetilde{e}_n)$ is also exponentially decaying.   It is easy to check that the following properties hold true: 
\begin{itemize}
\item The convergence actually holds for any $a\in \Gamma^k$ (dropping the compact support assumption and allowing polynomial growth of $a$).
\item The measure $\mu$ has total mass $1$.
\end{itemize}
We will now prove   that the measure $\mu$ is the Liouville measure on the circle $ \{ (z, \xi); \;  z^2 + \xi^2 =1\}$. We have
%\begin{multline*}
$$0 = \Bigl( Op_{h_{k_n}} (a) \widetilde{e}_{n_k}, (- {h^2_{k_n}} \partial_z^2 + z^2 -1)\widetilde{e}_{n_k} \Bigr)_{L^2( \mathbb{R})\times L^2(\R)} 
=\Bigl( (- {h^2_{k_n}} \partial_z^2 + z^2 -1) Op_{h_{k_n}} (a) \widetilde{e}_{n_k}, \widetilde{e}_{n_k} \Bigr)_{L^2( \mathbb{R})\times L^2(\R)} .
$$
%\end{multline*}
Then,  a direct computation gives
\begin{equation*}
(- {h^2_{k_n}} \partial_z^2 + z^2 -1) Op_{h_{k_n}} (a) 
= Op_{h_{k_n}} \big(( \xi^2 + z^2 -1)  a\big)- 2ih_{k_n}    Op_{h_{k_n}} (\xi \partial_z  a)   - h^2_{k_n}  Op_{h_{k_n}} (\partial_z^2  a) ,
\end{equation*}
and therefore by~\eqref{harmonicosc}, we deduce that
\begin{equation*}
\Bigl( (- {h^2_{k_n}} \partial_z^2 + z^2 -1) Op_{h_{k_n}} (a) \widetilde{e}_{n_k}, \widetilde{e}_{n_k} \Bigr)_{L^2( \mathbb{R})\times L^2(\R)} \longrightarrow\langle \mu,  (\xi^2+ z^2 -1)a\rangle.
\end{equation*}
As a consequence, the measure $\mu$ is supported in the circle $\{ (z, \xi); \,\; z^2 + \xi^2 =1\}$. On the other hand, by \eqref{sccom} and Lemma~\ref{lem-measure} 
 \begin{multline*}
0 = \frac {i} {h_k} \Bigl( \big[   (- {h^2_{k_n}} \partial_z^2 + z^2 -1), Op_{h_{k_n}} (a)\big] \widetilde{e}_{n_k}, \widetilde{e}_{n_k} \Bigr)_{L^2( \mathbb{R})\times L^2(\R)} \\
=\Big( Op_{h_k} \big( \{ z^2 + \xi^2, a\} \big)  \widetilde{e}_{n_k}, \widetilde{e}_{n_k} \Big)_{L^2( \mathbb{R})\times L^2(\R)} + \mathcal{O}(h_k)
\longrightarrow \langle \mu, \{ z^2 + \xi^2, a\}\rangle.
\end{multline*}
Denoting by 
$$ J_m = 2\xi \partial_z - 2z \partial_\xi$$ the Hamiltonian vector field of the symbol $m(z,\xi) =z^2 + \xi^2 -1$, we get 
$$\forall a \in \Gamma^0, \quad \langle J_m \mu, a\rangle = \langle \mu, J_m(a)\rangle= \langle \mu , \{ p, a\}\rangle=0,$$
which implies that the measure satisfies $J_m\mu =0$. Summarizing, we proved
\begin{itemize}
\item The measure $\mu$ has total mass $1$.
\item 
The measure $\mu$ is supported on the circle $\{ m(z, \xi)=1\} = \big\{ (z, \xi); \; z^2 + \xi^2 =1\big\}$
\item The measure $\mu$ is invariant by  the flow of the vector field $ J_m$.
\end{itemize}
We deduce that the measure $\mu$ is the (uniform)  Liouville measure on the circle, which we denote by~$dL$.  
Remark finally that since the limit measure $\mu$ is unique (the Liouville measure), the extraction process in the construction of the measure was unnecessary (if it is possible to extract converging subsequences and if for all such converging subsequences the limit is the same, then the initial sequence was already converging to the unique possible limit measure). 

From the previous considerations, we deduce 
  \begin{multline}
 \big( Op_h\big( (m_{\q_1,h_n}/ m-|\alpha|^2)\big)^2 \widetilde{e}_n, \widetilde{e}_n\big)_{L^2(\R)\times L^2(\R)}\\
 \longrightarrow  \int_{z^2 + \xi^2=1}  \Bigg[ \frac{\frac{(1+4\beta^4s^2)}{\beta^2}\xi^2   -4\beta^2sz \xi +\beta^2z^2}{m(z,\xi)}-|\alpha|^2 \Bigg]^2dL  :=C(\q_1) >0\label{inte} 
    \end{multline}
     Notice that  $C(\q_1)>0$ for any $(s,|\alpha|,\beta) \neq (0,1,1)$, simply because we integrate on $(0,2\pi)$ a continuous non negative function which is not identically $0$. This proves \eqref{lim}.
  
  We now study the action of the translations. We still assume that $\q_0=(0,1, 1,0)$ and we set $\q_1=(0,1, 1,\theta)$. We will prove that 
  \begin{equation}\label{translation}
  \big( \big((M_{\q_1,{h_n}}^{h_n})^{1/2}(M^{h_n})^{-1}(M_{\q_1,{h_n}}^{h_n})^{1/2}-|\alpha|^2\big)^2 \widetilde{e}_n, \widetilde{e}_n\big)_{L^2(\R)\times L^2(\R)}    \sim 2{h_n} \theta^2,
  \end{equation}
which by \eqref{equal} (recall that $h_n = (2n+1)^{-1}$) will imply 
  \begin{equation}
 \big((H_{\q_1}^{1/2}H^{-1}H_{\q_1}^{1/2}-|\alpha|^2)^2e_n,e_n\big)_{L^2(\R)\times L^2(\R)}   \sim \frac{\theta^2}{n},   
\end{equation}
  By the semi-classical calculus, we have
   $$(M_{\q_1,{h_n}}^{h_n})^{1/2} =Op_{h_n}(m_{\q_1,{h_n}}^{1/2})+ \mathcal{O}({h_n}) =Op_{h_n}(m^{1/2}-\sqrt{{h_n}} \theta z m^{-1/2})+ \mathcal{O}({h_n}) $$
  thus 
  $$(M_{\q_1,{h_n}}^{h_n})^{1/2}(M^{h_n})^{-1}(M_{\q_1,{h_n}}^{h_n})^{1/2}-1= -2 \sqrt{{h_n}}\theta Op_{h_n}(z m^{-1})+ \mathcal{O}({h_n}),$$
  hence
    $$\big((M_{\q_1,{h_n}}^{h_n})^{1/2}(M^{h_n})^{-1}(M_{\q_1,{h_n}}^{h_n})^{1/2}-1\big)^2= 4 {h_n}\theta^2 Op_{h_n}(z^2 m^{-2})+ O({h_n^{3/2}}).$$
  As a consequence, when $n\longrightarrow +\infty$
  \begin{eqnarray*}
  \big( \big((M_{\q_1,{h_n}}^{h_n})^{1/2}(M^{h_n})^{-1}(M_{\q_1,{h_n}}^{h_n})^{1/2}-|\alpha|^2\big)^2 \widetilde{e}_n, \widetilde{e}_n\big)_{L^2(\R)\times L^2(\R)}  
 &\sim  & 4h_n \theta^2\int_{\{m(z,\xi)=1 \}}    \frac{z^2dL}{m^2(z,\xi)}  \\
 &\sim  &  2h_n \theta^2,
    \end{eqnarray*}
which implies~\eqref{translation}. As a consequence,  the series \eqref{B7} diverges, which implies that the measures $\q_0$ and~$\q_1$ are singular.

   \subsection{From  \texorpdfstring{$\mu_0$ to $\mu_\q$} {mu to mu-q}}\label{mesq} In this section we show how Theorem~\ref{thm2} and Theorem~\ref{thm0} for $\mu_0$ imply the same results for $\mu_\q$ the set of all parameters   $\q \in \mathcal{Q}$. The proof is in two steps. First we pass from $\q_0=(0,1, 1,0)$ ({\it i.e.} $\mu_{\q_0} = \mu_0$) to $\q_1= (0,  \alpha,  \beta, \theta)$.  This first step is harmless as the parameter~$\theta$ is just a translation in space, the parameter $\alpha$ is just an homothety, and $\beta$ is a scaling parameter. Hence this transformation amounts just to perform the following changes
 \begin{itemize}
 \item We change the harmonic oscillator  $H= - \partial_x^2 + x^2,$
for another harmonic oscillator (see \eqref{twH})
 $$H_\q = -\frac{1}{\beta^2} \partial^2_x+(\beta x-\theta)^2.$$
 \item We change the lens transform from~\eqref{lens}, \eqref{lensbis} to another lens transform
 \begin{equation*} 
 u(t,x):=  \mathscr{L}_\q (U)(t,x).
\end{equation*}
\item We change the law of our random variables~\eqref{defmu} to
\begin{equation*} 
\Omega \ni  \omega\mapsto \dis\gamma_{\q}^{\om}=\sum_{n=0}^{+\infty} \frac{1 }{\lambda_n}g_{n}(\om)e^{\q}_{n}, \qquad \mu_{\q}= {\p} \circ \gamma_{\q}^{-1},
 \end{equation*}
 where now $(e^\q_{n})_{n \geq 0}$ is the $L^2-$eigenbasis of eigenfunctions (with $\|e^{\q}_n\|_{L^2}=|\alpha|$) of our new harmonic oscillator $H_{\q}$, see Lemma~\ref{lem-B1}.
 \end{itemize}
 This first step is harmless as modulo these simple changes, the proof is the same word by word.
Once this step is achieved, it remains to study the action of the time translation $e^{is_0\partial_y^2}$ and pass from $\q_0=  (0, \alpha, \beta, \theta)$ to $\q=(s_0, \alpha, \beta, \theta)$. We are going to take benefit from the time translation invariance of~\eqref{C1} and use  that if $U$ solves ~\eqref{C1} with initial data distributed according to $\mu_{\q}= e^{is_0 \partial_y^2}_{\#}  \mu_{0,\alpha, \beta, \theta}$, then 
$\widetilde{U} (s,y) = U(s-s_0, y)$ solves also ~\eqref{C1} with  data at $s_0$ distributed according to $\mu_{\q}= e^{is_0 \partial_y^2}_{\#}  \mu_{0,\alpha, \beta, \theta}$.
However, from Step 1, we know that for all data in $\mathbf{S}_0$ (which is of full $\mu_{\q_0}-$measure), we can solve globally~\eqref{C1}, and the set $\mathbf{S}_{s_0}= \Psi(s_0,  0)( \mathbf{S}_0)$ is of full $\mu_{\q}-$measure. As a consequence, we can solve globally for all data at $s=s_0$ in the set $\mathbf{S}_{s_0}$. The estimates in  Theorem~\ref{thm2} and Theorem~\ref{thm0}  for $s_0 \neq 0$ follow from this argument and the estimates for $s_0 =0$.

\section{On the Liouville theorem}\label{Sec.E}
Let us recall that the flow of a  vector field with vanishing divergence preserves the Lebesgue measure. Though it is often tought in the context of time independent vector fields (mis)leading to believe that it only true in this context,  this assumption is unnecessary and we can allow time dependent vector fields as shown by the classical proof. See {\it e.g.} also \cite[Theorem~9.9 p. 529]{Betounes}.

Let us denote by $\phi(t,x_0) $ the solution of the ODE
$$ \dot{x} = X(t, x(t)), \qquad x(0) = x_0.$$
Let us denote by $\Psi(t)$ the map $ x_0 \mapsto \phi(t, x_0)$. It is well known that (for small time, and we shall see for all times) this map is a $C^1$ diffeomorphism and that the family of differentials $d\Psi(t)$ satisfies the ODE
$$  \frac{ d} {dt} d\Psi(t)_{\phi(t, x_0)} = dX(t, \phi(t, x_0)) d\Psi(t)_{\phi(t, x_0)},$$
and consequently using the chain trule and the fact that the differential of the determinant at $A$ is 
$$ B \longmapsto  \text{det} (A) \text{Tr} (A^{-1} B) ,$$
we get 
\begin{eqnarray*}
\frac{ d } {dt } \big(\text{det} \big(d\Psi(t)_{\phi(t, x_0)}) \big)&= &\text{det} \big(d\Psi(t)_{\phi(t, x_0)} \text{Tr} (d\Psi(t)^{-1} _{\phi(t, x_0)} \frac{ d} {dt} d\Psi(t)_{\phi(t, x_0)}) \big)\\
&=& \text{det} \big(d\Psi(t)_{\phi(t, x_0)} \text{Tr} (d\Psi(t)^{-1} _{\phi(t, x_0)} dX(t, \phi(t, x_0)) d\Psi(t)_{\phi(t, x_0)})\big)\\
&=& \text{det} \big(d\Psi(t)_{\phi(t, x_0)}  \text{Tr} ( dX(t, \phi(t, x_0))) \big)\\
&=& \text{det} \big(d\Psi(t)_{\phi(t, x_0)}  \text{div} (X) (t, \phi(t, x_0))\big)=0.
\end{eqnarray*}
Hence the Jacobian of the change of variables is constant along the integral curves. At time $t=0$ the change of variable is the identity. Hence the Jacobian is identically equal to $1$.

\section{Weighted estimates on Hermite functions}\label{app.E}
Recall that the family $(e_n)_{n \geq 0}$ denotes the Hermite functions. The purpose of this section is to prove the following result.
\begin{prop}\label{prop.F.1}
Let $\gamma< \frac 1 4$. Then there exists $C>0$ such that for any eigenfunction of the harmonic oscillator satisfies
\begin{equation}\label{borne.F.1}
\big\| \frac {e_n} { |x|^\gamma} \big\| _{L^4} \leq C\log^{\frac 1 4}( \lambda_n)  \lambda_n^{-\frac 1 4 - \gamma}.
\end{equation}
\end{prop}
In order to prove Proposition \ref{prop.F.1}, we first need some technical results. Define the function $E$ for $(x,y,\alpha)\in \R\times \R\times [0,1[$ by 
\begin{equation*}
E(x,y,\alpha)=\sum_{n\geq 0}\alpha^{n}\,e_n(x)\,e_n(y).
\end{equation*}
Then we have an explicit formula for $E$. 
\begin{lem}[\protect{\cite[Lemma A.5]{BTT}}] 
For all $(x,y,\alpha)\in \R\times \R\times [0,1[$
\begin{equation} \label{funct_E}
E(x,y,\alpha)=\frac1{\sqrt{\pi(1-\alpha^{2})}}\exp\big(-\frac{1-\alpha}{1+\alpha}\,\frac{(x+y)^{2}}{4}-\frac{1+\alpha}{1-\alpha}\,\frac{(x-y)^{2}}{4}\big).
\end{equation}
\end{lem}

We will also need the following expansions
\begin{lem}\label{lem-D3}
Let $0<\eps<1/2$, then we have 
\begin{equation*}
\begin{aligned}
 (1-x^2)^{-\frac 1 2} &= \sum_{n \geq0}  a_{n} x^{2n}, \text{ with } |a_{n}| \leq C (1+n)^{- \frac 1 2}\\
(1-x)^{-\epsilon} &= \sum_{n \geq0}  b_n x^{n}, \text{ with } |b_n| \leq C (1+n)^{\epsilon -1}\\
(1+x)^{\epsilon} &= \sum_{n \geq0}  c_n x^{n}, \text{ with } |c_n| \leq C (1+n)^{- (1+ \epsilon)}.
\end{aligned}
\end{equation*}
\end{lem}

\begin{proof}[Proof of Lemma \ref{lem-D3}]
Indeed, 
\begin{equation*}
\begin{aligned}
a_n &= \frac{1} {n!} ( \frac 1 2\times \frac 3 2 \times \cdots \times \frac {2n-1} 2)=\frac{(2n)!}{2^{2n}(n!)^2}\\
b_n &= \frac{1} {n!} \epsilon\times (\epsilon +1) \times \cdots \times (\epsilon + n -1)=\frac{\Gamma(n+\eps)}{n!\,\Gamma(\eps)}\\
c_n &= \frac{(-1)^{n-1}} {n!} \epsilon\times (1-\epsilon) \times \cdots \times ( n -1- \epsilon)=\eps  \frac{(-1)^{n-1}} {n!} \frac{\Gamma(n-\eps)}{\Gamma(1-\eps)}
\end{aligned}
\end{equation*}
and the estimates follow from the Stirling formula. 
\end{proof}

We are now able to prove Proposition \ref{prop.F.1}.
 \begin{proof}[Proof of Proposition \ref{prop.F.1}]   Denote by  
 \begin{eqnarray*}
 I(\alpha, \beta, \gamma) &:=& \int_{\R} E(x,x,\alpha) E(x,x,\beta) |x|^{-4\gamma} dx \\
 &=& \sum_{n,m \geq 0} \alpha^n \beta^m \int_\R e^2_n(x) e^2_m(x) |x|^{-4\gamma}dx.
  \end{eqnarray*}
Then using~\eqref{funct_E}, we get with $\epsilon= \frac 1 2 - 2\gamma$, 
\begin{equation*}
 I(\alpha, \beta, \gamma) =\delta_0  (1- \alpha^2)^{- \frac 1 2} (1- \beta^2)^{- \frac 1 2} \Bigl(\frac{1- \alpha\beta} {(1+ \alpha)(1+ \beta)}\Bigr)^{-\epsilon}\,,
\end{equation*}
with $\delta_0:=\pi^{-1} \int_\R e^{-2y^2} |y|^{-4 \gamma} dy $. We deduce
$$(1-x^2)^{-\frac 1 2} (1+x)^\epsilon= \sum_{n \geq0} d_n x^n, \qquad d_n = \sum_{p+q=n} a_pc_q,
$$ with 
$$|d_n| \leq \sum_{1\leq p\leq \frac n 2} a_p c_{n-p} + \sum_{\frac n 2 < p \leq n}a_p c_{n-p} \leq C n^{- \frac 1 2}
$$
which implies 
$$ \int_\R e^2_n(x) e^2_m(x) |x|^{-4\gamma}dx = \delta_0\sum_{\genfrac{}{}{0pt}{1}{p+q=n}{r+q=m} }d_pb_qd_r.$$
Therefore, 
$$ \int_\R e^4_n(x)   |x|^{-4\gamma}dx = \delta_0\sum_{\genfrac{}{}{0pt}{1}{p+q=n}{0\leq p\leq \frac{n}2} } d^2_pb_q+\delta_0 \sum_{\genfrac{}{}{0pt}{1}{p+q=n}{\frac{n}2 < p\leq n} } d^2_pb_q\leq C \log(n) n^{\epsilon -1},
$$
which (recall that $\lambda_n^2= (2n+1)$) is Proposition~\ref{prop.F.1}.
\end{proof}

\section{Proof of the embedding \texorpdfstring{$ \mathcal{B}^{\rho} _{p,q} \subset B^{\rho} _{p,q}$} {between the besov spaces}}

Recall the definitions of Section~\ref{sect32}. In this section we prove
 
 \begin{lem}\label{lem-Besov}
Then for any $1\leq p, q \leq + \infty$, and any $\rho\geq 0$, 
$$ \mathcal{B}^{\rho} _{p,q} \subset B^{\rho} _{p,q},$$
with continuous injection.
 \end{lem}
 
 \begin{proof}
Recall that the Fourier multiplier $D_x$ is defined by the formula $\mathcal{F}(D_x  u)(\xi)= |\xi| \mathcal{F}(u)(\xi)$ for $f \in \mathscr{S}'(\R)$. Let $\chi_j \in C^\infty_0 ( \mathbb{R})$ be as in \eqref{partition} and set      $\Delta_j= \chi (2^{-j} D_x) $ and $ \widetilde{\Delta}_k=\chi(2^{-k} \sqrt{H})$. We write
 $$ \Delta_j f = \sum_{k=0}^{+\infty}  \Delta_j  \widetilde{\Delta}_k f.$$
 In this sum, we distinguish two contributions from $\{ k \leq j-3\}$ and $\{ k >j-2\}$ respectively.
 To study the first contribution, let us just recall that for any $\chi \in C^\infty_0 ( \mathbb{R})$, the operator $\chi(h \sqrt{H})$ is an $h-$pseudo-differential operator with symbol in $\Gamma^0$ (recall the notations of Section~\ref{appen.b}), 
 and its (formal) symbol $a$  is supported in $\{ (x, \xi) ; \; \sqrt{ x^2+ \xi^2} \in \text{ supp }  \chi\}$. As a consequence, if  $\chi \in C^\infty_0 ( \frac 1 2 , 2)$, then  $ \chi (2^{-j} D_x) \chi(2^{-k} \sqrt{H})$ is, for $k \leq j-3$, a pseudo-differential operator with vanishing (formal) symbol,
 hence gaining any number of derivatives and any power of $|x|$. We deduce
 $$ \forall N>0, \; \exists C>0; \; \forall j,\; \forall k \leq j-3, \;\;\; \big\| \chi (2^{-j} D_x) \chi(2^{-k} \sqrt{H})\big\|_{\mathcal{L} ( L^p)} \leq C 2^{-jN}.$$
 As a consequence, 
 $$ \big\| \sum_{k=0}^{j-3}  \Delta_j  \widetilde{\Delta}_k f\big\|_{L^p} \leq C_N 2^{-jN} \| f\|_{L^p} .$$
 To study the contribution of the second term, we just use that $\chi(2^{-j} D_x)$ is bounded on $L^p$ with uniform bound with respect to $j$, and 
 \begin{equation}\label{som-par}
 \| 2^{j\rho} \Delta_j \sum_{k\geq j-2} \widetilde{\Delta}_k u\|_{L^p}  \leq C \sum_{k\geq j-2} 2^{- (k-j) \rho} \| 2^{k\rho}\widetilde{\Delta}_k \|_{L^p}.
 \end{equation}
 According to Schur lemma, the convolution by $2^{\ell \rho} 1_{\ell\leq -2}$ is bounded on $\ell ^1$, and $\ell^\infty$, and hence by interpolation on $\ell^q$. This implies
 $$\big\| \| 2^{j\rho} \Delta_j \sum_{k\geq j-2} \widetilde{\Delta}_k u\|_{L^p}\big\| _{\ell ^q_j} \leq C \big\| \| 2^{k\rho}\widetilde{\Delta}_k \|_{L^p}\big\| _{\ell^q_k}\;,$$
 which together with \eqref{som-par} enables to complete the proof.
 \end{proof}

%%%%%%%%%%%%%%%%%%%%%%%%%%%%%%%%%%%%%%%%%%%%%%%%%%%%%%%%%%%%%%%
%%%%%%%%%%%%%%%%%%%%%%%%%%%%%%%%%%%%%%%%%%%%%%%%%%%%%%%%%%%%%%%

\end{document}